\documentclass[12pt,a4paper,reqno]{amsart}

\usepackage{amsthm, amsmath, amssymb,amscd, amsfonts, mathtools,mathrsfs}
\usepackage[margin=2.5cm, headsep=0.75cm, footskip=0.75cm]{geometry}
\usepackage[T1]{fontenc}
\usepackage[utf8]{inputenc}
\usepackage[english]{babel}
\usepackage{enumitem}
\setenumerate{listparindent=\parindent}
\setlist[enumerate]{font=\normalfont}
\usepackage[draft=false]{hyperref}
\usepackage[noabbrev, capitalise, nameinlink]{cleveref}                                      
\usepackage{bookmark}
\hypersetup{
  pdfauthor={Kevin Aguyar Brix, Toke Meier Carlsen, and Aidan Sims},
  pdftitle={Ideal structure of \texorpdfstring{$C^*$}{C*}-algebras of commuting local homeomorphisms},
  colorlinks=true,
  linkcolor={blue},
  citecolor={red!50!black},
  linktocpage=true                                                      
}

\usepackage[dvipsnames]{xcolor}
\usepackage{tikz}                                                
\usetikzlibrary{cd, matrix, arrows, decorations.markings}

\usepackage{verbatim}       

\newcommand{\N}{\mathbb{N}}                                         
\newcommand{\Z}{\mathbb{Z}}                                         
\newcommand{\Q}{\mathcal{Q}}

\newcommand{\Dd}{\mathcal{D}}
\newcommand{\RR}{\mathbb{R}}
\newcommand{\T}{\mathbb{T}}                                         

\newcommand{\coc}{c}        

\newcommand{\Hh}{\mathcal{H}}

\newcommand{\I}{\mathcal{I}}
\newcommand{\J}{\mathcal{J}}
\newcommand{\Bb}{\mathcal{B}}

\DeclareMathOperator{\Stab}{Stab}                                   

\DeclareMathOperator{\Aut}{Aut}
\DeclareMathOperator{\Span}{span}

\DeclareMathOperator{\supp}{supp}

\DeclareMathOperator{\Prim}{Prim}

\DeclareMathOperator{\Per}{Per}
\DeclareMathOperator{\Ad}{Ad}

\newcommand{\Iess}{\mathcal{I}^{\operatorname{ess}}}
\newcommand{\ess}{\operatorname{ess}}

\let\tilde\widetilde{}
\let\hat\widehat{}
\let\subset\subseteq{}
\let\supset\supseteq{}

\newtheorem{lemma}{Lemma}
\numberwithin{lemma}{section} %

\newtheorem{corollary}[lemma]{Corollary}
\newtheorem{theorem}[lemma]{Theorem}
\newtheorem*{theorem*}{Theorem}
\newtheorem{proposition}[lemma]{Proposition}

\theoremstyle{definition}
\newtheorem{definition}[lemma]{Definition}
\newtheorem{example}[lemma]{Example}
\newtheorem{notation}[lemma]{Notation}
\newtheorem{remark}[lemma]{Remark}

\numberwithin{equation}{section}                                     

\makeatletter
\renewcommand{\tocsection}[3]{
  \indentlabel{\@ifnotempty{#2}{\ignorespaces#1 #2.\quad}}#3\dotfill%
}
\renewcommand{\tocsubsection}[3]{%
  \indentlabel{\@ifnotempty{#2}{\ignorespaces#1 #2.\quad}}#3\dotfill%
}
\makeatother

\makeatletter
\let\origsection\section
\renewcommand\section{\@ifstar{\starsection}{\nostarsection}}
\newcommand\sectionspace{\vspace{0.5ex}}
\newcommand\nostarsection[1]{\sectionspace\origsection{#1}\sectionspace}
\newcommand\starsection[1]{\sectionspace\origsection*{#1}\sectionspace}
\makeatother

\makeatletter\g@addto@macro\bfseries{\boldmath}\makeatother

\makeatletter
\@namedef{subjclassname@2020}{%
  \textup{2020} Mathematics Subject Classification}
\makeatother

\date{\today}
\thanks{We thank Johannes Christensen and Sergiy Neshveyev for pointing out an error in the original
proof of Proposition~\ref{prp:fP}. Aidan thanks the Winton Public Library and the Musical Fence
Caf\'e (especially Tom the barista) in Winton for good coffee, reliable wifi and pleasant places to
think about maths in outback Queensland. This research was supported by ARC Discovery Project
DP200100155. Kevin was supported by the Carlsberg Foundation through an Internationalisation Fellowship
and the Independent Research Fund Denmark (1025-00004B).}

\begin{document}

\title{Ideal structure of $C^*$-algebras of commuting local homeomorphisms}

\author[K.A.~Brix]{Kevin Aguyar Brix}
\address[K.A.~Brix]{School of Mathematics and Statistics, University of Glasgow, Glasgow G12 8QQ, United Kingdom}
\email{kabrix.math@fastmail.com}
\author[T.M.~Carlsen]{Toke Meier Carlsen}
\address[T.M.~Carlsen]{K\o{}ge, Denmark}
\email{toke.carlsen@gmail.com}
\author[A.~Sims]{Aidan Sims}
\address[A.~Sims]{School of Mathematics and Applied Statistics, University of Wollongong, Wollongong, NSW 2522, Australia}
\email{asims@uow.edu.au}


\begin{abstract}
We determine the primitive ideal space and hence the ideal lattice of a large class of
separable groupoid $C^*$-algebras that includes all $2$-graph $C^*$-algebras. A key
ingredient is the notion of harmonious families of bisections in \'etale groupoids
associated to finite families of commuting local homeomorphisms. Our results unify and
recover all known results on ideal structure for crossed products of commutative
$C^*$-algebras by free abelian groups, for graph $C^*$-algebras, and for Katsura's
topological graph $C^*$-algebras.
\end{abstract}

\maketitle

\setcounter{tocdepth}{1}
\tableofcontents

\section{Introduction}
\subsection*{Background}
The lattice of ideals of a $C^*$-algebra is a fundamental structural feature that is
notoriously difficult to compute. In the case of commutative $C^*$-algebras, type~I
$C^*$-algebras (e.g. continuous trace $C^*$-algebras), and other continuous fields of
simple $C^*$-algebras, or just-infinite $C^*$-algebras \cite{Grigorchuk-Musat-Rordam} the
primitive ideal space is a key piece of data for classifying the $C^*$-algebras
\cite{Gelfand-Naimark, Dixmier-Douady, Grigorchuk-Musat-Rordam}. However, in most cases
where $C^*$-algebras are built from dynamical or combinatorial data such as shifts of
finite type \cite{Cuntz-Krieger1980}, local homeomorphisms \cite{Renault, Deaconu,
Anantharaman-Delaroche, Exel-Renault}, or directed graphs (and their higher-rank
analogues) \cite{RaeburnCBMS}, work has focussed on conditions that ensure simplicity
\cite{KumPasRae, Kumjian-Pask2000}, or that reduce the complexity of the ideal structure
of the $C^*$-algebra of a dynamical system \cite{Renault:ideals, BatHonRaeSzy, WvW};
perhaps due to the Elliott classification program (see, for example, \cite{Elliott,
Elliott2, Phillips:DM00,Kirchberg, TWW, Winter}) whereby simple $C^*$-algebras can be classified by
$K$-theory and traces.

Cuntz \cite{Cuntz1981} determined the ideal structure of non-simple Cuntz--Krieger
algebras assuming condition (II) in terms of the irreducible components of the underlying
shift of finite type. Cuntz' condition (II) ensures that all ideals are gauge-invariant
(such ideals are called \emph{dynamical} in our companion paper
\cite{BriCarSim:sandwich}). The later groundbreaking work of an Huef and Raeburn
\cite{anHuef-Raeburn1997} further developed this technique and classified all
gauge-invariant ideals of non-simple Cuntz--Krieger algebras. This was shortly followed
by complete results for graph algebras \cite{Hong-Szymanski}, and very recently for
topological graphs (via actions of $\N$ by local homeomorphisms) \cite{Katsura2021}. The
key idea of an Huef and Raeburn underpins both analyses: the primitive ideals are indexed
by a quotient of $X\times \T$, where $X$ is the space of infinite paths; and the
hull-kernel topology is computed using a \emph{sandwiching lemma}: each primitive ideal
is sandwiched between a pair of gauge-invariant ideals for which the subquotient is
Morita equivalent to a crossed product of the form $C_0(U) \rtimes \Z$.

As a result of a significant program dating back to work of Mackey
\cite{Mackey}, Rieffel \cite{Rieffel}, and Green \cite{Green}, the
primitive-ideal spaces of such crossed products are well understood (see, for
example, \cite{Williams2007}): if $X$ is a second-countable locally compact
Hausdorff space and $G$ is a second-countable locally compact abelian group
acting on $X$, then there is an equivalence relation on $X \times
\widehat{G}$ such that $(x,\chi) \sim (y,\rho)$ if $x$ and $y$ have the same
orbit closure, and $\chi\overline{\rho}$ annihilates the stabiliser of $x$.
Each $(x,\chi)$ determines an irreducible representation $\pi_{x, \chi}$ of
$C_0(X) \rtimes G$ on $L^2(G \cdot x)$, and the kernels of $\pi_{x,\chi}$ and
$\pi_{y,\rho}$ coincide precisely if $(x, \chi) \sim (y, \rho)$. So $\pi
\colon (x,\chi) \to \ker(\pi_{x,\chi})$ descends to a bijection from $(X
\times \widehat{G})/{\sim}$ onto $\Prim(C_0(X) \rtimes G)$, and it transpires
that this map is in fact a homeomorphism. Few other general results on ideal
structure of $C^*$-algebras of dynamical systems like groupoids are available
in the literature, beyond those such as B\"onicke and Li's results
\cite{Bonicke-Li} on strongly effective \'etale groupoids for which every
ideal is dynamical, and those of van Wyk and Williams \cite{WvW} in which
continuity conditions are imposed on the isotropy groups. Katsura's results
\cite{Katsura2021} are the furthest reaching results that do not impose
regularity conditions on isotropy groups.

\subsection*{Our results}
In this paper, we make substantial new progress on the problem of ideal structure in
separable $C^*$-algebras of \'etale groupoids. Our main result \cref{cor:base for
topology} (see also~\cref{thm:neighbourhood-basis}) describes a base for the topology of
the primitive ideal space, for a large class of separable $C^*$-algebras. The formal
statement of the result is complicated, but we show by example that in many cases of
interest our description is genuinely computable (see, for example, \cref{sec:hrgs}). In
particular, our results are the first of their kind to cover large classes of higher-rank
graph $C^*$-algebras including all row-finite $2$-graphs with no sources.

The specific class of $C^*$-algebras that we study are those arising from actions
$T\colon \N^k\curvearrowright X$ by local homeomorphisms of second-countable locally
compact Hausdorff spaces $X$. All the cases mentioned above (actions by free abelian
groups, shifts of finite type, directed graphs, and Katsura's topological graphs) provide
examples of such actions. The associated topological groupoid $G_T$ (sometimes referred
to as the \emph{Deaconu--Renault groupoid}) is well behaved in the sense that it is
second-countable locally compact Hausdorff, amenable, and \'etale. Conceptually, $G_T$
can be regarded as a proxy for the orbit space of $T$.

For such actions, Sims and Williams \cite[Theorem 3.2]{Sims-Williams} discover a surjective map
\[
  \pi \colon X \times \T^k \to \Prim(C^*(G_T))
\]
and an equivalence relation on $X\times \T^k$ (akin to the equivalence relation in the
case of crossed products mentioned above) such that $\pi(x, z) = \pi(x', z')$ precisely
when $(x,z)$ and $(x',z')$ are equivalent. Importantly, the map $\pi$ is a
parameterisation and so does not \emph{a priori} say anything about the hull-kernel
topology on the primitive ideal space. The sandwiching lemma for general \'etale
groupoids in our companion paper \cite[Lemma~3.6]{BriCarSim:sandwich} applies to this
setting, so every ideal is optimally sandwiched between dynamical ideals. However, the
resulting subquotients do not admit natural descriptions as crossed products, and the
approach of \cite{anHuef-Raeburn1997, Hong-Szymanski, Katsura2021} does not naturally
extend to this setting; moreover, the isotropy-group bundle is typically badly
discontinuous, so the results of \cite{WvW} do not apply either. Without access to the
powerful Mackey--Green--Rieffel machine, that has underpinned previous analyses, we
needed a new approach.

We describe the ideals in $C^*(G_T)$ by directly analysing which subsets of $X\times
\T^k$ are preimages under $\pi$ of open sets in $\Prim(C^*(G_T))$. To do this, we
introduce the \emph{essential isotropy} (\cref{def:essential-isotropy}) of $G_T$ that
will play a key role; in particular, it determines certain isotropy subgroups $\J_x$ over
a unit $x$. We first describe a necessary condition for a subset $A$ of $X \times \T^k$
to be the preimage of an open set in terms of a family $\Bb = (B_\alpha)_{\alpha \in
\J_x}$ of homogeneous open bisections of $G_T$ indexed by the isotropy subgroups $\J_x$
where $x$ is a unit in the the projection of $A$ onto $X$. Given any such family
$(B_\alpha)_{\alpha \in \J_x}$ of bisections, we write $\Bb^{\ess}$ for its intersection
with the essential isotropy. We identify a system of \emph{$\Bb$-saturated} subsets $(U
\times V) \cdot (\Bb^{\ess})^\perp$ of $X \times \T^k$, indexed by pairs $(U, V)$
consisting of an open neighbourood $U$ of $x$ and an open subset $V$ of $\T^k$, with the
property that if $A \subseteq X \times \T^k$ is the preimage of an open set of primitive
ideals and contains a point $(x,z)\in X\times \T^k$, then there is a pair $(U, V)$ such
that $(x,z) \in U \times V$, and $(U \times V) \cdot (\Bb^{\ess})^\perp$ is contained in
$A$ (\cref{thm:Aopen}). This, our first main result, applies to any $G_T$ and gives
useful information about ideal structure; in particular, it follows that $\pi$ is
continuous (\cref{cor:pi-continuous}).
\begin{theorem*}
  Let $X$ be a locally compact Hausdorff space and suppose that $T\colon \N^2\curvearrowright X$ is an action by local homeomorphisms.
  The surjection $\pi\colon X \times \T^k \to \Prim(C^*(G_T))$ described above (cf.~\cite{Sims-Williams}) is continuous.
\end{theorem*}

Next, we aim for a complete description of the ideal structure. The key idea is the
notion of \emph{harmonious families of bisections} (\cref{def:Bf}). These are families of
bisections $\Bb$ as above, whose intersections with the essential isotropy of $G_T$
satisfy additional consistency conditions. These conditions allow us to employ harmonic
analysis on $\T^k = \widehat{\Z}^k$ to prove a kind of noncommutative Urysohn lemma
(\cref{prp:fP}): given a harmonious family of bisections $\Bb = (B_\alpha)_{\alpha \in
\J_x}$, a point $z \in \T^k$ and open neighbourhoods $U$ of $x$ and $V$ of $z$, we
construct an element of $C^*(G_T)$ that does \emph{not} belong to the ideal $\pi(x, z)$
but does belong to $\pi(y,w)$ for all $(y,w)$ in the complement of $(U \times V) \cdot
(\Bb^{\ess})^\perp$. We believe this result is of independent interest, but here it is
the engine room of our proof that the sets $(U \times V) \cdot (\Bb^{\ess})^\perp$ are
preimages of open sets (\cref{thm:neighbourhood-basis}). Our main result (\cref{cor:base
for topology}) uses this to describe a base for the topology on $\Prim(C^*(G_T))$:
\begin{theorem*}
  Let $X$ be a second-countable locally compact Hausdorff space and suppose that $T\colon \N^k\curvearrowright X$ is an action by local homeomorphisms
  that admits harmonious families of bisections.
  A base for the hull-kernel topology of $\Prim(C^*(G_T))$ is given by sets of the form
  \[
    \pi((U \times V) \cdot (\Bb^{\ess})^\perp)
  \]
  where $\Bb$ is a harmonious family of bisections at a unit $x$, $U$ is an open neighbourhood of $x$, and $V$ is open in $\T^k$.
\end{theorem*}
In particular, we recover known results about effective groupoids: if $G_T$ is
minimal and effective then the primitive ideal space is a singleton (so $C^*(G_T)$ is
simple); and if $G_T$ is strongly effective, then the primitive ideal space is
homeomorphic to the quasi-orbit space of $G_T$.

Using our main theorem, we provide an explicit description of the lattice of ideals of
$C^*(G_T)$ in terms of subsets of $X \times \T^k$ (\cref{prp:ideal generators}), and a
characterisation of convergence of sequences in $\Prim(C^*(G_T))$ (\cref{thm:necsuff for
convergence}).

At present, we do not know whether every action by commuting local homeomorphisms admits
harmonious families of bisections, but we confirm this for actions by free abelian
groups, (topological) graphs, and more importantly a large class of higher-rank graphs
including all row-finite $2$-graphs with no sources. These are the first general results
for irreversible dynamical systems of rank greater than~1.

We conclude the paper with a detailed analysis of the case of higher-rank graph
$C^*$-algebras \cite{Kumjian-Pask2000}, which were a significant motivator and source of
examples for our work. Given a row-finite higher-rank graph $\Lambda$ with no sources and
with infinite-path space $\Lambda^\infty$, the associated groupoid $G_\Lambda$ is the
Deaconu--Renault groupoid for the action $T\colon \N^k \curvearrowright \Lambda^\infty$
by shift maps. Our main theorem identifies the collection $\mathcal{A}_\Lambda$ of
subsets of $\Lambda^\infty \times \T^k$ that are the preimages of open subsets of
$\Prim(C^*(\Lambda))$ with an appropriate collection $\mathcal{D}_\Lambda$ of subsets of
$\Lambda^0 \times \T^k$ (\cref{cor:DLambda description}), so that such subsets index the
ideals of $C^*(\Lambda)$. This result is in the spirit of
\cite[Theorem~5.1]{Carlsen-Sims}. We finish by working through two concrete examples,
completely determining the ideal structure of two $2$-graph $C^*$-algebras that are not
accessible to any pre-existing computations of ideal structure.

The ideas and techniques we develop here are flexible, and we suspect they can be applied
to significantly larger classes of groupoid $C^*$-algebras, particularly when combined
with the sandwiching lemma of \cite{BriCarSim:sandwich} as we do for the case of single
local homeomorphisms in \cref{sec:Katsura stuff}. Although we restrict our attention here
to Deaconu--Renault groupoids, the notions of essential isotropy and of harmonious
families of bisections make sense, and are potentially useful, for arbitrary \'etale
groupoids.

\subsection*{Outline}
We start in \cref{sec:background} by outlining necessary background and notation for
topological groupoids with examples from dynamics and graphs, the ideal structure of
separable $C^*$-algebras, the results of \cite{Sims-Williams} on a parametrisation of
primitive ideals in Deaconu--Renault-groupoid $C^*$-algebras, and harmonic analysis on
$\T^k$. In \cref{sec:nice reps} we describe and analyse a family of representations of
Deaconu--Renault groupoids that interpolate between the regular representations and the
orbit-space representations. These are a key ingredient in the proof of our main theorem.
In \cref{sec:necessary condition}, we establish our first main result (\cref{thm:Aopen}):
a necessary condition for a subset of $X \times \T^k$ to be the preimage of an open
subset of $\Prim(C^*(G_T))$. It follows that $\pi \colon X \times \T^k \to
\Prim(C^*(G_T))$ is continuous (\cref{cor:pi-continuous}). In \cref{sec:sandwiching vs
SW}, we clarify how the nested open invariant sets of our general sandwiching lemma
\cite{BriCarSim:sandwich} relate to subsets of $X \times \T^k$, and we show that
Deaconu--Renault groupoids admit obstruction ideals in the sense of \cite{AraLolk,
BriCarSim:sandwich}. In \cref{sec:bisection-families} we introduce and study harmonious
families of bisections, which are the main new tool we apply to study ideal structure. We
identify two sufficient conditions to generate a harmonious family of bisections: one of
these (\cref{lem:relatively commuting}) is particularly useful in ample groupoids such as
those of (higher-rank) graphs; the other (\cref{lem:commuting-bisections}) is applicable
to a wider class of groupoids but requires more stringent hypotheses. In
\cref{sec:primitive-ideal-space}, we prove our main result
(\cref{thm:neighbourhood-basis}) using harmonious families of bisections: we determine
the preimages in $X \times \T^k$ of open subsets of $\Prim(C^*(G_T))$ for
Deaconu--Renault groupoids that admit harmonious families of bisections. We use this to
describe the ideal lattice of the $C^*$-algebras of such groupoids in \cref{sec:lattice}.
We characterise convergence of sequences in $\Prim(C^*(G_T))$ in \cref{sec:convergence}.
\cref{sec:examples} details a number of examples including all actions by commuting
homeomorphisms, all graph $C^*$-algebras, and actions by a local homeomorphism on a
second-countable locally compact Hausdorff space. Moreover, we show that many higher-rank
graphs (including all $2$-graphs) admit harmonious families of bisections. Finally, in
\cref{sec:hrgs} we use our main theorem to completely describe the ideal structure of the
$C^*$-algebras of higher-rank graphs whose groupoids admit harmonious families of
bisections.

\section{Background material}\label{sec:background}

Let $\Z$ and $\N$ denote the integers and the nonnegative integers, respectively.

\subsection{Isotropy in \texorpdfstring{\'etale}{etale} groupoids}\label{sec:isotropy}

We use the notation and conventions for groupoids of \cite{SimsCRM}.

A \emph{groupoid} is a small category in which every morphism is invertible. A groupoid
$G$ is a Hausdorff \emph{\'etale} groupoid if it carries a locally compact Hausdorff
topology with respect to which the range and source maps $r$ and $s$ are local
homeomorphisms onto the unit space $G^{(0)} = \{r(\alpha) : \alpha \in G\}$, the
inversion map $\alpha \mapsto \alpha^{-1}$ is continuous, and composition is continuous
in the subspace topology on the space $G^{(2)}$ of composable pairs. If $G$ is \'etale
then the open subsets of $G$ on which both $r$ and $s$ restrict to homeomorphisms form a
basis for the topology; we call such sets \emph{open bisections}. The unit space
$G^{(0)}$ in a Hausdorff \'etale groupoid $G$ is both closed and open in $G$.

The \emph{orbit} $[x]$ of a unit $x \in G^{(0)}$ is the set $\{r(\gamma) : \gamma\in G, s(\gamma) = x\}
\subseteq G^{(0)}$, and the \emph{orbit closure} $\overline{[x]}$ is the closure of $[x]$ in
$G^{(0)}$. The restriction $G|_{\overline{[x]}} = r^{-1}\big(\overline{[x]}\big) \cap
s^{-1}\big(\overline{[x]}\big)$ is itself a Hausdorff \'etale groupoid when $G$ is.
The groupoid is \emph{minimal} if every orbit is dense.

Let $G$ be an \'etale groupoid and take $x\in G^{(0)}$. The \emph{range fibre} of $x$ is
$G^x = r^{-1}(x)$ and similarly the source fibre is $G_x = s^{-1}(x)$; they are both
discrete subsets of $G$. An element $\gamma\in G$ is \emph{isotropy} at $x$ if $r(\gamma)
= x = s(\gamma)$, and the \emph{isotropy group} at $x\in G^{(0)}$ is the discrete group
$\I(G)_x \coloneqq \{ \gamma\in G : r(\gamma) = x = s(\gamma)\} = G^x\cap G_x$. The
\emph{isotropy subgroupoid} $\I(G)$ is the group bundle $\bigsqcup_{x\in G^{(0)}}
\I_x(G)$. This is an algebraic subgroupoid of $G$. The topological interior of the
isotropy is an open subgroupoid of $G$ denoted $\I^\circ(G)$. We write $\I^\circ(G)_x =
\I^\circ(G) \cap \I(G)_x$.
The groupoid is \emph{effective} if $\I^\circ(G) = G^{(0)}$.

The notion of essential isotropy will be important for our main results.

\begin{definition}\label{def:essential-isotropy}
Let $G$ be an \'etale groupoid. The \emph{essential isotropy at $x\in G^{(0)}$} is
$\Iess_x(G) \coloneqq \I^\circ(G_{\overline{[x]}})_x \subset G_{\overline{[x]}}$, and the
\emph{essential isotropy of $G$} is the bundle of discrete groups
\[
  \Iess(G) \coloneqq \bigsqcup_{x\in {G^{(0)}}} \Iess_x(G);
\]
that is, $\Iess(G)$ is the collection of all points that are interior to the isotropy in
the restriction of $G$ to the orbit-closure of their source.
\end{definition}

By a \emph{normal subgroupoid} of the isotropy of a groupoid $G$ we mean a subset $H
\subseteq \I(G)$ that is closed under inversion and composition and has the property that
if $\alpha \in H$ and $\beta \in G_{r(\alpha)}$ then $\beta\alpha\beta^{-1} \in H$.

\begin{lemma}\label{lem:Icirc-Iess normal}
Let $G$ be an \'etale groupoid. Then $\I^\circ(G)$ and $\Iess(G)$ are both normal
subgroupoids of $G$.
\end{lemma}
\begin{proof}
If $(\alpha,\beta) \in \I^\circ(G) \cap G^{(2)}$, then there are open bisections $A \owns
\alpha$ and $B \owns \beta$ consisting of isotropy, and then $AB$ is an open subset of
the isotropy containing $\alpha\beta$, and $A^{-1}$ is an open subset of the isotropy
containing $\alpha^{-1}$. So $\I^\circ(G)$ is closed under composition and inversion.

To see that it is normal, suppose $\gamma \in \I^\circ(G)$ and $\eta \in G_{r(\gamma)}$.
Choose a bisection $U \subseteq \I(G)$ with $\gamma \in U$, and a bisection $V$
containing $\eta$. Then $V U V^{-1}$ is an open subset of $\I(G)$ containing
$\eta\gamma\eta^{-1}$, so $\eta\gamma\eta^{-1} \in \I^\circ(G)$.

To see that $\Iess(G)$ is a normal subgroupoid, observe that if $\alpha,\beta \in
\Iess(G) \cap G^{(2)}$, and if $\eta \in G_{r(\alpha)}$, then the units $r(\eta)$,
$r(\alpha)$, $s(\alpha) = r(\beta)$ and $s(\beta)$ all have the same orbit closure $K$.
Since $\alpha,\beta \in \I^\circ(G|_K)$, the first statement of the lemma shows that
$\alpha\beta$, $\alpha^{-1}$ and $\eta\alpha\eta^{-1}$ all belong to $\I^\circ(G|_K)$,
and hence to $\Iess(G)$.
\end{proof}

\begin{remark}
The essential isotropy of a groupoid is an \emph{algebraic} subgroupoid. It is in general
not an open subgroupoid as demonstrated by \cref{eg:dumbbell} below. We do not
know whether the essential isotropy is always closed, but we suspect not.
\end{remark}

\begin{notation}\label{ntn:Jx}
We define $\J = \J(G) \coloneqq \overline{\Iess(G)}$, the smallest closed subgroupoid of $G$
that contains the essential isotropy. For $x\in G^{(0)}$, we let $\J_x\coloneqq \J\cap \I_x$. Since
$\I$ is closed, we have $\J \subseteq \I$.
\end{notation}

\subsection{Deaconu--Renault groupoids}\label{sec:DR gpds}

As above, we follow the notational conventions of \cite{SimsCRM} for Deaconu--Renault
groupoids.

If $X$ is a locally compact Hausdorff space, then an \emph{action} of $\N^k$ on $X$ by
local homeomorphisms is a monoid homomorphism $T \colon n \mapsto T^n$ from $\N^k$ to the
monoid of local homeomorphisms of $X$. We use the shorthand $T \colon \N^k
\curvearrowright X$ to mean that $T$ is an action of $\N^k$ on $X$ by local
homeomorphisms. The orbit of a point $x$ under $T$ is the set
\[
  [x]_T = \bigcup_{n,m\in \N^k} T^{-m}(T^n x),
\]
and $T$ is \emph{irreducible} if there exists $x \in X$ such that $\overline{[x]}_T = X$;
we say that $T$ is \emph{minimal} if $\overline{[x]}_T = X$ for all $x \in X$. If $T$ is
clear from context, we simply write $[x]$ for $[x]_T$.

Suppose that $X$ is a locally compact Hausdorff space and that $T \colon \N^k
\curvearrowright X$ is an action by local homeomorphisms. We write $G_T$ for the set
\[
  \{(x, m, y) \in X \times \Z^k \times X : \text{there exist }p,q \in \N^k\text{ such that } T^p(x) = T^q(y)\text{ and }p-q = m\}.
\]
This set is a groupoid, called the \emph{Deaconu--Renault groupoid} of $T$ with
composable pairs $\big\{\big((x,m,y), (y',n,z)\big) \in G_T\times G_T : y = y' \big\}$ and
multiplication map
\[
  (x,m,y)(y,n,z) = (x, m+n, z).
\]
The inversion operation is $(x,m,y) = (y, -m, x)$.
The unit space of $G_T$ is
\[
  G_T^{(0)} = \{(x,0,x) : x \in X\},
\]
and we silently identify it with $X$. With this identification the orbit $[x]_T$ of $x
\in X$ under $T$ as defined above agrees with the orbit $[x]$ of $x$ in $G_T$ as defined
in \cref{sec:isotropy}.

For open sets $U, V \subseteq X$ and elements $p, q \in \N^k$, we define $Z(U, p, q, V)
\subseteq G_T$ by
\[
Z(U, p, q, V) \coloneqq \{(x,p-q, y) : x \in U, y \in V\text{ and } T^p(x) = T^q(y)\}.
\]
The collection of all such sets is a basis for a locally compact Hausdorff topology on
$G_T$ under which it becomes an \'etale groupoid. If $(x, p-q, y) \in Z(U, p, q, V)$,
then using that $T^p$ and $T^q$ are local homeomorphisms, we can choose precompact open
neighbourhoods $U'$ of $x$ and $V'$ of $y$ such that $T^p|_{U'}$ and $T^q|_{V'}$ are
homeomorphisms onto their ranges. Putting $W = T^p(U') \cap T^q(V')$ and then setting
$U'' = U' \cap (T^p)^{-1}(W)$ and $V'' = V' \cap (T^q)^{-1}(W)$, we obtain a basic open
set $Z(U'', p, q, V'') \subseteq Z(U, p-q, V)$ containing $(x, p-q, y)$ with the property
that $T^p|_{U''}$ and $T^q_{V''}$ are homeomorphisms onto the same precompact open subset
$W$ of $X$. So the collection of all such sets is a basis of precompact open bisections
for the same topology on $G_T$.

There is a canonical 1-cocycle $\coc_T\colon G_T \to \Z$ (that is, a group-valued
groupoid homomorphism) on $G_T$ given by $\coc_T(x,m,y) = m$ for all $(x,m,y)\in G_T$. If
$T$ is clear from context, we just write $\coc$ for $\coc_T$. A subset $B$ of $G_T$ is
\emph{$\coc_T$-homogeneous} (or simply \emph{homogeneous}) if $\coc_T(B)$ is a singleton
subset of $\Z^k$.

Suppose that $T$ is irreducible. Writing $G$ for $G_T$, Proposition~3.1 of
\cite{Sims-Williams} says that there is an open set $Y \subset X$ such that in the
reduction $G|_Y = \{\gamma \in G : r(\gamma),s(\gamma) \in Y\}$ of $G$ to $Y$, the
interior of the isotropy is closed. We shall need to know that in fact the interior of
the isotropy in $G$ itself is closed.

\begin{lemma}\label{lem:Io closed}
Let $X$ be a second-countable locally compact Hausdorff space and suppose that $T \colon \N^k \curvearrowright X$
be an action by local homeomorphisms. If $T$ is irreducible, then $\I^\circ(G_T)$ is closed in $G_T$.
\end{lemma}

\begin{proof}
By \cite[Proposition~3.10]{Sims-Williams} there is an open set $Y \subseteq X$ such that $T^p Y
\subseteq Y$ for all $p \in \N^k$ and $\I^\circ({G_T}|_Y)$ is closed in ${G_T}|_Y$. Since $T$ is
irreducible and $Y$ is open, we have $X = \bigcup_{p \in \N^k} T^{-p}(Y)$.

Suppose that $(\gamma_n)_{n=1}^\infty$ is a sequence in $\I^\circ(G_T)$ that converges to
$\gamma \in G_T$. We must show that $\gamma \in \I^\circ(G_T)$. It clearly belongs to
$\I(G_T)$. For each $n$ let $x_n = r(\gamma_n)$ and let $r(x) = \gamma$. So each $\gamma_n =
(x_n, m_n, x_n)$ and $\gamma = (x, m, x)$ for some $m_n$ and $m$ in $\Z^k$. By discarding finitely
many terms, we can assume without loss of generality that $m_n = m$ for all $n$.

Fix $p$ such that $r(\gamma) \in T^{-p}(Y)$. Since $Y$ and hence $T^{-p}(Y)$ is open, by
discarding finitely many terms again, we can assume that $r(\gamma_n) \in T^{-p}(Y)$ for all $n$
as well.

Write $m = a - b$ with $a,b \in \N^k$ satisfying $T^a(x) = T^b(x)$ and such that $a,b \geq p$
(we can arrange this by replacing $a$ with $a + p$ and $b$ with $b + p$ for example).
Choose a neighbourhood $U \subseteq T^{-p}(Y)$ of $x$ such that $T^a$ and $T^b$ restrict to homeomorphisms
of $U$. Then $T^a U, T^b U \subseteq Y$ because $a,b \geq p$. Since $Z(U, a, b, U)$ is an open
neighbourhood of $\gamma = (x,p,x)$, we can assume without loss of generality that
$\gamma_n = (x_n, p, x_n)$ belongs to $Z(U, a, b, U)$ for all $n$.

Let $B \coloneqq \{(z, p, T^p z) : z \in U\}$. Then $\gamma \mapsto B \gamma B^{-1}$ is a
homeomorphism of $Z(U, a, b, U)$ onto $B Z(U, a, b, U) B^{-1} \subseteq {G_T}|_Y$. Hence
the sequence $B \gamma_n B^{-1}$ converges in ${G_T}|_Y$ to $B \gamma B^{-1}$. Since each
$\gamma_n \in \I^\circ(G_T)$, which is normal in $G_T$ by \cref{lem:Icirc-Iess normal},
each $B \gamma_n B^{-1} \in \I^\circ(G_T)$ too. Since the interior of the isotropy in
${G_T}|_Y$ is closed by \cite[Proposition~3.10]{Sims-Williams}, it follows that $B \gamma
B \in \I^\circ((G_T)|_Y)$. Since ${G_T}|_Y \subseteq G_T$ is open, this gives $B \gamma B
\in \I^\circ(G_T)$. Using \cref{lem:Icirc-Iess normal} again, we see that $\gamma \in
\I^\circ(G_T)$.
\end{proof}

\begin{example}
  \begin{enumerate}
    \item If $G_T$ is minimal, then the essential isotropy of $G_T$ coincides with interior of the isotropy,
      and by the above result the interior of the isotropy is closed, so $\J(G_T) = \Iess(G_T) = \I^\circ(G_T)$.
    \item If $G_T$ is strongly effective (every reduction to a closed invariant set is effective),
      then the essential isotropy is trivial and $\J(G_T)$ may be identified with the unit space $X$.
    \item If $G_T$ is minimal and effective, then $\J(G_T) = X$.
  \end{enumerate}
\end{example}

\subsection{Graph groupoids}

We will frequently use graph groupoids to describe examples. Let $E$ be a row-finite
directed graph with no sources (see \cite{RaeburnCBMS} for definitions and conventions).
The path space $E^*$ of $E$ consists of finite strings $\mu = \mu_1 \cdots \mu_n$ of
edges of $E$ such that $s(\mu_i) = r(\mu_{i+1})$ for all $i < n$; we write $r(\mu) =
r(\mu_1)$ and $s(\mu) = s(\mu_n)$. The infinite-path space $E^\infty$ of $E$ consists of
strings $x = x_1 x_2 \cdots$ all of whose initial segments $x_1 \cdots x_n$ are paths; we
write $r(x)$ for $r(x_1)$. The space $E^\infty$ is a totally disconnected locally compact
Hausdorff space under the topology generated by the cylinder sets $Z(\mu) = \{\mu x :
s(\mu) = r(x)\}$ of finite paths $\mu$, and the shift map $\sigma \colon x \mapsto x_2
x_3 \dots$ is a local homeomorphism (it restricts to a homeomorphism on $Z(\mu)$ whenever
$\mu$ has length at least 1). The resulting Deaconu--Renault groupoid $G_E = G_\sigma$ is
called the \emph{graph groupoid} of $E$.

For finite paths $\mu,\nu \in E^*$ such that $s(\mu) = s(\nu)$, both
$\sigma^{|\mu|}|_{Z(\mu)}$ and $\sigma^{|\nu|}|_{Z(\nu)}$ are homeomorphisms onto
$Z(s(\mu)) \subseteq G_\sigma^{(0)} \cong E^\infty$, and the open sets
\[
Z(\mu,\nu) \coloneqq Z(Z(\mu), |\mu|, |\nu|, Z(\nu)) = \{(\mu x, |\mu| - |\nu|, \nu x) : x \in Z(s(\mu))\}
\]
in $G_\sigma$ indexed by such pairs constitute a basis of compact open sets for
$G_\sigma$.

We will refer to the following elementary but illustrative example a number of times
throughout the paper.

\begin{example}[The Dumbbell graph]\label{eg:dumbbell}
Consider the directed graph $E$ depicted below.
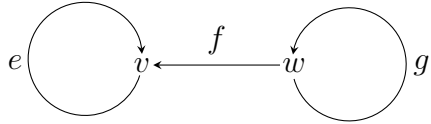
\[
\begin{tikzpicture}
    \node[inner sep=0pt, circle] (v) at (0,0) {$v$};
    \node[inner sep=0pt, circle] (w) at (2,0) {$w$};
    \draw[-stealth] (w) to node[pos=0.5, above] {$f$} (v);
    \draw[stealth-] (v.north) arc[radius=0.75, start angle=5, end angle=343];
    \draw[-stealth] (w.south) arc[radius=0.75, start angle=195, end angle=530];
    \node[inner sep=0pt, circle, anchor=east] at (-1.5,0) {$\vphantom{g}e$};
    \node[inner sep=0pt, circle, anchor=west] at (3.5,0) {$g\vphantom{e}$};
\end{tikzpicture}
\]

We call this the \emph{dumbbell graph}. The infinite-path space of $E$ is
\[
  E^\infty = \{e^\infty\} \cup \{e^n f g^\infty : n \in \N\} \cup \{g^\infty\}.
\]
A straightforward argument shows that in the topology on $E^\infty$ the subset $\{e^n f
g^\infty : n \in \N\} \cup \{g^\infty\}$ is a discrete open subset, and $E^\infty$ is
homeomorphic to the one-point compactification of this subset with $e^\infty$ as the
point at infinity. The Deaconu--Renault groupoid is
\begin{align*}
G_E = \{(e^\infty, n, e^\infty) : n \in \Z\}
    &{} \cup \{(\alpha g^\infty, |\alpha|-|\beta|, \beta g^\infty) : \alpha,\beta \in \{e^n f : n \in \N\}\}\\
    &{} \cup \{(g^\infty, n, g^\infty) : n \in \Z\}.
\end{align*}

An important point is that the topology of cylinder sets is finer than the topology
inherited from $E^\infty \times \Z \times E^\infty$. To see this, note that although
$(e^n f g^\infty, 0, e^n f g^\infty) \to (e^\infty, 0, e^\infty)$ as $n \to \infty$, for
$m \in \Z \setminus \{0\}$ the sequence $(e^n f g^\infty, m, e^n f g^\infty)_{n \in \N}$
belongs to $G_\sigma$ but has no convergent subsequence, and in particular does not
converge to $(e^\infty, m, e^\infty)$. To see this note that for any $p,q$ with $p - q =
m$, when $n \ge \max\{p,q\}$ we have $\sigma^p(e^n f g^\infty) = e^{n-p} f g^\infty \not=
e^{n-q}fg^\infty = \sigma^q(e^n f g^\infty)$, and so $(e^n f g^\infty, m, e^n f
g^\infty)$ does not belong to the open neighbourhood $Z(e^p, e^q)$ of $(e^\infty, m,
e^\infty)$.

We claim that $\Iess_{G_\sigma} = \I(G_\sigma) = \{(x, m, x) : x \in E^\infty, m \in
\Z\}$. Clearly $\Iess_{G_\sigma} \subseteq \{(x, m, x) : x \in E^\infty, m \in \Z\}$.
Since $E^\infty \setminus \{e^\infty\}$ is an open discrete subset of $E^\infty$, the set
$\{(x, m, x) : x \in E^\infty\setminus\{e^\infty\}, m \in \Z\}$ is an open discrete
subset of $G_\sigma$ and hence contained in $\Iess_{G_\sigma}$. So it suffices to show
that $\{(e^\infty, m, e^\infty) : m \in \Z\} \subseteq \Iess_{G_\sigma}$. For this, note
that the orbit of $e^\infty$ is the singleton $\{e^\infty\}$, so $\overline{[e^\infty]} =
\{e^\infty\}$, and $(G_\sigma)|_{\overline{[x]}} = \{(e^\infty, m, e^\infty) : m \in
\Z\}$ is a discrete group isomorphic to $\Z$. In particular, the interior of the isotropy
in this groupoid is the whole groupoid, and we obtain $\{(e^\infty, m, e^\infty) : m \in
\Z\} \subseteq \Iess_{G_\sigma}$ as claimed.

Since, as discussed above, $(e^{m+n} f g^\infty, m, e^n f g^\infty)_{n \ge |m|}$, the
space $\Iess_{G_\sigma}$ is not open.

It is instructive to describe the relative topology on $\Iess_{G_\sigma}$. The complement
$\Iess_{G_\sigma} \setminus E^\infty$ of the unit space is a discrete clopen subset,
while the unit space is homeomorphic to the one-point compactification of $\{e^n f
g^\infty : n \in \N\} \cup \{g^\infty\}$ as described above.

We find it helpful to picture $\Iess_{G_\sigma}$ as a subset of $\RR^3$ as follows: for
$x \in E^\infty$, we let $\theta(x) \coloneqq 1/\min\{n : x_n = g\}$ with the convention that
$\theta(e^\infty) = 0$; we then define points in $\Iess_{G_\sigma}$ with points in
$\RR^3$ by
\[
(x, m, x) \mapsto
    \begin{cases}
        (\theta(x), m, 0) &\text{ if $x \not= e^\infty$}\\
        (\theta(x), 0, m) &\text{ if $x = e^\infty$}.
    \end{cases}
\]
\end{example}

\begin{example}[Essential isotropy of graph groupoids]\label{eg:graph groupoids}
For a general graph groupoid, we can describe the essential isotropy relatively cleanly.
It suffices to discuss row-finite graphs with no sources, because up to groupoid equivalence all graph groupoids can be realised by such graphs.
Let $E$ be a row-finite with no sources (in particular, the unit space of the graph groupoid is identified with the infinite-path space).
A \emph{maximal tail} of $E$ is a set $T \subseteq E^0$ of vertices with the property that:
\begin{enumerate}
  \item $s(\alpha) \in T$ implies $r(\alpha) \in T$;
  \item if $v \in T$ then there exists $e \in vE^1$ such that $s(e) \in T$; and
  \item $T$ is cofinal in the sense that whenever $u, v \in T$ there exists $w \in T$ such that $uE^* w$ and $vE^* w$ are both nonempty.
\end{enumerate}
The orbit closures in $E^\infty$ are the sets $V_T \coloneqq \{x \in E^\infty : s(x_i) \in T
\text{ for all }i\}$ indexed by maximal tails $T$ of $E$: if $x$ is an infinite path,
then $T_x \coloneqq \bigcup_n \{v \in E^0 : v E^* s(x_n) \not= \varnothing\}$ is a maximal tail,
and $\overline{[x]} = V_{T_x}$. By \cite[Lemma~2.1]{Hong-Szymanski}, if $T$ is a maximal
tail, then $T$ can contain (up to cyclic permutation of edges and vertices) at most one
cycle $\mu$ with no entrance in $T$; if there is such a $\mu$, then $T = T_{\mu^\infty} =
\{v \in E^0 : v E^* r(\mu)\not= \varnothing\}$. It is routine to check using the arguments
of \cref{eg:dumbbell} that for a maximal tail $T$, the interior of the isotropy in
$G_{ET} \coloneqq G|_{V_T}$ is trivial if $T$ contains no cycle with an entrance, and is equal
to
\[
\{(\alpha \mu^\infty, n|\mu|, \alpha\mu^\infty) : \alpha \in E^*r(\mu), \alpha \not\in E^*\mu, n \in \Z\}
\]
if $T = T_{\mu^\infty}$ is a maximal tail containing a cycle $\mu$ with no entrance in $T$. It follows that
\begin{align*}
  \Iess(G_E) = G_E^{(0)} \cup \{(\alpha \mu^\infty, n|\mu|, \alpha\mu^\infty) : r(\mu) = s(\mu),\;
	&\mu\text{ has no entrance in } T_{\mu^\infty},\\
	&\alpha \in E^*r(\mu), \alpha \not\in E^*\mu, n \in \Z\}.
\end{align*}

To describe the topology on $\Iess$, first note that $G_E^{(0)}$ is a clopen subset of $\Iess$.
We claim that the complement of $G_E^{(0)}$ in $\Iess$ is discrete.
For this, fix a cycle $\mu$ of nonzero length with no entrance in $T = T_{\mu}$, a path $\alpha$ with $s(\alpha) = r(\mu)$ and an integer $n$.
We must show that given any sequence $\nu_i$ of cycles each having no entrance in $T_{\nu_i^\infty}$,
any sequence $\beta_i$ of paths with $s(\beta_i) = r(\nu_i)$ and $\beta_i \not\in E^*\nu_i$, and any sequence $p_i \in \Z$,
if $(\beta_i\nu^\infty_i, p_i|\beta_i|, \beta_i\nu^\infty_i) \to (\alpha\mu^\infty, n|\mu|, \alpha\mu^\infty)\}$,
then $(\beta_i\nu^\infty_i, p_i|\beta_i|, \beta_i\nu^\infty_i) = (\alpha\mu^\infty, n|\mu|, \alpha\mu^\infty)$ for large $i$.
We will argue the case when $n > 0$; the case $n < 0$ follows by taking inverses.
Observe that since the range and source maps and the cocycle $\coc : (z, m, y) \mapsto m$ are continuous,
we have $\beta_i\nu^\infty_i \to \alpha\mu^\infty$ and $p_i|\nu_i| = n|\mu|$ for large $i$; we may as well assume that $p_i|\nu_i| = n|\mu|$ for all $i$.
Fix $I$ such that $(\beta_i\nu^\infty_i, p_i|\beta_i|, \beta_i\nu^\infty_i) \in Z(\alpha \mu^{2n}, n|\mu|, \alpha\mu^n)$ for $i \ge I$.
Fix $i \ge I$.
We claim that $\beta_i\nu^\infty_i \in Z(\alpha\mu^{(2+k)n})$ for all $k \ge 0$. We prove this by induction.
The base case is trivial since $(\beta_i\nu^\infty_i, p_i|\beta_i|, \beta_i\nu^\infty_i) \in Z(\alpha \mu^{2n}, n|\mu|, \alpha\mu^n)$ implies $\beta_i\nu^\infty_i \in Z(\alpha \mu^{2n})$.
So suppose inductively that $\beta_i\nu^\infty_i \in Z(\alpha \mu^{(2+k)n})$.
Then $\beta_i\nu^\infty_i = \alpha\mu^{(2+k)n} y$ for some $y$.
Since $(\beta_i\nu^\infty_i, p_i|\beta_i|, \beta_i\nu^\infty_i) \in Z(\alpha \mu^{2n}, n|\mu|, \alpha\mu^n)$, we have
\[
\sigma^{|\alpha| + 2n|\mu|}(\beta_i \nu^\infty_i)
	= \sigma^{|\alpha| + n|\mu|}(\beta_i \nu^\infty_i)
	= \sigma^{|\alpha| + n|\mu|}(\alpha\mu^{(2+k)n} y)
	= \mu^{(1+k)n} y.
\]
Since $\beta_i\nu^\infty_i \in Z(\alpha \mu^{2n})$ by the base case,
\[
\beta_i \nu^\infty_i
	= \alpha \mu^{2n} \sigma^{|\alpha| + 2n|\mu|}(\beta_i \nu^\infty_i)
	= \alpha\mu^{2n}\mu^{(1+k)n} y
	\in Z(\alpha\mu^{(2+(k+1))n}).
\]
So $\beta_i\nu^\infty_i \in Z(\alpha\mu^{(2+k)n})$ for all $k \ge 0$ by induction, and hence $\beta_i\nu^\infty_i = \alpha\mu^\infty$.
\end{example}

\subsection{Ideals in \texorpdfstring{$C^*$}{C}-algebras}\label{sec:ideals}
Here we use the exposition from~\cite[Appendix A2]{Raeburn-Williams} to which the reader is also referred for details.

Let $A$ be a $C^*$-algebra.
By an \emph{ideal} in $A$, we will always mean a closed and two-sided ideal.
The ideals in $A$ are therefore exactly the kernels of $^*$-homomorphisms defined on $A$.
An ideal $I$ in $A$ is \emph{primitive} if it is the kernel of an irreducible representation of $A$,
and the collection of primitive ideals, $\Prim A$, is endowed with the \emph{hull-kernel} topology (or \emph{Jacobson} topology)
which is specified by the closure operation:
the closure of a subset $F\subset \Prim A$ is given by
\[
  \bar{F} = \{ P\in \Prim A : \bigcap_{I\in F} I \subset P\}.
\]
Let $I$ be an ideal in $A$.
The \emph{hull} of $I$ is the closed set of primitive ideals
\[
  h(I) = \{ P\in \Prim A : I \subset P\}.
\]
Conversely, if $F$ is a closed set of primitive ideals, then the \emph{kernel} of $F$ is
\[
  k(F) = \bigcap_{J\in F} J.
\]
The two operations are inverses of each other and therefore define a bijection between the ideals in $A$ and the closed subset of $\Prim A$.
Similarly, ideals in $A$ correspond bijectively to open sets of primitive ideals as in $I \mapsto \{P : I \not\subset P\}$.

\subsection{Primitive ideals in \texorpdfstring{$C^*$}{C*}-algebras of Deaconu--Renault groupoids}\label{sec:DR prelims}

Recall that if $G$ is an \'etale groupoid then $C^*(G)$ is the universal $C^*$-algebra
generated by a $^*$-representation of the convolution algebra $C_c(G)$ (see, for example,
\cite{SimsCRM}). Since $G^{(0)}$ is a clopen subset of $G$, the completion of
$C_c(G^{(0)})$ in $C^*(G)$ is a subalgebra isomorphic to $C_0(G^{(0)})$. We identify the
two and regard $C_0(G^{(0)})$ as a subset of $C^*(G)$.

If $G = G_T$ is the Deaconu--Renault groupoid for an action $T \colon \N^k
\curvearrowright X$, then for $z \in \T$ and for each $f \in C_c(G_T)$ the function
$\gamma_z(f) \colon (x, n, y) \mapsto z^nf(x,n,y)$ also belongs to $C_c(G_T)$. The map
$\gamma_z$ is a $^*$-homomorphism, so the universal property of $C^*(G_T)$ implies that
$\gamma_z$ extends to an endomorphism of $C^*(G_T)$. A routine $\varepsilon/3$-argument
shows that $z \mapsto \gamma_z(a)$ is continuous for $a \in C^*(G)$. Since $\gamma_z
\circ \gamma_w = \gamma_{zw}$ on $C_c(G_T)$ and since $\gamma_1$ is the identity on
$C_c(G_T)$, this $\gamma$ is an action of $\T^k$ on $C^*(G_T)$ called the \emph{gauge
action}.

Let $G_T$ be the Deaconu--Renault groupoid of an action $T\colon \N^k\curvearrowright X$.
For each $(x, z) \in X \times \T^k$, there is an irreducible representation $\pi_{x,z} :
C^*(G_T) \to \Bb(\ell^2([x]))$ such that
\[
  \pi_{x,z}(f) e_y = \sum_{\gamma\in (G_T)_y} z^{\coc(\gamma)} f(\gamma) e_{r(\gamma)}
\]
for every $f\in C_c(G_T)$ and $y\in [x]$. By~\cite[Theorem~3.2]{Sims-Williams}, there is
a surjection $\pi\colon X\times \T^k \to \Prim(C^*(G_T))$ such that $\pi(x,z) =
\ker(\pi_{x,z})$ for all $(x,z)\in X\times \T^k$. Moreover, $\pi_{x, z}$ and $\pi_{x',
z'}$ have the same kernel precisely if $\overline{[x]} = \overline{[x']}$ and $z$ and
$z'$ induce the same character on the group
\[
  H(x) \coloneqq \bigcup_{\substack{\varnothing \neq U\subset \overline{[x]} \\ U~\textrm{relatively open}}}
  \{m-n : T^m y = T^n y~\textrm{for all}~y\in U\} \subseteq \Z^k
\]
described immediately before~\cite[Theorem~3.2]{Sims-Williams}. Observe that $H(x) =
H(y)$ whenever $\overline{[x]} = \overline{[y]}$.

We claim that $H(x) = \coc(\I^\circ((G_T)|_{\overline{[x]}}))$: the containment
$\supseteq$ follows from the definition of $H(x)$, and the reverse containment follows
from \cite[Lemma~3.9]{Sims-Williams} since $(G_T)|_{\overline{[x]}}$ is irreducible by
definition. We claim further that
\[
H(x) = \coc(\Iess_x).
\]
Indeed, if $n \in H(x)$, then there exists $\gamma \in \I^\circ((G_T)|_{\overline{[x]}})$
with $\coc(\gamma) = n$. So there is an open bisection $B$ in
$\I^\circ((G_T)|_{\overline{[x]}})$ containing $\gamma$, and since $\coc$ is locally
constant, we can assume that $B \subseteq \coc^{-1}(n)$. Since $[x]$ is dense in
$\overline{[x]}$, there exists $\eta \in (G_T)_x$ with $r(\eta) \in s(B)$.
\cref{lem:Icirc-Iess normal} shows that $\eta^{-1}\gamma\eta \in
\I^\circ((G_T)|_{\overline{[x]}})_x$, and so $n \in \coc(\Iess_x)$. The reverse
containment is clear because $\Iess_x \subseteq \I^\circ((G_T)|_{\overline{[x]}})$.

\subsection{Harmonic analysis}

For the proof of our main result we will need a bit of harmonic analysis. We use notation
and results from~\cite[Chapter~8]{Folland}.

Let $K \subset \Z^k$ be a subgroup. Its annihilator is the compact subgroup $K^\perp
\coloneqq \{z\in \T^k : z^h = 1,\text{ for all }h\in K\} \le \T^k$. The annihilator acts
on $\T^k$ by translation. The Pontryagin dual of $K$ (defined as the group of continuous
homomorphisms from $H$ into $\T$) is isomorphic to the quotient group $\T^k / K^\perp
\cong \hat{K}$. Let $\psi\in C^\infty(\hat{K})$. The Fourier coefficients of $\psi$ are
\[
  \hat{\psi}(h) \coloneqq \int_{\hat{K}} \psi(\eta) \overline{\eta(h)}~d\eta,
\]
where the integration is with respect to normalised Haar measure on $\hat{K}$. When
$\psi$ is smooth (that is, all partial derivatives exist and are continuous), its Fourier
coefficients are absolutely summable, cf.~\cite[p. 257]{Folland}.

Given $h_0\in K$, we may \emph{perturb} $\psi$ by $h_0$ to obtain $\psi_{h_0}\in
C^\infty(\hat{K})$ given by
\begin{equation}\label{eq:perturbation}
  \psi_{h_0}(\eta) = \eta(h_0) \psi(\eta)
\end{equation}
for $\eta\in \hat{K}$. The Fourier coefficients of $\psi_{h_0}$ are $\hat{\psi}_{h_0}(h)
= \hat{\psi}(h - h_0)$, for $h\in K$. If $\psi$ is supported on an open subset $V \subset
\hat{K}$, then $\psi_{h_0}$ is also supported on $V$.

If $H \le K$ is a subgroup, then $K^\perp \subset H^\perp$, and there is a canonical
quotient map $\hat{q}_{H,K}\colon \hat{K} \to \hat{H}$. We define an \emph{averaging map}
$\Phi_{H,K}\colon C(\hat{K}) \to C(\hat{H})$ by
\begin{equation} \label{eq:averaging-map}
  \Phi_{H,K}(\psi)(\hat{q}_{H,K}(\eta)) \coloneqq \int_{H^\perp} \psi(\chi \eta)~d\chi
\end{equation}
for $\psi\in C(\hat{K})$ and $\eta\in \hat{K}$. Again we integrate with respect to
normalised Haar measure on $H^\perp$. If $\psi\in C(\hat{K})$ is smooth, then it is the
$\|\cdot\|_\infty$-limit of its Fourier series:
\[
  \psi(\eta) = \sum_{h\in K} \eta(h) \hat{\psi}(h).
\]
In this case,
\begin{equation} \label{eq:averaging-as-sum}
  \Phi_{H,K}(\psi)(q_{H,K}(\eta)) = \sum_{h\in K} \eta(h) \hat{\psi}(h) \int_{H^\perp} \chi(h)~d\chi = \sum_{h\in H} \eta(h) \hat{\psi}(h).
\end{equation}
In particular, $\big(\Phi_{H,K}(\psi)\hat{\;}\big)(h) = \hat{\psi}(h)$ for all $h\in H$.

If $\psi\in C(\hat{K})$ is supported on $\hat{V} \subset \hat{K}$, then
$\Phi_{H,K}(\psi)$ is supported on $\hat{q}_{H,K}(\hat{V})$.

\begin{lemma} \label{lem:H-restriction-is-nonzero}
Suppose that $z \in \T^k$ and $\psi\in C^\infty(\T^k)$ satisfy $\psi(z) \not= 0$.
For $h_0 \in \Z^k$ let $\psi_{h_0}$ be as in~\labelcref{eq:perturbation}.
For any subgroup $H \subset \Z^k$ there exists $h_0\in \Z^k$ such that
\[
  \sum_{h \in H} z^h \hat{\psi}_{h_0}(h) \not= 0.
\]
\end{lemma}

\begin{proof}
Since $\psi$ is smooth, it is the norm limit of its Fourier series, so
\[
  0 \not= \psi(z) = \sum_{h\in \Z^k} z^h \hat{\psi}(h).
\]

Choose a section $\sigma$ for the quotient map $\Z^k \to \Z^k / H$.
For each $y\in \Z^k/H$, consider the $\sigma(y)$-pertubation of $\psi$,
\begin{equation} \label{eq:psi_y}
  \psi_{\sigma(y)}(\eta) = \eta(\sigma(y)) \psi(\eta).
\end{equation}
As discussed above, for $h \in \Z^k$ the corresponding Fourier coefficient of
$\psi_{\sigma(y)}$ is $\hat{\psi}_{\sigma(y)}(h) = \hat{\psi}\big(h - \sigma(y)\big)$. We
have
\begin{equation}\label{eq:perturbation formula}
  \Phi_{H, \Z^k}(\psi_{\sigma(y)})\big(\hat{q}(\chi)\big) = \sum_{h \in H} z^h \hat{\psi}\big(h - \sigma(y)\big).
\end{equation}
Let $q : \T^k \to \hat{H}$ be the quotient map so that $q(z)(h) = z^h$ for $h \in H$ and
$z \in \T^k$.  Since the Fourier coefficients of $\psi$ are absolutely summable, we can
rearrange the summation~\eqref{eq:perturbation formula} to see that
\[
  0\not= \sum_{h \in \Z^k} z^h \hat{\psi}(h)
  = \sum_{y \in \Z^k/H} \Big(\sum_{h \in H} z^h \hat{\psi}\big(h - \sigma(y)\big)\Big)
  = \sum_{y \in \Z^k/H} \Phi_{H, \Z^k}(\psi_{\sigma(y)})(q(z)).
\]
In particular, there exists $y\in \Z^k/H$ such that $h_0 = \sigma(y)$ satisfies
\[
  0 \neq \Phi_{H, \Z^k}(\psi_{h_0})(z) = \sum_{h\in \Z^k} z^h \hat{\psi}(h).\qedhere
\]
\end{proof}

\section{A family of representations}\label{sec:nice reps}

Our analysis of ideals depends upon the existence and behaviour of representations of the
$C^*$-algebra of a Deaconu--Renault groupoid that interpolate between the regular
representation on $\ell^2((G_T)_x)$ and the representation on $\ell^2([x])$ induced by
the action of $G_T$ on $[x]$. We establish these technical results in this section.

Let $X$ be a second-countable locally compact Hausdorff space and let $T \colon \N^k
\curvearrowright X$ be an action by local homeomorphisms. Suppose that $H$ is a subgroup
of $\Z^k$. We let $\sim_H$ be the equivalence relation on $G_T$ given by
\[
    (x_1,h_1,y_1)\sim_H (x_2,h_2,y_2)\iff x_1=x_2,\ y_1=y_2,\text{ and }h_1-h_2\in H.
\]
This is an equivalence relation, and for $\xi\in G_T$, we let $[\xi]_H$ denote the equivalence class of $\xi$ with respect to
$\sim_H$.

\begin{proposition} \label{prp:H-representations}
Let $X$ be a second-countable locally compact Hausdorff space and suppose
that $T \colon \N^k \curvearrowright X$ is an action by local homeomorphisms.
For each $(x,z)\in X\times\T$ there is a representation $\pi^H_{(x,z)} \colon
C^*(G_T) \to B(\ell^2((G_T)_x/{\sim_H}))$ such that for each $f \in C_c(G_T)$
and $\xi_1, \xi_2\in (G_T)_x$,
\begin{equation}\label{eq:piHxz def}
  \langle e_{[\xi_1]_H},\pi^H_{(x,z)}(f)e_{[\xi_2]_H}\rangle=\sum_{(x_1,h,x_2)\in [\xi_1\xi_2^{-1}]_H} z^h f(x_1,h,x_2).
\end{equation}
If $H_1$ and $H_2$ are subgroups of $\Z^k$ and $H_1\subseteq H_2$, then
$\ker(\pi^{H_1}_{(x,z)})\subseteq \ker(\pi^{H_2}_{(x,z)})$.
\end{proposition}

\begin{proof}
The canonical cocycle $\coc_T  \colon G_T \to \Z^k$ defined by $c_T((y,n,x)) = n$ for every $(y,n,x)\in G_T$ determines a coaction (dual to the gauge action)
$\delta_T \colon C^*(G_T) \to C^*(G_T) \otimes C^*(\Z^k)$ that satisfies $\delta_T(f) = f \otimes u_n$
whenever $f \in C_c(c_T^{-1}(n)) \subseteq C_c(G_T)$.

Fix $x\in X$ and consider the representation $\varepsilon_x \colon C^*(G_T) \to
B(\ell^2([x]))$ of~\cite[Proposition 5.2]{Brown-Clark-Farthing-Sims} satisfying
$\varepsilon_x(f)e_y = \sum_{\gamma \in (G_T)_y} f(\gamma)e_{r(\gamma)}$ for all $f \in
C_c(G_T)$ and $y \in [x]$. Fix a subgroup $H \le \Z^k$ and $z\in \T^k$. Let $\{u_n : n
\in \Z^k\} \subseteq C^*(\Z^k)$ be the canonical generators. Similarly, let $\{u_{n+H} :
n+H \in \Z^k/H\} \subseteq C^*(\Z^k/H)$ be the canonical generators. Since $n + H \mapsto
z^n u_{n+H}$ is a unitary representation of $\Z^k$, there is a homomorphism $\rho_{z, H}
\colon C^*(\Z^k) \to C^*(\Z^k/H)$ such that $\rho_{z,H}(u_n) = z^n u_{n+H}$. Finally, let
$\lambda^{\Z^k/H}$ be the left regular representation of $C^*(\Z^k/H)$ on $\ell^2
(\Z^k/H)$.

Consider the representation
\begin{equation}
  \psi^H_{(x,z)} \coloneqq (1 \otimes \lambda^{\Z^k/H})\circ (\varepsilon_x \otimes \rho_{z, H})\circ \delta_T
    \colon C^*(G_T) \to B\big(\ell^2([x]) \otimes \ell^2(\Z^k/H)\big).
\end{equation}

Observe that $\coc(\I(G_T)_x)$ is a subgroup of $\Z^k$, so $\coc(\I(G_T)_x) + H$ is also
a subgroup. Let $\Sigma \subseteq \Z^k$ be a complete set of coset representatives for
the cosets of $\coc(\I(G_T)_x) + H$ in $\Z^k$ such that $0 \in \Sigma$ is the
representative of the coset $\coc(\I(G_T)_x) + H$ itself.

Our strategy is as follows. We decompose $\ell^2([x]) \otimes \ell^2(\Z^k/H)$ into a
direct sum of subspaces $\mathcal{H}_t$, indexed by elements $t$ in $\Sigma$, that are
invariant for $\psi^H_{(x,z)}$. We then identify a unitary isomorphism $W_j\colon
\ell^2((G_T)_x/{\sim_H}) \to \mathcal{H}_0$ such that $\pi^H_{(x,z)} \coloneqq \Ad(W_j^*)
\circ \psi^H_{(x,z)}|_{\mathcal{H}_0}$ satisfies~\labelcref{eq:piHxz def}, establishing
the first statement. We then show that the restrictions $\psi^H_{(x,z)}|_{\mathcal{H}_t}$
are all unitarily equivalent to one another, and conclude that $\ker(\pi^H_{(x,z)}) =
\ker(\psi^H_{(x,z)})$. Finally we will show that if $H_1 \le H_2 \le \Z^k$, then
$\ker(\psi^{H_1}_{(x,z)}) \subseteq \ker(\psi^{H_2}_{(x,z)})$, giving the final
statement.

We first claim that
\begin{equation}\label{eq:set decomp}
    [x] \times \Z^k/H = \bigsqcup_{t \in \Sigma} \{(y, n + t + H) : (y, n, x) \in G_T\}.
\end{equation}
It is clear that the left-hand side contains the union of the sets on the right, so we
have to show that the union on the right contains the left-hand side, and that the sets
on the right are disjoint. For the first assertion, fix $y \in [x]$ and $n\in \Z^k$ so
that $(y, n+H)$ is a typical element of $[x] \times \Z^k/H$. Since $y \in [x]$ there
exists $n_y$ such that $(y, n_y, x) \in G_T$. Let $t \in \Sigma$ be the unique element
such that $t \in n - n_y + H + \coc(\I(G_T)_x)$. So $n = n_y + t + h + m$ for some $h \in
H$ and $m \in \coc(\I(G_T)_x)$. Since $m \in \coc(\I(G_T)_x)$ we have $(x, m, x) \in
G_T$, and so $(y, n_y + m, x) = (y, n_y, x)(x, m, x) \in G_T$. We then have $(y, n + H) =
(y, n - h + H) = (y, (n_y + m) + t + H)$, which belongs to the right-hand side
of~\labelcref{eq:set decomp}. We must now show that $\{(y, n + t + H) : (y, n, x) \in
G_T\}_{t \in \Sigma}$ are mutually disjoint. So suppose that $s,t \in \Sigma$ and that
$(y, n + t + H) = (y, n' + s + H)$ for some $(y, n, x), (y, n', x) \in G_T$. So $s - t
\in n - n' + H$. We have $(x, n-n', x) = (x, n, y)(x, n',y)^{-1} \in G_T$ so that $n - n'
\in \coc(\I(G_T)_x)$. Hence $s - t \in \coc(\I(G_T)_x) + H$. Therefore, $s\in \Sigma\cap
(t + \coc(\I(G_T)_x) + H) = \{t\}$, so $s = t$ as required.

For $t \in \Sigma$, we define
\[
  \mathcal{H}_t \coloneqq \overline{\operatorname{span}}\{e_y \otimes e_{n + t + H} : (y,n,x) \in G_T\}.
\]
Then \labelcref{eq:set decomp} shows that $\ell^2([x]) \otimes \ell^2(\Z^k/H) \cong \bigoplus_{t \in \Sigma} \mathcal{H}_t$.
To show that each $\mathcal{H}_t$ is invariant for $\psi^H_{(x,z)}$, we fix $t \in \Sigma$.
Since $C^*(G_T)$ is the closed linear span of functions supported on basic open
sets, it suffices to fix open sets $U,V \subseteq X$ and $p,q \in \N^k$ such that
$T^p|_U$ and $T^q|_V$ are homeomorphisms onto the same open set $W \subset X$, a function $f \in
C_c(Z(U, p, q, V))$ and an element $(y, n, x) \in G_T$ so that $e_y \otimes e_{n + t +
H}$ is a typical basis element of $\mathcal{H}_t$, and show that $\psi^H_{(x,z)}(f)(e_y
\otimes e_{n + t + H}) \in \mathcal{H}_t$. This is a straightforward calculation:
\begin{align*}
  \psi^H_{(x,z)}(f)(e_y \otimes e_{n + t + H})
  &= (1\otimes \lambda^{\Z^k/H})\circ (\varepsilon_x(f) \otimes z^{p-q}u_{p-q + H})(e_y \otimes e_{n + t + H}) \\
    &= \sum_{\gamma \in (G_T)_y} z^{p-q}f(\gamma) (e_{r(\gamma)} \otimes e_{n+p-q+t+H}).
\end{align*}
If $y \not\in V$, then $f(\gamma) = 0$ for all $\gamma \in (G_T)_y$, and then $\psi^H_{(x,z)}(f)(e_y \otimes e_{n + t + H}) = 0$.
If $y \in V$, then there is a unique $u \in U$ such that $T^p(u) = T^q(y)$,
and then $\gamma = (u, p-q, y)$ is the unique element of $(G_T)_y$ such that $f(\gamma) \neq 0$, so the calculation above gives
\begin{equation}\label{eq:psizH action}
  \psi^H_{(x,z)}(f)(e_y \otimes e_{n + t + H}) = z^{p-q}f(u, p-q, y) (e_u \otimes e_{n+p-q+t+H}).
\end{equation}
Since $(u, n + p-q, x) = (u, p-q, y)(y, n, x) \in G_T$, we have $e_u \otimes e_{n+p-q+t+H} \in \mathcal{H}_t$,
so $\mathcal{H}_t$ is invariant for $\psi^H_{(x,z)}$ as claimed.

In particular, the subspace $\mathcal{H}_0 = \overline{\Span}\{e_y \otimes e_{n + H} :
(y, n, x) \in G_T\}$ is invariant for $\psi^H_{(x,z)}$. We show that $\mathcal{H}_0$ is
isomorphic to $\ell^2((G_T)_x/{\sim_H})$. To see this, observe that there is a map
$\tilde j \colon (G_T)_x \to [x]\times\Z^k/H$ satisfying $\tilde j(y,n,x) = (y, n+H)$,
and we have $(y, n, x) \sim_H (y', n', x)$ if and only if $y = y'$ and $n - n' \in H$,
and hence if and only if $\tilde j(y, n, x) = \tilde j(y', n', x)$. This means that there
is an injective map $j \colon (G_T)_x/{\sim_H} \to [x] \times \Z^k/H$ satisfying
$j([y,n,x]) = (y, n+H)$. Since $j$ induces a bijection $e_{[y, n, x]} \mapsto e_y \otimes
e_{n+H}$ between orthonormal bases for $\ell^2((G_T)_x/{\sim_H})$ and $\mathcal{H}_0$, it
induces a unitary $W_j  \colon \ell^2((G_T)_x/{\sim_H}) \to \mathcal{H}_0$.

Since $\mathcal{H}_0$ is invariant for $\psi^H_{(x,z)}$ we obtain a representation
\[
  \pi^H_{(x,z)} \coloneqq \Ad(W^*_j) \circ \psi^H_{(x,z)}\colon C^*(G_T) \to B(\ell^2((G_T)_x/{\sim_H})).
\]
We claim that this representation satisfies~\labelcref{eq:piHxz def}. Once again it
suffices to establish~\labelcref{eq:piHxz def} for $f \in C_c(Z(U, p-q, V))$ where
$T^p|_U$ and $T^q|_V$ are homeomorphisms onto the same open set $W$. Fix $[\xi_1] = [w,
m, x]$ and $[\xi_2] = [y, n, x]$ in $(G_T)_x/{\sim_H}$. As in the paragraph
including~\labelcref{eq:psizH action}, we have
\[
    \langle e_{[\xi_1]_H},\pi^H_{(x,z)}(f)e_{[\xi_2]_H}\rangle
        =\begin{cases}
            z^{p-q}f(w, p-q, y) &\text{ if $y \in V$, $w \in U$, $T^p(z) = T^q(y)$,}\\
                        &\qquad\text{and $n+p-q+H = m+H$,}\\
            0 &\text{ otherwise.}
        \end{cases}
\]
Note that $[\xi_1 \xi_2^{-1}] = \{(w, m-n+h, y) : h \in H\} \cap G_T$ and that $f$ is
nonzero at at most one point in this set, which occurs if $y \in V$, $w \in U$, $T^p(w) =
T^q(y)$ and $m-n \in p-q + H$. So
\[
    \sum_{(x_1,h,x_2)\in [\xi_1\xi_2^{-1}]_H} z^h f(x_1,h,x_2)
        = \begin{cases}
            z^{p-q}f(w, p-q, y) &\text{ if $y \in V$, $w \in U$, $T^p(z) = T^q(y)$,}\\
                    &\qquad\text{and $p-q+H = m-n+H$,}\\
            0 &\text{ otherwise.}
        \end{cases}
\]
Comparing the last two displayed equations we see that $\pi^H_{(x,z)}$ satisfies~\labelcref{eq:piHxz def} as claimed.
This completes the proof of the first statement.

Next we claim that $\ker(\pi^H_{(x,z)}) = \ker(\psi^H_{(x,z)}) = \ker((\varepsilon_x \otimes \rho_{z, H})\circ \delta_T)$.
First note that for each $t \in \Sigma$ the map $e_y \otimes e_{n + t + H} \mapsto e_y \otimes e_{n + H}$ is
a bijection between orthonormal bases for $\mathcal{H}_t$ and $\mathcal{H}_0$ and so
induces a unitary $V_t \colon \mathcal{H}_t \to \mathcal{H}_0$.
The definition of $V_t$ and the paragraph including~\labelcref{eq:psizH action} show that $U_t$ commutes with $\psi^H_{(x,z)}(f)$ for $f \in C_c(Z(U, p, q, V)$.
Hence $\Ad(V_t) \circ \psi^H_{(x,z)}|_{\mathcal{H}_0} = \psi^H_{(x,z)}|_{\mathcal{H}_t}$.
It follows that
\[
  \psi^H_{(x,z)}
  = \bigoplus_{t \in \Sigma} \Ad(V_t) \circ \psi^H_{(x,z)}|_{\mathcal{H}_0}
  = \bigoplus_{t \in \Sigma} \Ad(V_t W_j) \circ \pi^H_{(x, z)},
\]
so $\ker(\pi^H_{(x, z)}) = \ker(\psi^H_{(x,z)})$. Since $\Z^k/H$ is amenable, the regular
representation $\lambda^{\Z^k/H}$ is faithful, and $C^*(\Z^k/H)$ is nuclear so $1 \otimes
\lambda^{Z^k/H}$ is faithful on $B(\ell^2([x])) \otimes C^*(\Z^k/H)$. Therefore,
$\ker(\psi^H_{(x,z)}) = \ker((\varepsilon_x \otimes \rho_{z, H})\circ \delta_T)$.

It remains to verify that if $H_1 \le H_2 \le \Z^k$ are subgroups then $\ker(\pi^{H_1}_{(x,z)}) \subseteq \ker(\pi^{H_2}{(x,z)})$ for all $(x,z)\in X\times \T^k$.
From the above observations, we just need to show that
\begin{equation}\label{eq:containment}
\ker(\varepsilon_x \otimes \rho_{z, H_1}) \subseteq \ker(\varepsilon_x \otimes \rho_{z, H_2}).
\end{equation}
The quotient map $n + H_1 \mapsto n + H_2$ from $\Z^k/H_1$ to $\Z^k/H_2$ induces a homomorphism
$q \colon C^*(\Z^k/H_1) \to C^*(\Z^k/H_2)$ satisfying $\rho_{(z, H_2)} = q \circ \rho_{(z, H_1)}$.
Therefore $(1 \otimes q)\circ(\varepsilon_x \otimes \rho_{z, H_1}) = \varepsilon_x \otimes (q\circ \rho_{z, H_1}) = \varepsilon_x \otimes \rho_{z, H_2}$,
and this proves~\labelcref{eq:containment}.
\end{proof}

\begin{remark}\label{rmk:H=Zk}
When $H = \Z^k$, the equivalence relation $\sim_{\Z^k}$ is given by $(x, m,
y) \sim_{\Z^k} (w, n, z)$ if and only if $x = w, y = z$ and $m - n \in \Z^k$;
that is, if and only if $x = w$ and $y = z$. Hence there is a bijection
$(G_T)_x/{\sim_{\Z^k}} \to [x]$ given by $[y, n, x] \mapsto y$. Using this
bijection to induce a unitary $\ell^2((G_T)_x/{\sim_{\Z^k}}) \to
\ell^2([x])$, we see that $\pi^{\Z^k}_{(x, z)}$ is unitarily equivalent to
the representation $\pi_{x, z}$ of $C^*(G_T)$ on $\ell^2([x])$ appearing
in~\cite[Theorem~3.2]{Sims-Williams} (cf.~\cref{sec:DR prelims}).
\end{remark}

It will be important later that the nesting of kernels described in the final statement of
Proposition~\ref{prp:H-representations} is equality if $H_1$ is the image under $\coc_T$ of the
essential isotropy at $x$, and $H_2$ is the whole of $\Z^k$.

\begin{lemma}\label{lem:piH kernels equal}
Let $X$ be a second-countable locally compact Hausdorff space and suppose that
$T \colon \N^k \curvearrowright X$ is an action by local homeomorphisms. Let
$\coc_T : G_T \to \Z^k$ be the canonical cocycle. Fix $(x, z) \in X \times
\T^k$, and let $K := \coc_T(\Iess_x)$. Then $\ker(\pi^K_{(x,z)}) =
\ker(\pi_{(x,z)})$.
\end{lemma}
\begin{proof}
Let $Y := \overline{[x]} \subseteq X$, the orbit-closure of $x$. Then $Y$ is a
closed invariant subset of $X$ for $T$, so $T$ restricts to an action $S : \N^k
\curvearrowright Y$. The resulting Deaconu--Renault groupoid $G_S$ is precisely
the reduction of $(G_T)$ to the closed invariant subspace $Y$ of its unit
space. Let $q_Y : C^*(G_T) \to C^*(G_S)$ be the surjective homomorphism
extending restriction of functions $C_c(G_T) \to C_c(G_S)$
\cite[Proposition~10.3.2]{SimsCRM}.

By definition, $\pi^K_{(x, z)}$ and $\pi_{(x, z)}$ annihilate $\ker(q_Y)$. Consequently, 
$\pi^K_{(x, z)}$ and $\pi_{(x, z)}$ factor through the corresponding representations $\widetilde{\pi}^K_{(x,
z)}$ and $\widetilde{\pi}_{(x, z)}$ of $C^*(G_S)$ respectively. So it suffices to show that
$\ker(\widetilde{\pi}^K_{(x, z)}) = \ker(\widetilde{\pi}_{(x, z)})$.

Since $G_S^{(0)} = Y = \overline{[x]}$, the groupoid $G_S$ is
irreducible. Hence Lemma~\ref{lem:Io closed} implies
that $\I^\circ(G_S)$ is closed in $G_S$, and so
\cite[Proposition~2.5]{Sims-Williams} implies that $G' := G_S/\I^\circ(G_S)$ is
an amenable effective locally compact Hausdorff \'etale groupoid. By
definition, $\I^\circ(G_S) = \Iess(G_T)_x = \{(y, h, y) : y \in
\overline{[x]}\text{ and } h \in K\}$. In particular, the set $G'_x$ is equal
to $(G_S)_x/{\sim_K}$, and the orbit $[x]_{G'}$ is equal to the orbit
$[x]_{G_S}$. So we can identify $\ell^2(G'_x)$ with $\ell^2((G_S)_x/{\sim_K})$ and
$\ell^2([x]_{G'})$ with $\ell^2([x]_{G_S})$. Consequently, we can regard the
regular representation $\pi^{G'}_x$ of $C^*(G')$ for $x \in (G')^{(0)}$ as a
representation on $\ell^2((G_S)_x/{\sim_K})$, and the
representation $\varepsilon_x : C^*(G') \to \ell^2([x]_{G'})$ such that
$\varepsilon_x(f)(\delta_y) = \sum_{\alpha \in G'_y} f(\alpha)
\delta_{r(\alpha)}$ (see \cite[Proposition~5.2]{Brown-Clark-Farthing-Sims}) as
a representation on $\ell^2([x]_{G_S})$.

Proposition~2.6 of \cite{Sims-Williams} shows that there is a homomorphism
$\kappa : C^*(G_S) \to C^*(G')$ such that $\kappa(f)([\gamma]) = \sum_{\eta \in
[\gamma]} f(\eta)$ for all $f \in C_c(G_S)$. Let $\gamma_z \in \Aut(C^*(G_S))$
be the automorphism such that $\gamma_z(f)(u, n, v) = z^n f(u, n, v)$ for $f
\in C_c(G_S)$ as in Section~\ref{sec:DR prelims}. Direct calculation shows
that, with the identifications of Hilbert spaces in the preceding paragraph,
$\widetilde{\pi}^K_{(x, z)} = \pi^{G'}_x \circ \kappa \circ \gamma_z$ and
$\widetilde{\pi}_{(x, z)} = \varepsilon_x \circ \kappa \circ \gamma_z$. Since
$\overline{[x]} = (G')^{(0)}$, both $\pi^{G'}_x$ and $\varepsilon_x$ are injective on
$C_0((G')^{(0)})$. Hence \cite[Theorem~10.2.7]{SimsCRM} shows that they are
both injective. Consequently, $\ker(\widetilde{\pi}^K) = \ker(\kappa \circ
\gamma_z) = \ker(\widetilde{\pi}_{(x, z)})$ as required.
\end{proof}

The following technical lemma will be helpful in identifying elements in the
kernels of the representations $\pi^H_{(x, z)}$.

\begin{lemma}\label{lem:easier zero check}
Let $H$ be a subgroup of $\Z^k$, fix $(x,z) \in X \times \T$, and let $\pi^H_{(x,z)}$ be
the representation of~\cref{prp:H-representations}. For $f \in C^*(G_T)$, we have
$\pi^H_{(x,z)}(f) = 0$ if and only if for every pair of homogeneous bisections $B_1, B_2
\subseteq G_T$ and every pair of functions $h_i \in C_c(B_i)$, we have $\langle
e_{[(x,0,x)]_H},\pi^H_{(x,z)}(h_1fh_2)e_{[(x,0,x)]_H}\rangle = 0$. In particular, an
ideal $I$ is contained in $\ker(\pi^H_{(x, z)})$ if and only if
\[
  \langle e_{[(x,0,x)]_H},\pi^H_{(x,z)}(f)e_{[(x,0,x)]_H}\rangle = 0,
\]
for all $f \in I$.
\end{lemma}
\begin{proof}
 If $\pi^H_{(x,z)}(f) = 0$, then each $\pi^H_{(x,z)}(h_1fh_2) = 0$, so the ``only if'' implication
is immediate. We just need to prove the ``if'' implication.

Fix $f\in C_c(G_T)$ and suppose that for every pair of homogeneous bisections $B_1, B_2
\subseteq G_T$ and every pair of functions $h_i \in C_c(B_i)$, we have $\langle
e_{[(x,0,x)]_H},\pi^H_{(x,z)}(h_1fh_2) e_{[(x,0,x)]_H}\rangle = 0$. It suffices to show
that for all $\xi_1 = (y_1, n_1, x)$ and $\xi_2 = (y_2, n_2, x)$ in $(G_T)_x$, we have
\[
\langle e_{[\xi_1]_H},\pi^H_{(x,z)}(f) e_{[\xi_2]_H}\rangle = 0.
\]
By definition of the topology on $G_T$ there are open bisections $B_1 \subseteq c_T^{-1}(-n_1)$ and $B_2 \subseteq c_T^{-1}(n_2)$
such that $\xi_1^{-1} \in B_1$ and $\xi_2 \in B_2$,
and by Urysohn's lemma we can find $h_i \in C_c(B_i)$ such that $h_1(\xi_1^{-1}) = 1 = h_2(\xi_2)$.

Let $\eta\in (G_T)_x$. Since $h_1$ is supported on a bisection containing $\xi_1$, the
formula~\labelcref{eq:piHxz def} implies that
\begin{equation*}
  \langle e_{[\eta]_H}, \pi_{(x,z)}^H(h_1) e_{[\xi_1]_H} \rangle
  = \sum_{(u, p, y_1)\in [\eta \xi_1^{-1}]_H} z^p h_1(u,p,y_1)
  = z^{-n_1} h_1(\xi_1^{-1}) = z^{-n_1}.
\end{equation*}
We deduce that $\pi_{(x,z)}^H(h_1) e_{[\xi_1]_H} = z^{-n_1} e_{[\eta]_H}$ where $[\eta]_H = [(x,0,x)]_H$
and therefore that $\pi^H_{(x,z)}(h_1)^* e_{[(x,0,x)]_H} = z^{n_1} e_{[\xi_1]_H}$.
A similar calculation shows that $\pi^H_{(x,z)}(h_2) e_{[(x,0,x)]_H} = z^{n_2} e_{[\xi_2]_H}$.
Hence
\[
  0 = \langle e_{[(x,0,x)]_H},\pi^H_{(x,z)}(h_1 f h_2) e_{[(x,0,x)]_H}\rangle
    =z^{n_2-n_1}\langle e_{[\xi_1]_H},\pi^H_{(x,z)}(f) e_{[\xi]_H}\rangle,
\]
and this shows that $\pi_{(x,z)}^H(f) = 0$.

For the final statement, the ``only if'' implication is trivial. For the ``if''
direction, fix $f \in I$. Then for any pair of homogeneous bisections $B_1, B_2 \subseteq
G_T$ and any pair of functions $h_i \in C_c(B_i)$, we have $h_1fh_2 \in I$, and so
$\langle e_{[(x,0,x)]_H},\pi^H_{(x,z)}(h_1fh_2)e_{[(x,0,x)]_H}\rangle = 0$ by hypothesis.
Hence $\pi^H_{(x, z)}(f) = 0$ by the first statement.
\end{proof}

We shall need to know that averaging over subgroups of $H$ preserves the kernel of $\pi^H_{(x, z)}$.

\begin{lemma}\label{lem:averaging and kernels}
Let $X$ be a second-countable locally compact Hausdorff space and suppose that $T \colon \N^k
\curvearrowright X$ is an action by local homeomorphisms. Fix nested subgroups $H \leq K$
of $\Z^k$ and a point $(x, z) \in X \times \T^k$. Let $E_K$ be the conditional
expectation on $C^*(G_T)$ given by
\[
  E_K(f)(x, n, y) = \int_{K^\perp} w^n f(x, n, y)\,dw,
\]
for all $f\in C_c(G_T)$ and $(x,n,y)\in G_T$.
Then $E_K\big(\ker(\pi^H_{(x, z)})\big) \subseteq \ker(\pi^H_{(x, z)})$.
\end{lemma}

\begin{proof}
Let $\alpha \colon K^\perp \to \Aut(C^*(G_T))$ be the action given by $\alpha_w(f)(x, n,
y) = w^n f(x, n, y)$, for all $f\in C_c(G_T)$ and $(x,n,y)\in G_T$. Then $E_K(f) =
\int_{K^\perp} \alpha_w(f)\,dw$. Fix basis vectors $e_{[(u, m, x)]_H}, e_{([v, n, x)]_H}$
of $\ell^2((G_T)_x/{\sim_H})$. For any $f\in \ker(\pi_{(x,z)}^H)$, we calculate,
using~\cite[Lemma~C.2]{Raeburn-Williams} at the first equality, and that $H \le K^\perp$
at the last equality:
\begin{align*}
  \langle e_{[(u,m,x)]_H}, \pi_{(x,z)}^H(E_K(f)) e_{[(v,n,x)]_H} \rangle
  &= \int_{K^\perp} \sum_{h\in H} z^{h+m-n} w^{h+m-n} f(u, h+m-n, v) \,dw \\
  &= \int_{K^\perp} w^{m-n} \sum_{h\in H} z^{h+m-n} f(u, h+m-n, v) \,dw \\
  &= \int_{K^\perp} w^{m-n} \langle e_{[(u,m,x)]_H}, \pi_{(x,z)}^H(f) e_{[(v,n,x)]_H}\rangle \,dw\\
  &= 0.
\end{align*}
Therefore, $\pi_{(x,z)}^H(E_K(f)) = 0$ and the claim follows.
\end{proof}

\section{The map \texorpdfstring{$\pi$}{pi} is continuous}\label{sec:necessary condition}

We show in this section that the map from $X \times \Z^k$ to the
primitive-ideal space of the $C^*$-algebra of a Deaconu--Renault groupoid
determined by the irreducible representations $\pi^{\Z^k}_{(x,z)}$ from the
preceding section is continuous. We will use the following notation
throughout the remainder of the paper.

\begin{notation}\label{ntn:pi}
Let $X$ be a second-countable locally compact Hausdorff space and suppose that $T
\colon \N^k \curvearrowright X$ is an action by local homeomorphisms. For
each $x \in X$ and $z \in \T^k$, we write
\begin{equation}\label{eq:pisubxz def}
\pi_{(x,z)} \coloneqq \pi_{(x,z)}^{\Z^k} \colon C^*(G_T) \to B(\ell^2((G_T)_x/{\sim_{\Z^k}}))
\end{equation}
for the representation obtained from \cref{prp:H-representations} applied
with $H = \Z^k$. By \cite[Theorem~3.2]{Sims-Williams} and \cref{rmk:H=Zk},
the map
\begin{equation}\label{eq:pi def}
\pi \colon X \times \T^k  \to \Prim(C^*(G_T))
\end{equation}
defined by $\pi(x,z) \coloneqq \ker(\pi_{(x,z)})$ is a surjection.
\end{notation}

With the notation above, we have $\ker(\pi_{(x, z)}) = \ker(\pi_{(x', z')})$
if and only if $\overline{[x]} = \overline{[x']}$ and $z$ and $z'$ determine
the same character of $\coc(\Iess_x) = \coc(\Iess_{x'})$.

Suppose that $\Bb = \{B_i : i \in I\}$ is a family of bisections of $G_T$. We write
$\Bb^{\ess}$ for the intersection
\[\textstyle
\Bb^{\ess} \coloneqq \big(\bigcup \Bb\big) \cap \Iess = \big(\bigcup_{i \in I} B_i\big) \cap \Iess.
\]
This is an algebraic bundle of subsets of $\Z^k$ (its fibres are not necessarily groups,
and it need not be particularly well behaved topologically; for example, it is unlikely
to be locally compact). For $x \in X$, we write $\Bb^{\ess}_x$ for the fibre $\Bb^{\ess}
\cap G_x$ of this bundle over $x$. We then write
\[
(\Bb^{\ess})^\perp =
    \{(x, z) \in (s(\Bb^{\ess}) \times \T^k) : z^{\coc(\gamma)} = 1\text{ for all }\gamma \in \Bb^{\ess}_x\}.
\]
Algebraically, this is a bundle over $s(\Bb^{\ess})$ of subgroups of $\T^k$, though it
need not be topologically well-behaved. We think of it as the bundle of annihilators of
the fibres of $\Bb^{\ess}$.

Given a subset $W \subseteq s(\Bb^{\ess}) \times \T^k$, we define the
\emph{$\Bb$-saturation} of $W$ to be the set
\[
W \cdot (\Bb^{\ess})^\perp = \{(x, wz) : (x,w) \in W\text{ and } (x,z) \in (\Bb^{\ess})^\perp\}.
\]
Equivalently,
\[
W \cdot (\Bb^{\ess})^\perp = \bigcup_{(x,w) \in W} \{(x, z) : z^{\coc(\gamma)} = w^{\coc(\gamma)}\text{ for all } \gamma \in \Bb^{\ess}_x\}.
\]

\begin{theorem}\label{thm:Aopen}
Let $X$ be a second-countable locally compact Hausdorff space and suppose
that $T \colon \N^k \curvearrowright X$ is an action by local homeomorphisms.
Let $\pi \colon X \times \T^k \to \Prim(C^*(G_T))$ be as in \cref{ntn:pi}. Let
$A\subset \Prim(C^*(G_T))$ be an open subset. Suppose that $(x,z) \in
\pi^{-1}(A)$ and that $(B_{\alpha})_{\alpha\in\J_x}$ is a family of open
bisections such that $\alpha \in B_{\alpha}\subseteq \coc^{-1}(\coc(\alpha))$
for each $\alpha\in\J_x$. Then there exist an open neighbourhood $U\subseteq
B_x\cap X$ of $x$ and an open neighbourhood $V \subset \T^k$ of $z$ such that
the $\Bb$-saturation $(U \times V)\cdot(\Bb^{\ess})^\perp$ of $U \times V$ is
contained in $\pi^{-1}(A)$.
\end{theorem}

\begin{remark}
The condition in the conclusion of the theorem that $(U \times V)\cdot(\Bb^{\ess})^\perp
\subseteq \pi^{-1}(A)$ can be restated as follows: whenever $x_1\in U$, $z_1\in V$, and
$z_2\in\T^k$ satisfy $(z_1)^h = (z_2)^h$ for every $h \in \coc(\J_x) \cap
\coc(\Bb^{\ess}_{x_1})$, we have $(x_1,z_2)\in \pi^{-1}(A)$.
\end{remark}

\begin{proof}[Proof of \cref{thm:Aopen}]
We prove the contrapositive; that is, we consider a point $(x, z) \in X \times \T^k$ and
a sequence $(x_i, z_i, z_i')_{i\in \N} \in X\times V\times \T^k$ satisfying
\begin{enumerate}[label=(\roman*)]
  \item \label{item:convergence} $x_i \to x$ and $z_i \to z$ as $i\to \infty$,
  \item \label{item:characters} for every $i\in \N$, we have $(z_i)^h = (z_i')^h$ for
      all $h \in \coc(\bigcup \Bb \cap \Iess_{x_i})$, and
  \item \label{item:notinA} for every $i \in \N$, we have $(x_i, z_i') \notin
      \pi^{-1}(A)$,
\end{enumerate}
and we prove that $(x, z) \not \in \pi^{-1}(A)$.

Since $A$ is open in the hull-kernel topology, there is an ideal $I \subset C^*(G_T)$
such that $A = \{ P \in \Prim(C^*(G_T)) : I\not\subset P \}$. Condition~\ref{item:notinA}
implies that each $(x_i, z'_i) \not\in \pi^{-1}(A)$, and hence that each $\ker(\pi_{(x_i,
z'_i)}) \not\in A$; so
\begin{equation}\label{eq:I vs pixizi}
    I \subseteq \ker(\pi_{(x_i, z'_i)})\quad\text{ for all $i$}.
\end{equation}
Let $H$ be the subgroup of $\Z^k$ generated by
\begin{equation}\label{eq:H-generators}\textstyle
\bigcup^\infty_{i=1} \Big(\bigcap^\infty_{j=i} c(\Bb^{\ess}_{x_i})\Big)
    = \{h \in \coc(\J_x) : (x_i, h, x_i) \in B_{(x,h,x)} \cap \Iess\text{ for large }i\}.
\end{equation}
Since subgroups of $\Z^k$ are finitely generated, by discarding finitely many terms in
the sequence $(x_i, z_i, z'_i)_i$ and relabelling, we may assume that
\begin{equation}\label{eq:H in Iessi}
    H \subseteq \coc(\Bb^{\ess}_{x_i}) \subseteq \coc(\Iess_{x_i})\text{ for all $i$.}
\end{equation}

Let $\pi_{(x, z)}^H : C^*(G_T) \to \Bb(\ell^2([x]) \otimes \ell^2(\Z^k/H))$ be the
representation of \cref{prp:H-representations}. We will show that
\begin{equation} \label{eq:inclusion}
  I \subset \ker(\pi_{(x,z)}^H) \subset \ker(\pi_{(x,z)}),
\end{equation}
giving $(x,z) \notin \pi^{-1}(A)$ as required. Since $\ker(\pi_{(x, z)}) =
\ker(\pi^{\Z^k}_{(x, z)})$ (see Remark~\ref{rmk:H=Zk}), the second inclusion
in~\eqref{eq:inclusion} follows from Proposition~\ref{prp:H-representations}, so we just
have to establish that $I \subset \ker(\pi_{(x,z)}^H)$.

By the final statement of Lemma~\cref{lem:easier zero check} it suffices to prove that
\begin{equation}\label{eq}
  \langle e_{[(x,0,x)]_H},\pi^H_{(x,z)}(f) e_{[(x,0,x)]_H}\rangle = 0,
\end{equation}
for all $f\in I$.

Fix $f\in I$ and $\varepsilon > 0$, and choose $g\in C_c(G_T)$ such that $\|f - g\| <
\varepsilon / 3$. Since $g$ has compact support, there is a finite subset $F \subset
\Z^k$ such that
\begin{equation*}
  \supp(g) \subset \bigcup_{h\in F} \coc^{-1}(h).
\end{equation*}
We have
\begin{equation} \label{inn}
  \langle e_{[(x,0,x)]_H},\pi^H_{(x,z)}(g) e_{[(x,0,x)]_H}\rangle = \sum_{h\in H\cap F} z^h g(x, h, x).
\end{equation}
By \labelcref{item:convergence}, there exists $i \in \N$ such that
\begin{equation}\label{est}
  \left|\sum_{h\in H\cap F} z^h g(x, h, x) - \sum_{h\in H\cap F} (z_j)^h g(x_j, h, x_j)\right| < \varepsilon / 3
\end{equation}
for all $j \ge i$.

We now make a few observations. The set
\[
  \{h\in\Z^k : (z_i)^h = (z_i')^h\text{ and } (x_i,h,x_i)\in \Iess \text{ for large }i\}
\]
is a subgroup of $\Z^k$ that contains $H$ by~\labelcref{item:characters}. After possibly
discarding finitely many terms of the sequence $(x_i, z_i, z'_i)$ and re-indexing, we may
therefore assume that $(z_i)^h = (z_i')^h$ and $(x_i,h,x_i)\in \Iess$ for all $i\in \N$
and all $h \in H \cap F$.

Since $F$ is finite, by passing to a subsequence of the $(x_i, z_i, z'_i)_i$ and
re-indexing, we may further assume that for every $h'\in F$ either $(x_i, h',x_i) \in
\Iess$ for all $i \in \N$ or $(x_i, h', x_i) \not\in \Iess$ for all $i \in \N$.

Suppose that $h'\in F$ satisfies $(x_i,h',x_i)\in \Iess$ for all $i\in \N$ and that there
is a subsequence $(x_{i_j})_{j\in\N}$ of $(x_i)_{i\in\N}$ such that $(x_{i_j}, h',
x_{i_j})\in\supp(g)$ for all $j$. Since $\supp(g)$ is compact, every subsequence of
$(x_{i_j}, h', x_{i_j})_j$ has a convergent subsequence, and since $x_i \to x$, the limit
is $(x, h', x)$. Hence $(x_{i_j}, h', x_{i_j}) \to (x, h', x)$. In particular,
$h'\in\J_x$ and $(x_{i_j}, h', x_{i_j})\in \bigcup\Bb$ for large $j$, so $h'\in H$.

By the preceding paragraph, if $h'\in F$ satisfies $(x_i,h',x_i)\in \Iess$ for all $i$
but $h' \not\in H$, then $g(x_i,h,x_i) = 0$ for large $i$. So for each such $h'$, by
discarding finitely many terms of the sequence $(x_i, z_i, z'_i)$ and relabelling again,
we may assume that $g(x_i,h,x_i) = 0$ for all $i\in\N$, whenever $h'\in F\setminus H$
satisfies $(x_i,h',x_i)\in \Iess$ for all $i$.

For each $i \in \N$, let $E_i$ be the conditional expectation on $C^*(G_T)$ satisfying
\[
  E_i(f')((w, p, y)) = \int_{\coc(\Iess_{x_i})^\perp} z^p f((w,p,y))\,dz
\]
for $f'\in C^*(G_T)$ and $(w,p,y)\in G_T$.

Since $(z_i)^h = (z_i')^h$ and $(x_i,h,x_i)\in \Iess$ for $h \in H \cap F$, and
$g(x_i,h,x_i)=0$ for $h\in \{h'\in F: (x_i,h',x_i)\in \Iess~\textrm{for all}~i\in
\N\}\setminus H$,
\begin{align}
\sum_{h\in H\cap F}(z_i)^h g(x_i,h,x_i)
    &=\sum_{h\in H\cap F}(z'_i)^h g(x_i,h,x_i)\nonumber\\
    &=\sum_{h\in F \cap \coc(\Iess_{x_i})}(z'_i)^h g(x_i,h,x_i)\nonumber\\
    &=\langle \delta_{x_i},\pi_{(x_i,z'_i)}(E_i(g))\delta_{x_i}\rangle.\label{eqq}
\end{align}

Since $H \subseteq \Iess_{x_i}$ for all $i$ (see~\eqref{eq:H in Iessi}),
Lemma~\ref{lem:averaging and kernels} shows that $E_i(\ker(\pi^H_{(x_i,z_i')}))\subseteq
\ker(\pi^H_{(x_i,z_i')})$ for all $i$. So it follows from~\eqref{eq:I vs pixizi} that
that $E_i(f)\in \ker(\pi^H_{(x_i,z_i')})$ for all $i$. Since $\|f - g\| < \varepsilon /
3$ and each $E_i$ is a contraction, we deduce that
\[
\left|\langle e_{x_i},\pi^H_{(x_i,z'_i)}(E_i(g))e_{x_i}\rangle\right| < \varepsilon / 3.
\]
Combining this with \labelcref{inn}, \labelcref{est}, and \labelcref{eqq}, we conclude that
\[
  \left|\langle e_{[(x,0,x)]_H},\pi^H_{(x,z)}(g)e_{[(x,0,x)]_H}\rangle\right| < 2 \varepsilon / 3,
\]
and since $\|f - g\| < \varepsilon / 3$ it follows that
\begin{equation*}
  \left|\langle e_{[(x,0,x)]_H},\pi^H_{(x,z)}(f)e_{[(x,0,x)]_H}\rangle\right| < \varepsilon.
\end{equation*}
As $\varepsilon > 0$ was arbitrary this proves the claim~\labelcref{eq}.
\end{proof}

This result allows us to infer that the map $\pi$ is continuous. Elementary examples
(such as the dumbbell graph---see \cref{sec:dumbbell}) show that it is not typically
open.

\begin{corollary} \label{cor:pi-continuous}
The map $\pi\colon X\times \T^k \to \Prim(C^*(G_T))$ of \cref{ntn:pi} is
continuous.
\end{corollary}
\begin{proof}
Pick an open set $A\subset \Prim(C^*(G_T))$ and take $(x,z)\in \pi^{-1}(A)$. Since $G_T$
is \'etale, there exists a collection $\Bb = (B_\alpha)_{\alpha \in \J_x}$ of open
bisections such that $\alpha \in B_\alpha \subset \coc^{-1}(\coc(\alpha))$ for all
$\alpha \in \J_x$. \cref{thm:Aopen} implies that there are open sets $U\subset X$ and
$V\subset \T^k$ containing $x$ and $z$, respectively, such that $(U \times V) \cdot
(\Bb^{\ess})^\perp \subseteq \pi^{-1}(A)$. Since $(\Bb^{\ess})^\perp$ is a bundle over
$s(\Bb^{\ess})$ of subgroups of $\T^k$ it contains $s(\Bb^{\ess}) \times \{1\}$. Since
$B_0$ is a neighbourhood of $(x, 0, x)$, it contains an open neighbourhood $(x, 0, x) \in
U' \subseteq G_T^{(0)} \subseteq \Iess$, so $U' \times \{1\} \subseteq s(\Bb^{\ess})
\times \{1\}$. In particular $(U \cap U') \times V = (U \times V)\cdot(U' \times \{1\})$
is an open subset of $X \times \T^k$ and we have $(x, z) \in (U \cap U') \times V
\subseteq \pi^{-1}(A)$. Hence $\pi^{-1}(A)$ is open in $X\times \T^k$, and therefore
$\pi$ is continuous.
\end{proof}

\section{The sandwiching lemma for Deaconu--Renault groupoids}\label{sec:sandwiching vs SW}

In \cite[Lemma 3.3]{BriCarSim:sandwich}, we proved that for any \'etale groupoid $G$ and
any ideal $I \subseteq C^*(G)$, the set
\[
  U = \{ x\in G^{(0)} : j(f)(x) \neq 0 \textrm{ for some } f\in I\cap C_0(G^{(0)}) \}
\]
is the unique smallest open invariant set such that $I \subset I_U = C^*(G|_U)$, and
\[
  V = \{ x\in G^{(0)} : j(a)(x) \neq 0 \textrm{ for some } a\in I\}
\]
is the unique largest open invariant set such that $C^*(G|_V) = I_V \subset
I$. We refer to $U$ and $V$ as the sandwich sets related to $I$.

In this section, we identify the sandwich sets for an ideal $I$ of the
$C^*$-algebra of a Deaconu--Renault groupoid $G_T$, and relate them to the
open set $\{(x,z) \in X \times \T^k : I \not\subseteq \ker(\pi_{(x,z)})\}$
corresponding to a representation $\pi_{(x,z)}$ as in \cref{ntn:pi}. This
serves to relate the sandwiching lemma to Katsura's results for
singly-generated dynamical systems in \cref{sec:examples}.

We first need to know that if the groupoid $G_T$ admits a unit $x$ with dense
orbit, then the direct sum of the representations $\pi_{(x, z)}$ as $z$
ranges over $\T^k$ is faithful.

\begin{lemma}\label{lem:compare kernels}
Let $X$ be a second-countable locally compact Hausdorff space and suppose
that $T \colon \N^k \curvearrowright X$ is an action by local homeomorphisms.
For $x \in X$ and $z \in \T^k$ let $\pi_{(x,z)}$ be as in \cref{ntn:pi}.
Suppose that $x \in X$ satisfies $\overline{[x]} = X$. Then $\bigoplus_{z \in
\T^k} \pi_{(x,z)}$ is a faithful representation of $C^*(G_T)$. In particular,
writing $\lambda_x$ for the regular representation of $C^*(G_T)$ on
$\ell^2((G_T)_x)$, we have $\ker\big(\bigoplus_{z \in \T^k} \pi_{(x,z)}\big)
= \ker(\lambda_x)$.
\end{lemma}

\begin{proof}
We identify the groupoid $G_T$ with the groupoid of the topological higher-rank graph
$\Lambda$ defined by $\Lambda^n = X \times \{n\}$ for all $n$, range and sources maps
given by $s(x,n) = (T^n(x), 0)$ and $r(x, n) = (x , 0)$ and factorisation rules
$(x,m)(T^m(x), n) = (x, m+n) = (x, n)(T^n(x), m)$. The isomorphism $C^*(G_T) \cong
C^*(\Lambda)$ induced by this identification carries $C_0(G_T^{(0)})$ to
$C_0(\Lambda^0)$. Hence, by \cite[Corollary~5.21]{CarLarSimVit}, it suffices to show that
$\bigoplus_{z \in \T^k} \pi_{(x,z)}$ is faithful on $C_0(X)$ and that there is an action
$\beta$ of $\T^k$ on $\bigoplus_{z \in \T^k} \ell^2([x])$ such that $\beta_w \circ
(\bigoplus_{z \in \T^k} \pi_{(x,z)}) = (\bigoplus_{z \in \T^k} \pi_{(x,z)}) \circ \gamma_w$
for all $w \in \T^k$.

For the first statement, observe that since $\overline{[x]} = X$, if $f \in C_0(X)$ is
nonzero, then there exists $y \in [x]$ such that $f(y) \not= 0$. Thus, for any $z$ we
have $\langle e_y, \pi_{(x,z)}(f) e_y\rangle = f(y) \not= 0$. So $\bigoplus_z \pi_{(x,z)}(f)
\not= 0$.

For the second, to keep notation straight, identify $\bigoplus_{z \in \T^k} \ell^2([x])$
with $\ell^2([x] \times \T^k)$ so that the copy of $\ell^2([x])$ corresponding to $z \in
\T^k$ is identified with $\ell^2([x] \times \{z\}) \subseteq \ell^2([x] \times \T^k)$.

For each $z \in \T^k$ write $\Hh_z$ for the summand of $\bigoplus_{z \in \T^k}
\ell^2([x])$ corresponding to $z$, and denote the canonical orthonormal basis for $\Hh_z$
by $\{e^z_y : y \in [x]\}$. For each $w \in\T^k$, let $U_w : \bigoplus_z \Hh_z \to
\bigoplus_z \Hh_z$ be the unitary given by $U_w e^z_y = e^{wz}_y$. By definition,
$\pi_{(x, z)} = \pi_{(x,1)} \circ \gamma_z$, regarded as representations on $\ell^2([x])$.
So regarding $\pi_{(x,z)}$ as a representation on $\Hh_z$ and $\pi_{(x, 1)}$ as a
representation on $\Hh_1$, we have $\pi_{(x, z)} = \Ad_{U_z} \circ \pi_{(x, 1)} \circ
\gamma_z$. Hence
\begin{align*}
  \pi_{(x, z)} \circ \gamma_w
  &= \Ad_{U_z} \circ \pi_{(x, 1)}\circ \gamma_z \circ \gamma_w\\
  &= \Ad_{U_z} \circ \Ad_{U^*_{wz}} \circ \pi_{(x, wz)}\\
  &= \Ad_{U^*_w} \circ \pi_{(x, z)}.
\end{align*}
So $\beta_w = \Ad_{U^*_w}$ gives the desired action.
\end{proof}

Now we characterise the sandwich sets for an ideal in $C^*(G_T)$.

\begin{proposition}\label{prp:U-V for DR}
Let $X$ be a second-countable locally compact Hausdorff space and suppose
that $T \colon \N^k \curvearrowright X$ is an action by local homeomorphisms.
For $x \in X$ and $z \in \T^k$, let $\pi_{(x,z)}$ be as in \cref{ntn:pi}. Let
$I$ be an ideal of $C^*(G_T)$, and let $U$ and $V$ be the sandwich sets
related to $I$. The set
\begin{equation} \label{eq:setW}
	W = \{(x, z) \in X \times \T^k : I \not\subseteq \ker(\pi_{(x,z)})\}
\end{equation}
is open and
\[
X\setminus V = \{x \in X : (\{x\} \times \T^k) \cap W = \varnothing\}
	\quad\text{ and }\quad
U = \{x \in X : \{x\} \times \T^k \subseteq W\}.
\]
\end{proposition}

\begin{proof}
Since $I$ is an ideal in $C^*(G_T)$, the set $\{ P\in \Prim C^*(G_T) : I\subset P\}$ is open,
and $W$ is the preimage of this open set under $\pi\colon X\times \T^k \to \Prim C^*(G_T)$
which is continuous by \cref{cor:pi-continuous}.

Observe that if $f \in C_c(G_T|_V)$, $y \in X \setminus V$, and $z\in \T^k$, then we have
$\pi_{(y,z)}(f) = 0$, so by continuity, we have $I_V \subseteq \ker(\pi_{(y,z)})$ for all
$z$. In particular, $I \subseteq I_V \subseteq \ker(\pi_{(y,z)})$, which implies that
$(y,z) \not \in W$ for all $y \in X \setminus V$ and $z \in \T^k$. That is, $X\setminus V
\subseteq \{x \in X : (\{x\} \times \T^k) \cap W = \varnothing\}$. For the reverse
containment, let $x \in X$ and suppose that $(\{x\} \times \T^k) \cap W = \varnothing$.
Then $I \subseteq \bigcap_{z \in \T^k} \ker(\pi_{(x,z)})$, so $I \subseteq
\ker(\lambda_x)$ where $\lambda_x$ is the regular representation of $C^*(G_T)$ by
\cref{lem:compare kernels}. If $j \colon C^*(G_T) \to C_0(G_T)$ is Renault's map, then
$j(a)|_{(G_T)_x} = 0$ for all $a \in I$. In particular, $j(I)(x) = \{0\}$, so that $x
\not\in V$.

For the second statement, first observe that $C_0(U) \subseteq \ker(\pi_{(x,z)})$ if and
only if $x \not\in U$. Since $C_0(U) \subseteq I_U \subseteq I$, if $x \in U$ then $I
\not\subseteq \ker(\pi_{(x,z)})$ for all $z \in \T^k$. Hence if $x \in U$ then $\{x\} \times
\T^k \subseteq W$. Thus $U \subseteq \{x \in X : x \times \T^k \subseteq W\}$.

For the reverse containment, let $O = \{x \in X : \{x\} \times \T^k \subseteq W\}$.
Then $X\setminus O$ is the image of $(X\times \T^k)\setminus W$ under the projection map which is closed because $\T^k$ is compact,
so $O$ is open in $X$.
It is invariant by \cite[Theorem~3.2]{Sims-Williams}.
Since $U$ is the largest open invariant set such that $I_U \subseteq I$, so to see that $O \subseteq U$ it suffices to show that $I_O \subseteq I$.
Recall from \cite[Theorem~3.2]{Sims-Williams} that $\Prim(C^*(G_T)) = \{\ker(\pi_{(x,z)}) : x \in X, z \in \T^k\}$ as a set.
Since every ideal of a separable $C^*$-algebra is the intersection of the primitive ideals that contain it,
we have
\[
  I = \bigcap_{\{(x,z) : I \subseteq \ker(\pi_{(x,z)})\}} \ker(\pi_{(x,z)})
  = \bigcap_{(x,z) \in (X \times \T^k) \setminus W} \ker(\pi_{(x,z)}).
\]
Let $K = X \setminus O$ and observe that
$(X \times \T^k) \setminus W \subseteq (X \times \T^k) \setminus (O \times \T^k) = K \times \T^k$.
So
\[
  I \supset \bigcap_{(x,z) \in K \times \T^k} \ker(\pi_{(x,z)})
	= \bigcap_{x \in K} \Big(\bigcap_{z \in \T^k} \ker(\pi_{(x,z)})\Big)
	= \bigcap_{x \in K} \ker\Big(\bigoplus_{z \in \T^k}\ker(\pi_{(x,z)})\Big).
\]
For each $x$, write $\lambda_x$ for the regular representation of $C^*(G_T)$ on $\ell^2((G_T)_x)$.
Then \cref{lem:compare kernels} gives
\[
  I \supset \bigcap_{x \in K} \ker(\lambda_x)
	= \ker\Big(\bigoplus_{x \in K} \lambda_x\Big).	
\]
This is precisely the regular representation of $C^*(G_T|_K)$, and since Deaconu--Renault groupoids are amenable,
it is faithful on $C^*(G_T|_K)$.
So by, for example, \cite[Proposition~4.3.2]{SimsCRM},
we have $\ker\Big(\bigoplus_{x \in K} \lambda_x\Big) = I_{X \setminus K} = I_O$.
Therefore, $I_O \subseteq I$ as claimed.
\end{proof}

By \cite[Theorem~3.5]{BriCarSim:sandwich}, in order to understand the ideals of the $C^*$-algebras of an action $T \colon \N^k \curvearrowright X$
it suffices to describe the purely non-dynamical ideals with full support;
this is the ideals $I$ satisfying $I \cap C_0(X) = \{0\}$ and $\supp(I) = G_T$.
The following corollary describes an obstruction to the
existence of such an ideal (cf. \cite[Section~4]{BriCarSim:sandwich}).

\begin{corollary}\label{cor:obstruction}
Let $X$ be a second-countable locally compact Hausdorff space and suppose that $T \colon \N^k
\curvearrowright X$ is an action by local homeomorphisms. Suppose that $I$ is an ideal of
$C^*(G_T)$ such that $I \cap C_0(X) = \{0\}$ and $\supp(I) = G_T$. For every $x \in
X$ and every family $(B_\alpha)_{\alpha \in \J_x}$ of open bisections such that $\alpha
\in B_\alpha \subseteq \coc^{-1}(\coc(\alpha))$ for each $\alpha \in \J_x$, there is a
neighbourhood $U$ of $X$ such that $\Bb^{\ess}_y \not= \{y\}$ for all $y \in U$. In
particular, $\Iess_x(G_T) \not= \{x\}$ for all $x \in X$.
\end{corollary}
\begin{proof}
Let $W$ be as in~\labelcref{eq:setW} and let $W_x = \{z \in \T^k : (x,z) \in W\}$ for every $x\in X$.
Since $I \cap C_0(X) = \{0\}$ and $\supp(I) = G_T$, the sandwich sets for $I$ are $U = \varnothing$
and $V = X$, so each $W_x$ is a nonempty and proper subset of $\T^k$ by \cref{prp:U-V for DR}.

Fix $x \in X$, and a family $\Bb = (B_\alpha)_{\alpha \in \J_x}$ of open bisections such
that $\alpha \in B_\alpha \subseteq \coc^{-1}(\coc(\alpha))$ for each $\alpha \in \J_x$.
Since $W_x$ is nonempty, there exists $z \in \T^k$ such that $(x,z) \in W$.
Theorem~\ref{thm:Aopen} shows that there exist an open neighbourhood $U \subseteq B_x
\cap X$ of $x$ and an open neighbourhood $V \subseteq \T^k$ of $z$ such that $(U \times
V) \cdot (\Bb^{\ess})^\perp \subseteq W$. In particular, for each $y \in U$ we have
$(\Bb^{\ess}_y)^\perp \cdot z \subseteq W_y \subsetneq \T^k$. This forces
$(\Bb^{\ess}_y)^\perp \subsetneq \T^k$ and hence $\Bb^{\ess}_y \not= \{y\}$.

For the final statement, just observe that the preceding paragraph implies in particular
that $\Bb^{\ess}_x \not= \{x\}$, and since $\Bb^{\ess}_x \subseteq \Iess_x(G_T)$, the
result follows.
\end{proof}

We now obtain an obstruction ideal, in the sense of Ara--Lolk
\cite[Definition~7.11]{AraLolk} for a Deaconu--Renault groupoid.

\begin{corollary}\label{cor:obstruction ideal}
Let $X$ be a second-countable locally compact Hausdorff space and suppose that $T \colon \N^k
\curvearrowright X$ is an action by local homeomorphisms.
Let $W' = s(\Iess(G_T)\setminus X)^\circ$, and let $W$ be the set of points $x \in X$
such that for every family $(B_\alpha)_{\alpha \in \J_x}$ of open bisections such
that $\alpha \in B_\alpha \subseteq \coc^{-1}(\coc(\alpha))$ for each $\alpha \in \J_x$,
there is a neighbourhood $U$ of $X$ such that $\Bb^{\ess}_y \not= \{y\}$ for all $y \in U$.
Then $W \subseteq W'$, and if $I$ is an ideal of $C^*(G_T)$ such that $I \cap C_0(X) = \{0\}$,
then $I \subseteq I_W$.
\end{corollary}
\begin{proof}
The inclusion $W \subseteq W'$ is clear so we just need to prove the second containment.
Let $U$ and $V$ be the sandwich sets for $I$, and note that $U = \varnothing$ because $I \cap C_0(X) = \{0\}$.
Moreover, $I \subseteq I_V \cong C^*(G_T|_V)$, so we can regard $I$ as an ideal of $C^*(G_T|_V)$ with $I \cap C_0(V) = \{0\}$
and $\supp(I) = G_T|_V$.
Now \cref{cor:obstruction} implies that $V \subseteq W$.
\end{proof}

\begin{remark}
The definition of $W$ in \cref{cor:obstruction ideal} is more technical than
that of $W'$, but $W$ can be strictly smaller, so it provides a better estimate. For
example, in the instance of the dumbbell graph of \cref{eg:dumbbell}, the
essential isotropy at every unit is nontrivial, and so the set $W'$ of
\cref{cor:obstruction ideal} is all of $E^\infty$. However, as we will see in
\cref{sec:dumbbell} below, it is straightforward to construct a family $\Bb =
(B_\alpha)_{\alpha \in \J_{e^\infty}}$ of bisections such that for $n \not= 0$, we have
$B_{(e^\infty, n, e^\infty)} \cap \I(G_E) = \{(e^\infty, n, e^\infty)\}$. So the set $W$
of \cref{cor:obstruction ideal} is the (open) orbit of $g^\infty$, which is the
support of the minimal obstruction ideal described in
\cite[Definition~4.3]{BriCarSim:sandwich}.
\end{remark}

\section{Harmonious families of bisections}\label{sec:bisection-families}

Our main result requires the concept of a harmonious family of bisections based at a
unit. Recall from \cref{ntn:Jx} that for an etale groupoid $G$, we write $\J_x =
\overline{\Iess(G)}_x$. We emphasise that harmonious families of bisections are
meaningful for any \'etale groupoid but we shall only study the case of groupoids $G_T$
of actions $T\colon \N^2\curvearrowright X$ by local homeomorphisms.

\begin{definition}\label{def:Bf}
A \emph{harmonious family of bisections based at a unit $x\in X$} is a collection $\Bb = (B_{\alpha})_{\alpha\in \J_x}$ of open bisections of $G_T$
satisfying the following conditions:
\begin{enumerate}[label=(\roman*)]
  \item\label{it:Bf units} $B_{x} \subseteq X$;
  \item\label{it:Bf homogeneous} $\alpha \in B_\alpha \subseteq \coc^{-1}(\coc(\alpha))$ for all $\alpha \in
      \J_x$;
  \item\label{it:Bf inverses} $(B_\alpha \cap \Iess)^{-1} = B_{\alpha^{-1}}
      \cap \Iess$ for all $\alpha$;
  \item\label{it:Bf right inv} $B_\alpha (B_\beta \cap \Iess) \subseteq
      B_{\alpha\beta}$ for all $\alpha,\beta \in \J_x$; and
  \item\label{it:compact.open} for each $\alpha \in \J_x$ there exists a compact set
      $K_\alpha \subseteq G^{(0)}$ such that $s(B_\alpha) = K_\alpha \cap B_x$.
\end{enumerate}
In particular, $\bigcup \Bb \subset s^{-1}(B_x)$. We say a unit $x \in X$ \emph{admits a
harmonious family of bisections} if there exist a harmonious family of bisections based
at $x$ and we frequently just say that $(B_{\alpha})_{\alpha \in \J_x}$ \emph{is a
harmonious family of bisections}; it is implicit that it is based at $x$.
\end{definition}

\begin{remark}\label{rmk:H_B groups}
If $\Bb = (B_\alpha)_{\alpha \in \J_x}$ is a harmonious family of bisections in $G_T$,
then $\Bb^{\ess}_y$ is a subgroup of $(G_T)_y$ for every $y \in B_x \subset X$, and
determines a subgroup
\[
  H_\Bb(y) \coloneqq \coc(\Bb^{\ess}_y) \le \coc(\J_x) \le \Z^k.
\]
We have
\[
    H_\Bb(y) = \{n \in \Z^k : (x, n, x) \in \J_x \text{ and } (y, n, y) \in B_{(x,n,x)} \cap \Iess\}.
\]
\end{remark}

\begin{remark}\label{rem:Bf}
Condition~\ref{it:compact.open} in \cref{def:Bf} is a technical condition that captures
two special cases:
\begin{enumerate}
\item if there are bisections $B_\alpha$ satisfying \ref{it:Bf units}--\ref{it:Bf right
    inv} that are all compact open sets, then Condition~\ref{it:compact.open} is
    satisfied since we can take $K_\alpha = s(B_\alpha)$ for all $\alpha$; and
\item if there are bisections $B_\alpha$ satisfying \ref{it:Bf units}--\ref{it:Bf right
    inv} such that $s(B_\alpha) = B_x$ for all $\alpha \in \J_x$ (such as when $T$ is
    an action by homeomorphisms), then for any precompact open $U$ such that $x \in U
    \subseteq \overline{U}  \subseteq B_x$, the sets $(B_\alpha U)_\alpha$ also satisfy
    \ref{it:Bf units}--\ref{it:Bf right inv}, and they satisfy~\ref{it:compact.open}
    since we can take $K_\alpha = \overline{U}$ for all $\alpha$.
\end{enumerate}
The condition is used in the proof of \cref{prp:fP} to show that the functions
$1_{B_\alpha} \phi$ are continuous.
\end{remark}

\begin{example}\label{eg:isolated or strongly effective}
  \begin{enumerate}
    \item Any isolated unit admits harmonious families of bisections.
    \item If $G_T$ is a Deaconu--Renault groupoid that is strongly effective, then
        $\J_x = \{x\}$ for every unit $x$, and $G_T$ admits a harmonious family of
        bisections.
  \end{enumerate}
\end{example}

For the following, recall that a subset $C$ of the groupoid $G_T$ is homogeneous,
if $\coc(C)$ is a singleton.

\begin{lemma}\label{lem:bisection family transport}
Let $X$ be a second-countable locally compact Hausdorff space and suppose $T\colon \N^k
\curvearrowright X$ is an action by local homeomorphisms. Suppose that $\Bb =
(B_\alpha)_{\alpha \in \J_x}$ is a harmonious family of bisections based at $x \in X$,
and that $C$ is an open and homogeneous bisection such that $x \in s(C)$. Let $\gamma \in
C$ be the unique element such that $s(\gamma) = x$. Then $(C B_\alpha C^{-1})_{\alpha \in
\J_x}$ is a harmonious family of bisections at $r(\gamma)$. In particular, if $x\in X$
admits a harmonious family of bisections then every element in the orbit of $x$ admits a
harmonious family of bisections.
\end{lemma}
\begin{proof}
For any $\alpha \in \J_x$, we have
\[
  \gamma\alpha\gamma^{-1} \in C B_h(x) C^{-1}.
\]
Since $\Iess(G_T)$ is a normal subgroupoid by \cref{lem:Icirc-Iess normal}
the map $\alpha \mapsto \gamma \alpha \gamma^{-1}$ is a group isomorphism
between $\Iess(G_T)_x$ and $\Iess(G_T)_z$, so $(C B_\alpha C^{-1})_{\alpha \in
\J_x}$ is a family of bisections indexed by $\J_z$.

For each $\beta \in \J_z$, let $C_\beta \coloneqq C B_{\gamma^{-1} \beta \gamma} C^{-1}$ and
write $\mathcal{C} = (C_\beta)_{\beta \in \J_z} = (C B_\alpha C^{-1})_{\alpha \in \J_x}$.

The set $C B_x C^{-1}$ is contained in $X$ and contains $z \coloneqq r(\gamma)$, which is
\cref{def:Bf}\ref{it:Bf units}.

Fix $\beta \in \J_z$ and let $\alpha = \gamma^{-1}\beta\gamma$. Then $\alpha \in \J_x$, and
we have $\beta = \gamma \alpha \gamma^{-1} \in C_\beta$.  For any $\zeta \in C_\beta =
C B_\alpha C^{-1}$, there exist $\eta \in B_\alpha$ and $\rho,\tau \in C$ such that
$\zeta = \rho \eta \tau^{-1}$. Since $C$ is homogeneous, we have $\coc(\rho) =
\coc(\tau)$, and so $\coc(\zeta) = \coc(\rho) + \coc(\eta) - \coc(\tau) = \coc(\eta)
\in \coc(B_\alpha) = \{\coc(\alpha)\}$. Applying this to $\zeta = \beta$ we see that
$\coc(\beta) = \coc(\alpha)$ as well, and so $C_\beta \subseteq
\coc^{-1}(\coc(\beta))$, giving \cref{def:Bf}\ref{it:Bf homogeneous}.

Next fix $\alpha \in \J_x$ and $\eta\in C B_\alpha C^{-1}\cap \I^{\textrm{ess}}$. Then
$\eta^{-1} \in \Iess$ by \cref{lem:Icirc-Iess normal}. We need to verify that
$\eta^{-1}\in C_{(\gamma\alpha\gamma^{-1})^{-1}}$. Write $\eta = \xi \theta \xi^{-1}$.
Then by \cref{lem:Icirc-Iess normal} again, $\theta \in \Iess$. We have $\theta^{-1} \in
B_{\alpha^{-1}}$ by definition of a harmonious family of bisections, so we see that
$\eta^{-1} = \xi \theta^{-1} \xi^{-1} \in C B_{\alpha^{-1}} C^{-1}$, which is
\cref{def:Bf}\ref{it:Bf inverses}.

Since $C^{-1}C = s(C)$, for $\beta,\eta \in \J_z$, we have $C_\beta C_{\eta} = C
B_{\gamma^{-1}\beta\gamma} C^{-1}C C_{\gamma^{-1}\eta\gamma} C^{-1} \subseteq C
B_{\gamma^{-1}\beta\gamma\gamma^{-1}\eta\gamma} C^{-1} = B_{\gamma^{-1}\beta\eta\gamma} =
C_{\beta\eta}$, giving \cref{def:Bf}\ref{it:Bf right inv}.

Choose compact subsets $K_\alpha \in X$ satisfying $s(B_\alpha) = K_\alpha \cap B_x$ for
every $\alpha\in \J_x$. Let $K'_{\gamma\alpha\gamma^{-1}} = C K_\alpha C^{-1}$ for each
$\alpha\in \J_x$. Then
\[
  C K_\alpha(x) C^{-1} = r\circ (s|_C)^{-1}(K_\alpha(x)).
\]
Since $r\circ (s|_C)^{-1}$ is the canonical partial homeomorphism associated to the
bisection $C$, we see that $K'_\alpha(z)$ is a compact subset of $X$. Let $\beta \coloneqq
\gamma\alpha\gamma^{-1}$. We claim that $s(C_{\beta}) = K'_{\beta} \cap C_z$.

Computation shows that $s(B_\beta) = r\circ (s|_C)^{-1}(s(B_\alpha))$ and $C B_x C^{-1} =
r\circ (s|_C)^{-1}(B_x)$, so
\[
  K'_\beta C_z = r\circ (s|_C)^{-1}(K_\alpha) \cap r\circ (s|_C)^{-1}(B_x) = r\circ (s|_C)^{-1}(s(B_\alpha)) = s(C_\beta).
\]
This gives \cref{def:Bf}\ref{it:compact.open}.

For the final statement, suppose that $x$ admits a harmonious family of bisections $\Bb$,
fix $z \in [x]$ and $\gamma \in (G_T)^z_x$. Since $G_T$ is \'etale, there is an open
bisection $C$ containing $\gamma$. Now the preceding paragraphs show that $(C B_\alpha
C^{-1})_{\alpha \in \J_x}$ is a harmonious family of bisections based at $z$.
\end{proof}

The main obstacle in applying our strongest results is establishing the existence of
harmonious families of bisections at all points for a given action of $\N^k$ by local
homeomorphisms. Next we outline some techniques for constructing harmonious families of
bisections that apply to a large class of examples among them actions by commuting
homeomorphisms, a single local homeomorphism, and $2$-graphs in \cref{sec:examples}. We
leave it as an open problem to determine if every Deaconu--Renault groupoid admits
harmonious families of bisections.

The first existence result applies e.g. to actions of commuting homeomorphisms, cf.~\cref{subsec:cps}.

\begin{lemma} \label{lem:commuting-bisections}
Let $X$ be a second-countable locally compact Hausdorff spaces and suppose that $T\colon \N^k\curvearrowright X$ is an action by local homeomorphisms.
Let $x\in X$.
Suppose there are open bisections $(B_\alpha)_{\alpha \in \J_x}$ satisfying
$\alpha \in B_\alpha \subseteq \coc^{-1}(\coc(\alpha))$ and $B_\alpha B_\beta = B_{\alpha\beta}$ for all $\alpha,\beta \in \J_x$.
Then $x$ admits a harmonious family of bisections.
\end{lemma}

\begin{proof}
We have $B_\alpha B_{\alpha^{-1}} = B_x = B_{\alpha^{-1}}B_\alpha$ for each $\alpha$,
and this implies that $B_x \subseteq r(B_\alpha)$.
In particular, $B_x \subseteq X$ and this is \labelcref{it:Bf units} in \cref{def:Bf}.
By assumption, we have $\alpha \in B_\alpha \subseteq \coc^{-1}(\coc(\alpha))$ and this is \labelcref{it:Bf homogeneous}.
Since $B_x B_\alpha = B_{x\alpha} = B_\alpha$, we have $r(B_\alpha) \subseteq B_x$,
and combined with the first sentence this gives equality.
Hence the $B_\alpha$ are a subset of the group $\Bb(X)$ of open bisections with range and source equal to $B_x$,
and they constitute the range of the homomorphism $\alpha \mapsto B_\alpha$.
So $\{B_\alpha : \alpha \in \J_x\}$ is a subgroup of $\Bb(X)$.
Since the inverse operation in $\Bb(X)$ is implemented by pointwise inverses in $G$,
we obtain $B_{\alpha^{-1}} = B_\alpha^{-1}$,
and we already have $B_\alpha B_\beta = B_{\alpha\beta}$ for all $\alpha,\beta$, so \labelcref{it:Bf inverses,it:Bf right inv} are trivial.
The final condition \labelcref{it:compact.open} is automatic by \cref{rem:Bf}(2)
\end{proof}

\begin{corollary}\label{cor:finitely generated Bf}
Let $X$ be a second-countable locally compact Hausdorff space and suppose that $T\colon \N^k\curvearrowright X$ is an action by local homeomorphisms.
Let $x \in X$.
Suppose that $M \subseteq \Z^k$ satisfies $|M| = \operatorname{rank}(\coc(\J_x))$ and $\coc(\J_x) = \Span_\Z(M)$.
Suppose that there is an open neighbourhood $B_x \subseteq X$ of $x$ and a collection of open bisections $(B_\alpha)_{\alpha\in M}$ such that
\begin{enumerate}[label=(\arabic*)]
\item each $\alpha \in B_\alpha \subseteq \coc^{-1}(\coc(\alpha))$,
\item\label{it:commuting} $B_\alpha B_\beta = B_\beta B_\alpha$ for all
    $\alpha,\beta \in M \cup -M$, and
\item\label{it:source-range} $r(B_\alpha) = s(B_\alpha) = B_x$ for all
    $\alpha \in M$.
\end{enumerate}
Then $x$ admits a harmonious family of bisections.
\end{corollary}
\begin{proof}
The hypotheses guarantee that there is a well-defined map $\gamma \mapsto
B_\gamma$ from the group generated by $M$ to the collection of open
bisections of $G_T$ such that for every function $n : M \to \Z$, we have
\[
B_{\prod_{\alpha \in M} \alpha^{n(\alpha)}} = \prod_{\alpha \in M_+(x)}B_{\alpha}^{n(\alpha)}.
\]
An induction using \ref{it:commuting}~and~\ref{it:source-range} shows that these
$B_\gamma$ satisfy the hypotheses of \cref{lem:commuting-bisections}, and the result
follows.
\end{proof}

Our next existence result will be applied to the case of $2$-graphs in~\cref{sec:rank2}.

\begin{lemma}\label{lem:relatively commuting}
Let $X$ be a second-countable locally compact Hausdorff space and suppose that $T \colon \N^k \curvearrowright X$ is an action by local homeomorphisms.
Let $x \in X$.
Suppose that $M \subseteq \J_x$ generates $\J_x$ as a group and that $|M| = \operatorname{rank}(\J_x)$.
Suppose that $\{B_\alpha : \alpha \in M\}$ is a family of compact open bisections satisfying
\begin{enumerate}
\item $\alpha \in B_\alpha \subseteq c^{-1}(\coc(\alpha))$ for all $\alpha \in M$,
\item $B_\alpha B_\beta = B_\beta B_\alpha$ for all $\alpha,\beta \in M$,
    and
\item $B_\alpha B^{-1}_\beta \subseteq B^{-1}_\beta B_\alpha$ for all
    $\alpha,\beta \in M$.
\end{enumerate}
For each $\gamma \in \J_x$ and $\alpha \in M$, let $m_\alpha(\gamma)$ be the
integers such that $\gamma = \prod_{\alpha \in M} \alpha^{m_\alpha(\gamma)}$.
Taking the convention that the empty product is equal to $X$, the sets
\[
B_\gamma \coloneqq \Big(\prod_{m_\alpha(\gamma) < 0} B_\alpha^{m_\alpha(\gamma)} \Big)
                \Big(\prod_{m_\alpha(\gamma) > 0} B_\alpha^{m_\alpha(\gamma)} \Big),
\]
indexed by $\gamma \in \J_x$ constitute a harmonious family of bisections at $x$.
\end{lemma}
\begin{proof}
The bisections $(B_\alpha)_\alpha$ pairwise commute, so the formula for $B_\gamma$ is well-defined.
We must verify the five conditions of \cref{def:Bf}. We have $B_x
\subseteq X$, which is~\ref{it:Bf units}, by definition. Since each $\alpha \in B_\alpha$,
we have $\gamma \in B_\gamma$ for all $\gamma$, and since each $B_\alpha
\subseteq c^{-1}(\alpha)$ and $c$ is a homomorphism on $\J_x$ we have
$B_\gamma \subseteq c^{-1}(\gamma)$ for all $\gamma$; this gives~\ref{it:Bf homogeneous}.

To see~\ref{it:Bf inverses}, note that $m_{\alpha(\gamma^{-1})} = -m_\alpha(\gamma)$, and so
\begin{align*}
B_{\gamma^{-1}}
   &= \prod \Big(\prod_{m_\alpha(\gamma^{-1}) < 0} B_\alpha^{m_\alpha(\gamma^{-1})} \Big)
                \Big(\prod_{m_\alpha(\gamma^{-1}) > 0} B_\alpha^{m_\alpha(\gamma^{-1})} \Big)\\
   &= \prod \Big(\prod_{m_\alpha(\gamma) > 0} B_\alpha^{m_\alpha(\gamma)} \Big)^{-1}
                \Big(\prod_{m_\alpha(\gamma) < 0} B_\alpha^{m_\alpha(\gamma^{-1})} \Big)^{-1}
   = B_{\gamma}^{-1}
\end{align*}
for all $\gamma \in \J_x$. Since $\Iess$ is self-inverse, it follows that
$(B_\gamma \cap \Iess)^{-1} = B_{\alpha}^{-1} \cap \Iess$ for all $\gamma \in
\J_x$, which gives~\ref{it:Bf inverses}.

Finally, for~\ref{it:Bf right inv}, First note that a simple induction using that $B_\alpha
B^{-1}_\beta \subseteq B^{-1}_\beta B_\alpha$ for all $\alpha,\beta \in M$
shows that for any pair of functions $p,q : M \to \N$ we have
\[
\Big(\prod_{\alpha \in M} B_\alpha^{p(\alpha)}\Big)\Big(\prod_{\alpha \in M} B_\alpha^{-q(\alpha)}\Big)
    \subseteq \Big(\prod_{\alpha \in M} B_\alpha^{-q(\alpha)}\Big) \Big(\prod_{\alpha \in M} B_\alpha^{p(\alpha)}\Big).
\]
We deduce that for $\gamma,\delta \in \J_x$, we have
\begin{align*}
B_\gamma B_\delta
 &= \Big(\prod_{m_\alpha(\gamma) < 0} B_\alpha^{m_\alpha(\gamma)} \Big)
                \Big(\prod_{m_\alpha(\gamma) > 0} B_\alpha^{m_\alpha(\gamma)} \Big)
                \Big(\prod_{m_\alpha(\delta) < 0} B_\alpha^{m_\alpha(\delta)} \Big)
                \Big(\prod_{m_\alpha(\delta) > 0} B_\alpha^{m_\alpha(\delta)} \Big)\\
 &\subseteq \Big(\prod_{m_\alpha(\gamma) < 0} B_\alpha^{m_\alpha(\gamma)} \Big)
                \Big(\prod_{m_\alpha(\delta) < 0} B_\alpha^{m_\alpha(\delta)} \Big)
                \Big(\prod_{m_\alpha(\gamma) > 0} B_\alpha^{m_\alpha(\gamma)} \Big)
                \Big(\prod_{m_\alpha(\delta) > 0} B_\alpha^{m_\alpha(\delta)} \Big)\\
 &= B_{\gamma\delta}.
\end{align*}
Hence $B_\gamma (B_\delta \cap \Iess) \subseteq B_\gamma B_\delta \subseteq
B_{\gamma\delta}$, giving~\ref{it:Bf right inv}.

Since the $B_\alpha$ are compact open, so are the $B_\gamma$, and so
condition~\ref{it:compact.open} is automatic.
\end{proof}

\begin{example}\label{eg:minimal}
If $G_T$ is a minimal Deaconu--Renault groupoid, then $\J(G_T) = \I^\circ(G_T)$ by
\cite[Proposition~2.1]{KPS-simple}. As in \cref{sec:DR prelims}, $H(x) =
\coc(\I^\circ(G_T)_x)$ is constant with respect to $x$, and writing $H$ for this group,
$\J_x = \Iess_x = \I^\circ_x = \{x\} \times H \times \{x\}$ for all $x$. So for $n \in
H$, the set $B(n) \coloneqq \{(x, n, x) : x \in X\}$ is a bisection. These $B(n)$ pairwise
commute, and each commute with any subset of $G^{(0)}_T$. So for each $x \in X$ and $n
\in H$, choosing a precompact open neighbourhood $W_x$ of $x$ and putting $B_{(x,n,x)} \coloneqq
W_x B(n)$ for each $n$ determines a harmonious family of bisections $\Bb^x =
(B_\alpha)_{\alpha \in \J_x}$.
\end{example}

To make use of results like \cref{lem:relatively commuting} or \cref{cor:finitely
generated Bf}, we need to be able to identify free generators of a given subgroup of
$\Z^k$. When working with Deaconu--Renault groupoids, it is often helpful to work with
minimal collections of generators in $\N^k$. So we prove that every rank-$k$ subgroup of
$\Z^k$ admits $k$ free generators in $\N^k$ that are minimal with respect to the usual
algebraic order. This is surely well known, but we could not find a reference. We found
the key idea behind the proof on Math StackExchange \cite{StackExLink}.

We make use of the usual lattice order $\le$ on $\Z^k$; so $m \le n$ if $n-m \in \N^k$.
For $i < k$ we also identify $\Z^i$ with the subgroup of $\Z^k$ consisting of elements
whose final $k-i$ coordinates are zero. So $\N^i \setminus \N^{i-1} = \N^{i-1} \times
(\N\setminus \{0\}) \times \{0_{k-i}\}$.

\begin{lemma}\label{lem:positive-generation}
Let $H \subseteq \Z^k$ be a subgroup of rank $k$. Then there exist $m^1,
\dots, m^k \in \N^k \cap H$ such that
\begin{enumerate}[label=(\arabic*)]
	\item\label{it:minimal} each $m^i$ is a minimal element of $\N^k \cap H \setminus \{0\}$,
	\item\label{it:mi coords} for each $i \le k$, we have $m^i \in \N^i \setminus \N^{i-1}$,
	\item\label{it:free generators} $H = \bigoplus_i \Z m^i$.
\end{enumerate}
\end{lemma}
\begin{proof}
Since the rank of $H$ is $k$, any $k$ elements that generate $H$ are free abelian
generators, so it suffices to establish that there exist $m^1, \dots, m^k \in \N^k \cap
H$ satisfying \ref{it:minimal}~and~\ref{it:mi coords} such that $H = \Span_\Z\{m^1,
\dots, m^k\}$. 	

We argue by induction on $k$. For $k = 1$, the result is trivial. So suppose inductively
that every rank-$(k-1)$ subgroup $H'$ of $\Z^{k-1}$ has such generators. Let $\pi \colon \Z^k
\to \Z$ be the homomorphism $\pi(n) = n_k$ onto the $k$th coordinate, so $\ker(\pi) =
\Z^{k-1}$. Since $\pi(H) \le \Z$, we have $\pi(H) = a \Z$ for some $a \in \N$. Fix
$\tilde{m}^k \in H$ satisfying $\pi(\tilde{m}^k) = a$.

Let $H' = H \cap \Z^{k-1}$. For any $h \in H$, we have $a \mid \pi(h)$, and so $h' = h -
\frac{\pi(h)}{a} \tilde{m}^k \in H'$, and so $h = h' + \frac{\pi(h)}{a} \tilde{m}^k \in
\Z\tilde{m}^k + H'$. So $H$ is generated by $\{\tilde{m}^k\} \cup H'$. This implies in
particular that $\operatorname{rank}(H) \le 1 + \operatorname{rank}(H')$, so
$\operatorname{rank}(H') \ge k-1$. Hence $\operatorname{rank}(H') = k-1$ because $H'
\subseteq \Z^{k-1}$. By the inductive hypothesis, there are generators $m^1, \dots,
m^{k-1}$ of $H \cap \N^{k-1}$ satisfying \ref{it:minimal}~and~\ref{it:mi coords}. In
particular, $H = \Span_\Z\{m^1, \dots, m^{k-1}, \tilde{m}^k\}$ where each $m^i$ is in $\N^i \setminus \N^{i-1}$.

Since each $m^i_i \not= 0$, there exist $a_1, \dots, a_{k-1} \ge 0$ such that $a_i m^i_i
+ \tilde{m}^k_i \ge 0$. Since each $m^i_j \ge 0$, we deduce that $\big(\sum_i a_i
m^i\big) + \tilde{m}^k \in \N^k$. Hence $C \coloneqq \big(\tilde{m}^k_i + \Span_\Z \{m^1, \dots,
m^{k-1}\}\big) \cap \N^k$ is nonempty, and therefore has a minimal element. We fix $a_1,
\dots, a_k$ such that $m^k \coloneqq \tilde{m}^k + \sum_{i < k} a_i m^i_i$ is a minimal element
of $C$. This implies that $m^i \not\le m^k$ for $i < k$. Then $m^k \in H \cap \N^k$, and
$m^k_k = \tilde{m}^k_k > 0$. We have $\tilde{m}^k = m^k + \sum _{i < k} (-a_i)m^i \in
\Span_\Z\{m^1, \dots, m^k\}$. Hence $\Span_\Z\{m^1, \dots, m^{k-1}, m^k\} = \Span_\Z\{m^1,
\dots, m^{k-1}, \tilde{m}^k\} = H$.

It remains to show that $m^1, \dots, m^k$ are minimal in $H \cap \N^k
\setminus \{0\}$. Fix $i < k$. Since $m^1, \dots, m^{k-1} \subseteq \N^{k-1}$
and $m^k_k > 0$, we have $\{p \in H \cap \N^k : p < m^i\} = \{p \in H' \cap
\N^k : p < m^i\} = \{0\}$ by the inductive hypothesis. So $m^i$ is minimal in
$H \cap \N^k \setminus \{0\}$. To see that $m^k$ is minimal, suppose that $p
\in H \cap \N^k$ satisfies $0 < p \le m^k$. If $p \in H'$ then there exists
$j < k$ such that $m^j \le p \le m^k$ which is impossible as observed in the
preceding paragraph. So $p \in H \setminus H'$. Then $0 < \pi(p) \in a\Z$ and
$\pi(p) \le a$, forcing $\pi(p) = a$. Write $p = \sum_i b_i m^i$. Since
$\pi(m^i) = 0$ for $i < k$, we have $b_k = 1$. Hence $p \in m^k +
\Span_\Z\{m^1, \dots, m^{k-1}\} \cap \N^k = \tilde{m}^k + \Span_\Z\{m^1,
\dots, m^{k-1}\} \cap \N^k = C$. Since $m^k$ is minimal in $C$, we obtain $p
= m^k$.
\end{proof}

\section{The primitive-ideal space}\label{sec:primitive-ideal-space}
With the concept of harmonious families of bisections available, we can now
state our second main theorem, which describes a family of subsets of $X
\times \T^k$ that are preimages of open subsets of $\Prim(C^*(G_T))$.

\begin{theorem} \label{thm:neighbourhood-basis}
Let $X$ be a second-countable locally compact Hausdorff space and suppose
that $T \colon \N^k \curvearrowright X$ is an action by local homeomorphisms.
For $x \in X$ and $z \in \T^k$, let $\pi_{(x,z)}$ be as in \cref{ntn:pi}. Let
$(x_0, z_0)\in X\times \T^k$ and suppose that $\Bb = (B_\alpha)_{\alpha \in
\J_{x_0}}$ is a harmonious family of bisections, and that $V\subset \T^k$ is
an open neighbourhood of $z_0$. Then the set
\begin{align} \label{eq:basic-open}
  A(\Bb,V) \coloneqq \{\ker(\pi_{(x, z)}) : x\in B_{x_0}, z\in V H_\Bb(x)^\perp\} \subseteq \Prim(C^*(G_T))
\end{align}
is an open neighbourhood of $\ker(\pi_{(x_0, z_0)})$ in $\Prim(C^*(G_T))$.
\end{theorem}

The following technical proposition is the engine-room in the proof of
\cref{thm:neighbourhood-basis}. We think of this result as a kind of
noncommutative Urysohn lemma, and we state it separately so that we can use it
later to describe generators for the ideal of $C^*(G_T)$ corresponding to a
given open subset of $X \times \T^k$ in \cref{prp:ideal generators}. We thank
Johannes Christensen and Sergiy Neshveyev for pointing out an error in the
original proof of this result, and for helpful subsequent conversations.

\begin{proposition}\label{prp:fP}
Let $X$ be a second-countable locally compact Hausdorff space and suppose
that $T \colon \N^k \curvearrowright X$ is an action by local homeomorphisms.
For $x \in X$ and $z \in \T^k$, let $\pi_{(x,z)}$ be as in Notation~\ref{ntn:pi}.
Let $(x_0, z_0) \in X\times \T^k$ and suppose $\Bb = (B_\alpha)_{\alpha \in
\J_{x_0}}$ is a harmonious family of bisections, and that $V\subset \T^k$ is
an open neighbourhood of $z_0$. For any $(x, z) \in B_{x_0} \times V$ there
exist a function $\phi \in C_c(B_{x_0}, [0,1])$ such that $\phi(x) = 1$, a
function $\psi\in C^\infty(\T^k)$ such that $\psi(z) = 1$ and
$\psi|_{\T^k \setminus V} = 0$, and an element $h_0 \in \Z^k$ such that the
$h_0$-perturbation $\psi_{h_0}$ satisfies
\[
  \sum_{h\in H_\Bb(x)} z^h \hat{\psi}_{h_0}(h) \not= 0.
\]
For any such $\phi, \psi$ and $h_0$ the series
\[
  \sum_{(x_0, h, x_0) \in \J_{x_0}} \hat{\psi}_{h_0}(h) (1_{B_{(x_0, h, x_0)}} \phi)
\]
converges to an element $f$ of $C^*(G_T)$ such that $\pi_{(x, zw)}(f) \neq 0$
for every $w \in H_\Bb(x)^\perp$ and such that $f \in Q$ for every primitive ideal
$Q \not\in \pi((B_{x_0} \times V)\cdot(\Bb^{\ess})^\perp)$.
\end{proposition}
\begin{proof}
By Urysohn's lemma there exists $\phi\in C_c(X, [0,1])$ such that $\phi(x) =
1$ and $\phi|_{X \setminus B_{x_0}} = 0$. By, for example,
\cite[Chapter~8]{Folland}, there exists $\psi\in C^\infty(\T^k)$ such that
$\psi(z) = 1$ and $\psi|_{\T^k \setminus V} = 0$.
By~\cref{lem:H-restriction-is-nonzero}, there exists $h_0\in \Z^k$ such that
\[
  \sum_{h\in H_\Bb(x)} z^h \hat{\psi}_{h_0}(h) \not= 0.
\]

For each $\alpha \in \J_{x_0}$, by definition of a harmonious family of bisections, there
is a compact subset $K_\alpha$ of $X$ such that $s(B_\alpha) = K_\alpha \cap B_{x_0}$. We
claim that this implies that $1_{B_\alpha} \phi$ is continuous. Since $B_\alpha$ is open
it suffices to show that if $(\beta_n)_n$ is a sequence in $B_\alpha$ and $\beta_n \to
\beta$ then $(1_{B_\alpha} \phi)(\beta_n) \to (1_{B_\alpha} \phi)(\beta)$ as $n\to
\infty$. Since each $\beta_n \in B_\alpha$, we have
\[
(1_{B_\alpha} \phi)(\beta_n) = 1_{B_\alpha}(\beta_n)\phi(s(\beta_n)) = \phi(s(\beta_n)),
\]
so we must show that
\[
\phi(s(\beta_n)) \to
    \begin{cases}
        \phi(s(\beta)) &\text{ if $\beta \in B_\alpha$}\\
        0 &\text{ if $\beta\not\in B_\alpha$}.
    \end{cases}
\]
Since $\phi \circ s$ is continuous, it therefore suffices to show that if $\beta\not\in
B_\alpha$ then $\phi(s(\beta)) = 0$. So suppose that $\beta \not\in B_\alpha$. Since
$K_\alpha$ is compact and the $s(\beta_n)$ belong to $K_\alpha$, we have $s(\beta) \in
K_\alpha$. We claim that  $s(\beta) \not\in s(B_\alpha)$. To see this, suppose for
contradiction that  $s(\beta) = s(\beta')$ for some $\beta' \in B_\alpha$. Since $s$
restricts to a homeomorphism $B_\alpha \to s(B_\alpha)$ and $s(\beta_n) \to s(\beta')$ we
have $\beta_n \to \beta'$. Since $\beta_n \to \beta$ by hypothesis, and since $G_T$ is
Hausdorff, this forces $\beta = \beta' $, contradicting $\beta \not\in B_\alpha$. Thus
$s(\beta) \not\in B_\alpha = K_\alpha \cap B_{x_0}$. We saw that $s(\beta) \in K_\alpha$,
so we deduce that $s(\beta) \not\in B_{x_0}$. Since $\phi$ vanishes on $X \setminus
B_{x_0}$ by construction, we then have $\phi(s(\beta)) = 0$ as required, so $1_{B_\alpha} \phi$ is continuous.

For each $\alpha \in \J_{x_0}$, it follows from \cite[Corollary~9.3.4]{SimsCRM}
that $\|1_{B_\alpha} \phi\|_{C^*(G_T)} = \|\phi\|_\infty = 1$. Since $\psi$ is
smooth, its Fourier coefficients are absolutely summable
\cite[Chapter~8]{Folland}. Hence the series $\sum_{(x_0, h, x_0) \in \J_{x_0}}
\hat{\psi}_{h_0}(h) (1_{B_{(x_0, h, x_0)}} \phi)$ converges to an element $f$
of $C^*(G_T)$. Since Renault's map $j \colon C^*(G_T) \to C_0(G_T)$ of
\cite[Proposition~II.4.2]{Renault} is continuous, we have
\begin{equation}\label{eq:j(f) formula}
 j(f)(y',h',x') =
  \begin{cases}
    \phi(x') \hat{\psi}_{h_0}(h) & \textrm{if}~(y',h',x')\in \bigcup\Bb,\\
    0 & \textrm{otherwise},
  \end{cases}
\end{equation}
for all $(y', h', x') \in G_T$. Fix $w \in H_{\Bb}(x)^\perp$. We must show that
$\pi_{(x, zw)}(f) \not = 0$. Let $K := \coc_T(\Iess_x)$. By Lemma~\ref{lem:piH
kernels equal}, it suffices to show that $\pi^K_{(x, zw)}(f) \not= 0$. For
this, note that, by definition, $[(x, 0, x)]_K = \{(x, h, x) : h \in K\} = \Iess_X$, and
by~\eqref{eq:j(f) formula}, if $j(f)(x,h,x) \not= 0$ then $(x,h,x) \in \bigcup \Bb$. So 
if $(x, h, x) \in [(x, 0, x)]_K$ and $j(f)(x, h, x) \not= 0$, then $(x, h, x) \in \Iess_x \cap \bigcup \Bb = \Bb^{\ess}_x$; in 
particular $h \in H_\Bb(x)$. Hence
\[
\langle e_{[x, 0, x]_K}, \pi^K_{(x,zw)}(f) e_{[x, 0, x]_K} \rangle
  = \sum_{h\in H_\Bb(x)} z^h w^h \hat{\psi}_{h_0}(h)
  = \sum_{h\in H_\Bb(x)} z^h \hat{\psi}_{h_0}(h) \neq 0.
\]
Hence $\pi^K_{(x, zw)}(f) \not= 0$ as required.

It remains to show that $f \in Q$ for every primitive ideal $Q \not\in \pi((B_{x_0}
\times V)\cdot(\Bb^{\ess})^\perp)$. To see this, fix such a $Q$. By
\cite[Theorem~3.2]{Sims-Williams}, there exists $(x_1, z_1) \in X \times \T^k$ such that
$Q = \ker(\pi_{(x_1, z_1)})$, and it suffices to show that
\begin{equation}\label{eq:goal}
  \langle e_{y_2}, \pi_{(x_1,z_1)}(f) e_{y_1} \rangle = 0
\end{equation}
for all $y_1, y_2 \in [x_1]$.

Fix $y_1, y_2 \in [x_1]$. Let $H \coloneqq c(\I_{y_1}) = \{h \in \Z^k : (y_1,
h, y_1) \in G_T\}$. Then $H_\Bb(y_1) \le H$. Fix a
complete set $R \subseteq H$ of representatives of the cosets of $H_{\Bb}(y_1)$
in $H$. Then $H = \bigsqcup_{r \in R} (r + H_{\Bb}(y_1))$. Fix $h_1, h_2
\in \Z^k$ such that $(y_i, h_i, x_1)$ are in $G_T$. Then $(y_2, h_2 - h_1, y_1)
\in G_T$.

We have
\[
\langle e_{y_2}, \pi_{(x_1,z_1)}(f) e_{y_1} \rangle
    = \sum_{(y_2, h, y_1) \in G_T} z_1^h j(f)((y_2, h, y_1))
    = \sum_{h \in H} z_1^{h+h_2-h_1} j(f)((y_2, h+h_2-h_1, y_1)).
\]
By definition of $f$, for each $h$ the number $j(f)(y_2, h+h_2-h_1, y_1)$ is
equal to either $\phi(y_1) \widehat{\psi}_{h_0}(h)$ or to $0$. By
\cite[Chapter~8]{Folland} as before, the series $\sum_{h \in \Z^k}
\widehat{\psi}_{h_0}(h)$ is absolutely convergent, and it follows that the
series $\sum_{h \in H} z_1^{h+h_2-h_1}j(f)((y_2, h+h_2-h_1, y_1))$ is also
absolutely convergent. So we can rearrange its terms to obtain
\begin{align*}
\langle e_{y_2}, \pi_{(x_1,z_1)}(f) e_{y_1} \rangle
    &= \sum_{r \in R} \sum_{h \in H_{\Bb}(y_1)} z_1^{h+r+(h_2-h_1)}j(f)((y_2, h+r+(h_2-h_1), y_1)).
\end{align*}
To see that this is zero, fix $r \in R$. It suffices to show that
\[
    \sum_{h \in H_{\Bb}(y_1)} z_1^{h+r+(h_2-h_1)}j(f)((y_2, h+r+(h_2-h_1), y_1)) = 0.
\]

The formula~\eqref{eq:j(f) formula} shows that $j(f)((y_2, h+r+(h_2-h_1), y_1))
= 0$ for all $h$ if $y_1 \not \in B_{x_0}$, so we may assume that $y_1 \in
B_{x_0}$. Letting $h'_2 \coloneqq r + h_2$, we have
\begin{equation}\label{eq:sum to be zero}
\begin{split}
\sum_{h \in H_{\Bb}(y_1)} z_1^{h+r+(h_2-h_1)}&j(f)((y_2, h+r+(h_2-h_1), y_1))\\
    &= \sum_{h\in H_{\Bb}(y_1)}z_1^{h + (h_2' - h_1)}j(f)\big(y_2,h + (h_2'-h_1),y_1\big).
\end{split}
\end{equation}
If $(y_2,h_2'-h_1,y_1)\notin \bigcup\Bb$, then Condition~\labelcref{it:Bf
right inv} of \cref{def:Bf} implies that
\[
  (y_2,h + (h_2'-h_1),y_1) = (y_2, (h_2'-h_1), y_1)(y_1, h , y_1) \notin \bigcup\Bb
\]
for all $h\in H_{\Bb}(y_1)$, so once again~\eqref{eq:j(f) formula} shows
that~\labelcref{eq:sum to be zero} vanishes. So we may assume that $(y_2,
h_2'-h_1, y_1)\in \bigcup\Bb$.

Condition~\ref{it:Bf right inv} of \cref{def:Bf} then implies that $(y_2, h +
(h_2'-h_1), y_1) \in \bigcup\Bb$ for all $h\in H_{\Bb}(y_1)$. Thus,
using~\eqref{eq:j(f) formula} again, for each $h \in H_{\Bb}(y_1)$, we have
$j(f)\big(y_2,h + (h_2'-h_1),y_1\big) =
\phi(y_1)\widehat{\psi}_{h_0}(h+(h_2'-h_1))$, and so~\eqref{eq:sum to be zero}
becomes
\begin{align*}
\sum_{h \in H_{\Bb}(y_1)} z_1^{h+r+(h_2-h_1)}&j(f)((y_2, h+r+(h_2-h_1), y_1))\\
    &= \phi(y_1)\sum_{h\in H_{\Bb}(y_1)}z_1^{h + (h_2' - h_1)}\widehat{\psi}_{h_0}(h+(h_2'-h_1)).
\end{align*}

To see that this is zero, let $\psi_{h_0 + (h_1 - h_2')}$ be the $(h_1 -
h_2')$-perturbation of $\psi_{h_0}$, and let $\chi_1\in \hat{H}_{\Bb(y_1)}$ be
the character defined by $\chi_1(h) = z_1^h$, for all $h\in H_{\Bb}(y_1)$.
Recall that $\Phi_{H_{\Bb}(y_1), \Z^k}$ is the averaging map
of~\labelcref{eq:averaging-map}. We have
\begin{align}
\phi(y_1)\sum_{h \in H_{\Bb}(y_1)}& z_1^{h+r+(h_2-h_1)} \widehat{\psi}_{h_0}(h+r+(h_2-h_1))\\
  &= z_1^{h_2' - h_1} \phi(y_1) \sum_{h\in H_{\Bb}(y_1)} z_1^h \hat{\psi}_{h_0 + (h_1 - h_2')}(h) \nonumber \\
  &= z_1^{h_2' - h_1} \phi(y_1) \Phi_{H_{\Bb}(y_1), \Z^k} (\psi_{h_0 + (h_1 - h_2')})(\chi_1).\label{eq:Phi_H}
\end{align}

In order to show that~\labelcref{eq:Phi_H} is zero, it now suffices to
show that $\chi_1 \notin \hat{q}_{H_{\Bb}(y_1)}(V)$, because $\supp(\psi_{h_0
+ (h_1 - h_2')}) = \supp(\psi) \subseteq V$ as discussed immediately
after~\eqref{eq:perturbation}, and hence $\supp(\Phi_{H_{\Bb}(y_1)}(\psi_{h_0
+ (h_1 - h_2')})) \subseteq \hat{q}_{H_{\Bb}(y_1)}(V)$ as discussed just
prior to \cref{lem:H-restriction-is-nonzero}. So it suffices to show that
$z_1 w \notin V$ for all $w \in H_{\Bb}(y_1)^\perp$.

Suppose for contradiction that $w \in H_{\Bb}(y_1)^\perp$ satisfies $z_1 w
\in V$. Then $z_1 \in V H_{\Bb}(y_1)^\perp$, and
hence $(y_1, z_1)\in (B_{x_0} \times V)\cdot(\Bb^{\ess})^\perp$. Since $y_1
\in [x_1]$, Theorem~3.2 of \cite{Sims-Williams} implies that $\ker(\pi_{(y_1,
z_1)}) = \ker(\pi_{(x_1, z_1)}) = Q$, contradicting $Q \not\in \pi((B_{x_0}
\times V)\cdot(\Bb^{\ess})^\perp)$.

Thus~\labelcref{eq:Phi_H} is zero, whence~\labelcref{eq:sum to be zero} is
zero, giving~\labelcref{eq:goal}. Hence $f\in Q$.
\end{proof}

\begin{proof}[Proof of \cref{thm:neighbourhood-basis}]
By definition of the topology on $\Prim(C^*(G_T))$, the closed sets are the sets $\{Q : Q
\subseteq I\}$ indexed by ideals $I$ of $C^*(G_T)$. We must show that the complement of
$A(B, V)$ is closed, so we fix $x \in B_{x_0}$, $z \in V$ and $w \in H_\Bb(x)^\perp$ so
that $P = \ker(\pi_{(x, zw)})$ is a typical point in $A(B, V)$. We must show that
$\bigcap_{Q \in \Prim(C^*(G_T))\setminus A(B, V)} Q \not\subseteq P$. That is, we must
find $f_P \in \bigcap_{Q \in \Prim(C^*(G_T)) \setminus A(B, V)} Q$ such that $f_P \not\in
P$. But this is a direct application of \cref{prp:fP}.
\end{proof}

This now leads us to a complete description of the primitive ideal space when the action
$T\colon \N^k\curvearrowright X$ admits sufficiently many harmonious families of
bisections.

\begin{definition}\label{def:admits Bs}
Let $X$ be a second-countable locally compact Hausdorff space and suppose that $T \colon
\N^k \curvearrowright X$ is an action by local homeomorphisms. We say that $T$
\emph{admits harmonious families of bisections} if every $x \in X$ admits a harmonious
family of bisections in $G_T$.
\end{definition}

\begin{corollary}\label{cor:open-if}
Let $X$ be a second-countable locally compact Hausdorff space and suppose
that $T \colon \N^k \curvearrowright X$ is an action by local homeomorphisms
that admits harmonious families of bisections. Let $\pi \colon X \times \T^k \to
\Prim(C^*(G_T))$ be the map of \cref{ntn:pi}. Let $A$ be a subset of
$\Prim(C^*(G_T))$ such that for every $(x_0,z_0)\in \pi^{-1}(A)$ there exist
a harmonious family of bisections $\Bb$ at $x_0$ and an open neighbourhood
$V_0\subset \T^k$ of $z_0$ such that $(B_{x_0} \times V_0) \cdot
(\Bb^{\ess})^\perp \subseteq \pi^{-1}(A)$. Then $A$ is open in
$\Prim(C^*(G_T))$.
\end{corollary}
\begin{proof}
Fix $(x_0, z_0)\in \pi^{-1}(A)$. By hypothesis there exist a harmonious family of
bisections $\Bb$ based at $x_0$ and an open neighbourhood $V_0\subset \T^k$ of $z_0$ such
that $(B_{x_0} \times V_0)\cdot(\Bb^{\ess})^\perp \subseteq \pi^{-1}(A)$. Let $V
\coloneqq V_0 \cdot (H_\Bb(x_0))^\perp$, the $H(x_0)$-saturation of $V_0$ and let
$A(\Bb,V) \subset \Prim(C^*(G_T))$ be the corresponding basic open neighbourhood of
$\ker(\pi_{(x_0,z_0)})$ as in~\labelcref{eq:basic-open}. It suffices to show that
$A(\Bb,V) \subset A$. Let $(x,z)\in \pi^{-1}A(\Bb,V)$. Then there exist $z' \in V_0$ and
$w\in H(x)^\perp$ such that $z = z' w$. Since $(B_{x_0} \times
V_0)\cdot(\Bb^{\ess})^\perp \subseteq \pi^{-1}(A)$, we have $(x,z)\in \pi^{-1}(A)$. Hence
$A(\Bb,V) \subset A$ as required.
\end{proof}

Combining~\cref{thm:Aopen,cor:open-if}, we obtain the following complete description of the hull-kernel topology.

\begin{corollary}\label{cor:open-if-and-only-if}
Let $X$ be a second-countable locally compact Hausdorff space and suppose
that $T \colon \N^k \curvearrowright X$ is an action by local homeomorphisms
that admits harmonious families bisections. Let $\pi \colon X \times \T^k \to
\Prim(C^*(G_T))$ be the map of \cref{ntn:pi}. A subset $A \subset
\Prim(C^*(G_T))$ is open if and only if for every $(x,z)\in \pi^{-1}(A)$
there exist a harmonious family of bisections $\Bb$ based at $x$ and an open
neighbourhood $V\subset \T^k$ of $z$ such that $(x, z) \in B_0 \times V$ and
$(B_0 \times V)\cdot(\Bb^{\ess})^\perp \subseteq \pi^{-1}(A)$.
\end{corollary}

\begin{corollary}\label{cor:base for topology}
Let $X$ be a second-countable locally compact Hausdorff space and suppose
that $T\colon \N^k \curvearrowright X$ is an action by local homeomorphisms
that admits harmonious families of bisections. Let $\pi \colon X \times \T^k \to
\Prim(C^*(G_T))$ be the map of \cref{ntn:pi}. For each $x \in X$, fix a
harmonious family of bisections $\Bb(x) = (B(x)_\alpha)_{\alpha \in \J_x}$ at
$x$. Then the collection
\[
   \big\{\pi\big((U \times V) \cdot (\Bb(x)^{\ess})^\perp\big) : x \in X, x \in U \subseteq_{\text{open}} B(x)_x, V \subset_{\text{open}} \T\big\}
\]
is a base for the topology on $\Prim(C^*(G_T))$.
\end{corollary}

Finally, we mention a few immediate examples to illustrate how the base for the topology recovers well known results.
\begin{example}
  \begin{enumerate}
    \item If $G_T$ is strongly effective, then as in \cref{eg:isolated or
        strongly effective}(2) $\Bb \coloneqq \{G_T^{(0)}\}$ defines a harmonious family of
        bisections at each $x$. We then have $(U \times V) \cdot (\Bb(x)^{\ess})^\perp
        = U \times \T^k$, and \cref{cor:base for topology} reduces to the
        statement that the primitive ideal space is homeomorphic to the quasi-orbit
        space of $G_T$.
    \item If $G_T$ is minimal, then as in \cref{eg:minimal} we have $H(x) = H(y)
        = H$ for all $x,y \in X$, and $G_T$ admits harmonious families of bisections.
        For the harmonious family of bisections $W_x B(n)$ of
        \cref{eg:minimal}, we have $(U \times V) \cdot (\Bb(x)^{\ess})^\perp = U
        \times (V H^\perp)$ whenever $U \subseteq W_x$. The quasi-orbit space is a
        point, so $\ker(\pi_{(x,z)}) = \ker(\pi_{(y,w)})$ if and only if $zH^\perp =
        wH^\perp$, and so \cref{cor:base for topology} establishes that
        $\Prim(C^*(G_T))$ is homeomorphic to $H^\perp$. This recovers
        \cite[Theorem~4.7]{Sims-Williams} for $T$ minimal.
    \item In particular, if $G_T$ is minimal and effective, then the base for the
        topology is just a singleton so $C^*(G_T)$ is simple.
  \end{enumerate}
\end{example}

\section{The lattice of ideals of \texorpdfstring{$C^*(G_T)$}{the C*-algebra of a Deaconu--Renault
groupoid}}\label{sec:lattice}

In this section we apply the results and ideas from the preceding section to describe the
lattice of ideals of $C^*(G_T)$. We provide a partial description for arbitrary $T$, but
in order to obtain a complete description, we must assume that $T$ admits harmonious
families of bisections.

Let $X$ be a second-countable locally compact Hausdorff space. Suppose that $T \colon
\N^k \curvearrowright X$ is an action by local homeomorphisms. Recall
from~\cref{sec:ideals} that for each ideal $I$ in $C^*(G_T)$,
\[
  A_I \coloneqq \{ P \in \Prim(C^*(G_T)) : I \not\subset P \}
\]
is open in the hull-kernel topology on $\Prim(C^*(G_T))$. So
\cref{cor:pi-continuous} implies that with respect to the map $\pi \colon X \times
\T^k \to \Prim(C^*(G_T))$ of \cref{ntn:pi}, the map
\begin{equation} \label{eq:lattice-homomorphism}
  \theta\colon I \mapsto \pi^{-1}(A_I)
\end{equation}
is an injective lattice homomorphism from the ideals of $C^*(G_T)$ into the open subsets
of $X\times \T^k$ ordered by inclusion.

The range of this lattice homomorphism is difficult to describe in general, but
\cref{thm:Aopen} describes a (fairly technical) invariance condition that every $A_I$
must satisfy. The range of $\theta$ consists of open sets $W$ that are
$\pi$-saturated in the sense that $W = \pi^{-1}(\pi(W))$. However, not every open
$\pi$-saturated subset of $X$ need be in the range of $\theta$; if it were, then $\pi$
would be a quotient map, and \cite[Remark 3.3 and Example 3.4]{Sims-Williams} show that
this is not the case in general.

To describe the range of $\theta$, and to describe generators of the ideal
$\theta^{-1}(W)$ for a given set $W$, we must assume that $T$ admits harmonious families
of bisections. The description of the range follows directly from
\cref{cor:open-if-and-only-if}.

\begin{lemma}\label{lem:image-of-theta}
Let $X$ be a second-countable locally compact Hausdorff space and suppose that $T \colon
\N^k \curvearrowright X$ is an action by local homeomorphisms that admits harmonious
families of bisections. Let $\theta$ be the lattice homomorphism
of~\labelcref{eq:lattice-homomorphism}. A subset $W \subset X\times \T^k$ is in the image
of $\theta$ if and only if
\begin{enumerate}
  \item $W$ is open and $\pi$-saturated, and
  \item\label{it:near W} whenever $(x,z)\in W$, then there exist a harmonious families
      of bisections $\Bb$ at $x$ and an open neighbourhood $V\subset \T^k$ of $z$ such
      that $(B_x \times V) \cdot (\Bb^{\ess})^\perp \subseteq W$.
\end{enumerate}
\end{lemma}

We may now describe generators for the ideal $\theta^{-1}(A)$ corresponding to a given
set satisfying the conditions of \cref{lem:image-of-theta}.
The generating elements are exactly the elements of $C^*(G_T)$ we obtain from the
noncommutative Urysohn lemma,~\cref{prp:fP}.

\begin{proposition}\label{prp:ideal generators}
Let $X$ be a second-countable locally compact Hausdorff space and suppose that $T \colon
\N^k \curvearrowright X$ is an action by local homeomorphisms that admits harmonious
families of bisections. Let $\theta$ be the lattice homomorphism
of~\labelcref{eq:lattice-homomorphism}. Suppose that $W \subseteq X \times \T^k$
satisfies the conditions of \cref{lem:image-of-theta}. Then
\[
\theta^{-1}(W) = \bigcap_{(x, z) \in (X \times \T^k)\setminus W} \ker(\pi_{(x, z)}).
\]
For each $(x_0, z_0) \in W$, fix a harmonious families of bisections $\Bb^{(x_0, z_0)}$
at $x_0$ and an open neighbourhood $V^{(x_0, z_0)}$ of $z_0$ as in
Condition~\labelcref{it:near W} of \cref{lem:image-of-theta}, and let $f^{(x_0, z_0)}$ be
any element obtained from~\cref{prp:fP} applied to $(x_0, z_0)$, $\Bb^{(x_0, z_0)}$, and
$V^{(x_0, z_0)}$. Then $\theta^{-1}(W)$ is generated as an ideal by $\{f^{(x_0, z_0)} :
(x_0, z_0) \in W\}$.
\end{proposition}
\begin{proof}
Fix a set $W$ in the range of $\theta$. By~\cite[Proposition A.17(a)]{Raeburn-Williams}
the ideal $\theta^{-1}(W)$ is equal to the intersection of the primitive ideals
that contain it. Hence \cref{lem:image-of-theta} gives the first statement.

Let $I = \theta^{-1}(W)$. By definition of $\theta$, we have $W = \{(x,z) \in X \times
\T^k : I \not\subseteq \ker(\pi_{(x, z)})\}$. Let $J$ be the ideal generated by the
$f^{(x_0, z_0)}$. Since every ideal is the intersection of the primitive ideals
containing it and since $\pi$ is surjective, it suffices to show that for $(x, z) \in X
\times \T^k$, we have $J \subseteq \ker(\pi_{(x,z)})$ if and only if $(x,z) \not\in W$.

So fix $(x, z) \in X \times \T^k$. First suppose that $(x, z) \not\in W$. Then for each
$(x_0, z_0) \in W$,
\[
A(\Bb^{(x_0, z_0)}, V^{(x_0, z_0)})
    = (B^{(x_0, z_0)}_{x_0} \times V^{(x_0, z_0)}) \cdot ((\Bb^{(x_0, z_0)})^{\ess})^\perp
    \subseteq W,
\]
so $(x, z) \not\in A(\Bb^{(x_0, z_0)}, V^{(x_0, z_0)})$. Hence, by~\cref{prp:fP},
$f^{(x_0, z_0)} \in \ker(\pi_{(x, z)})$. Since $(x_0, z_0) \in W$ was arbitrary, it
follows that all the generators of $J$ belong to the kernel of $\pi_{(x, z)}$. So $J
\subseteq \ker(\pi_{(x, z)})$ as required. Now suppose that $(x, z) \in W$. Then
$\pi_{(x,z)}(f^{(x, z)}) \not= 0$ by~\cref{prp:fP}. Since $f^{(x, z)}$ is a generator of
the ideal $J$, it belongs to $J$ and so $J \not\subseteq \ker(\pi_{(x, z)})$ as required.
\end{proof}

\section{Convergence of primitive ideals}\label{sec:convergence}
Now we apply our results on the primitive ideal space of $C^*(G_T)$ to describe convergence of primitive ideals.
We consider only systems $T\colon \N^k\curvearrowright X$  on second-countable spaces, so the $C^*$-algebras $C^*(G_T)$ are separable,
and the primitive ideal space $\Prim(C^*(G_T))$ is second-countable (cf. e.g.~\cite[p. 231]{Raeburn-Williams}),
so it suffices to consider convergent \emph{sequences}.

The map $\pi\colon X\times \T^k \to \Prim(C^*(G_T))$ from \cref{ntn:pi} is
continuous, so if $(x_i, z_i)_i \to (x,z)$ in $X\times \T^k$, then $\pi(x_i,
z_i) \to \pi(x,z)$ in $\Prim(C^*(G_T))$. This is however far from a complete
descrition: many divergent sequences in $X \times \T^k$ descend to convergent
sequences in $\Prim(C^*(G_T))$.

We first describe a weaker sufficient condition for convergence of a sequence of primitive ideals.
In order to do this, we need to extend the notation of \cref{rmk:H_B groups}. Suppose
that $X$ is a second-countable locally compact Hausdorff space and suppose that $T \colon \N^k
\curvearrowright X$ is an action by local homeomorphisms. Fix $x \in X$, and let $\Bb \coloneqq
(B_\alpha)_{\alpha \in \J_x}$ be a collection of open bisections such that $\alpha \in
B_\alpha \subseteq \coc^{-1}(\coc(\alpha))$ for all $\alpha \in \J_x$. For each $y \in
X$, we consider the subgroup
\[
H_\Bb(y) \coloneqq \Span_{\Z} \coc\Big(\bigcup \Bb^{\ess}_y\Big)
\]
in $\Z^k$ generated by the values of $\coc$ on the intersection of $\bigcup \Bb$ with
$\Iess_y$. That is,
\[
H_\Bb(y) = \Span_{\Z} \{\coc(\alpha) : \alpha \in \J_x\text{ and } B_\alpha \cap \Iess_y \not= \varnothing\}.
\]
If $\Bb$ is a harmonious family of bisections and $y \in B_x$, then $\Bb^{\ess}_y$ is a
group by Conditions \ref{it:Bf right inv}~and~\ref{it:Bf inverses} of \cref{def:Bf}, so
this new definition of $H_\Bb(y)$ agrees with the one given in \cref{rmk:H_B groups}.

Now given a subgroup $H\subset \Z^k$, we let $\hat{q}_H\colon \T^k \to \hat{H}$ be the canonical quotient map.

\begin{definition}\label{dfn:converges along}
Suppose that $(H_n)_{n}$ is a sequence of subgroups of $\Z^k$. A sequence
$(z_n)_{n}$ in $\T^k$ \emph{converges to $z\in\T^k$ along $(H_n)_{n}$} if for
every open neighbourhood $V\subset \T^k$ of $z$ in $\T^k$ there exists $N\in\N$ such that
$\hat{q}_{H_n}(z_n)\in \hat{q}_{H_n}(V)$ for all $n\geq N$.
\end{definition}

The following sufficient condition for convergence of primitive ideals applies to any Deaconu--Renault system.

\begin{proposition}\label{prp:sufficient for convergence}
Let $X$ be a second-countable locally compact Hausdorff space and suppose
that $T \colon \N^k \curvearrowright X$ is an action by local homeomorphisms.
For $x \in X$ and $z \in \T^k$, let $\pi_{(x,z)}$ be as in \cref{ntn:pi}. Let
$(x,z)\in X\times \T^k$ and let $\Bb = (B_{\alpha})_{\alpha \in \J_x}$ be a
collection of open bisections such that $\alpha \in B_\alpha \subseteq
\coc^{-1}(\coc(\alpha))$ for all $\alpha$. If $(x_i, z_i)_i$ is a sequence in
$\in X \times \T^k$ satisfying $x_i \to x$ and that $z_i \to z$ along
$H_\Bb(x_i)$, then $\ker(\pi_{(x_i, z_i)}) \to \ker(\pi_{(x, z)})$ in
$\Prim(C^*(G_T))$.
\end{proposition}
\begin{proof}
Fix an open set $A \subseteq \Prim(C^*(G_T))$ that contains $\ker(\pi_{(x, z)})$. We must show that
$\ker(\pi_{(x_i, z_i)}) \in A$ for large $i$.
By~\cref{thm:Aopen} there exist an open neighbourhood $U \subset X$ of $x$ and an open neighbourhood $V\subset \T^k$ of $z$ such that
$(U \times V) \cdot (\Bb^{\ess})^\perp \subseteq A$.
Since $x_i \to x$ we see that $x_i \in U$ for large $i$,
so it suffices to show that $z_i \in V \cdot (\Bb^{\ess}_{x_i})^\perp$ for large $i$.

For each $i$, let $H_i \coloneqq H_\Bb(x_i)$.
Since $z_i \to z$ along $(H_i)_i$, we have $\hat{q}_{H_i}(z_i) \in \hat{q}_{H_i}(V)$ for large $i$.
So there exist $z_i'\in V$ such that for large $i$, we have $(z_i)^h = (z_i')^h$ for all $h \in H_i$.
That is, $w_i \coloneqq z_i \overline{z'_i} \in (\Bb^{\ess}_{x_i})^\perp$ for large $i$,
and it follows that $z_i = w_i z'_i \in V \cdot (\Bb^{\ess}_{x_i})^\perp$ for large $i$ as required.
\end{proof}

We can also identify a necessary condition that is valid for all Deaconu--Renault
systems. Recall that the \emph{quasi-orbit space} $\Q(G)$ of an \'etale groupoid $G$ is
the set $\big\{\overline{[x]} : x \in G^{(0)}\big\}$ of orbit closures in $G$, in the
quotient topology induced by the surjection $\Q\colon x \mapsto \overline{[x]}$ from
$G^{(0)}$ to $\Q(G)$.

\begin{lemma} \label{lem:prim-to-Q-cts}
Let $X$ be a second-countable locally compact Hausdorff space and suppose
that $T\colon \N^k\curvearrowright X$ is an action by local homeomorphisms.
For $x \in X$ and $z \in \T^k$, let $\pi_{(x,z)}$ be as in \cref{ntn:pi}.
Then the quasi-orbit map $q\colon \Prim(C^*(G_T)) \to \Q(G_T)$ given by
$q(\ker(\pi_{(x,z)})) = \overline{[x]}$ for all $(x,z)\in X\times \T^k$ is
continuous. In particular, if $\ker(\pi_{(x_n,z_n)}) \to \ker(\pi_{(x,z)})$
in $\Prim(C^*(G_T))$, then $\overline{[x_n]} \to \overline{[x]}$ in
$\Q(G_T)$.
\end{lemma}
\begin{proof}
  Take a closed subset $K$ of $\Q(G_T)$ and consider the preimage
  \[
    W = q^{-1}(K) = \{ \ker(\pi_{(x,z)}) : \overline{[x]}\in K, z\in \T^k\}.
  \]
  We aim to show that $W$ is closed in $\Prim(C^*(G_T))$.
  Take $\ker(\pi_{(y,w)})$ in the closure of $W$.
  By definition of the hull-kernel topology, this means that
  \[
    \bigcap_{\overline{x}\in K, z\in \T^k} \ker(\pi_{(x,z)}) \subset \ker(\pi_{(y,w)}).
  \]
  It suffices to show that $\overline{[y]} \in K$, so assume for contradiction that this is not the case.
  Then $y$ is not in the preimage $Y = \Q^{-1}(K)$ which is closed and invariant.
  By Urysohn's lemma, there is a function $f\in C_c(X, [0,1])$ satisfying $f(y) = 1$ and $f|_Y = 0$.
  But then $f\in \ker(\pi_{(x,z)})$ whenever $\overline{[x]}\in K$ and $f\notin \ker(\pi_{(y,w)})$,
  and this contradicts the inclusion above.
\end{proof}

B\"{o}nicke and Li~\cite[Remark~3.16]{Bonicke-Li} observe that $\Q$ is a continuous and
open surjection and that $\Q(G)$ is T0 and Baire. When $X$ is second-countable, the
quasi-orbit space $\Q(G_T)$ is also second-countable. Therefore, the quasi-orbit map is
\emph{sequence-covering} in the sense of Siwiec, see~\cite[Proposition~2.4]{Siwiec}
(Siwiec assumes that all spaces are Hausdorff but the proof is valid even if the codomain
is not Hausdorff). Indeed, since the quasi-orbit map is not just an almost-open map but a
\emph{bona fide} open map, the proof of \cite[Proposition~2.4]{Siwiec} establishes the
following strong sequence-covering condition:

\begin{lemma}\label{lem:siwiec+}
Let $X$ be a second-countable locally compact Hausdorff space and suppose
that $T\colon \N^k\curvearrowright X$ is an action by local homeomorphisms.
Suppose that $x \in X$ and that $(x_n)_{n}$ is a sequence in $X$ such that
$\overline{[x_n]} \to \overline{[x]}$ in $\Q(G_T)$. Then there is a sequence
$(x'_n)_n$ in $X$ such that $x'_n \to x$ and $\overline{[x'_n]} =
\overline{[x_n]}$ for all $n$.
\end{lemma}
\begin{proof}
Since $\Q$ is open, $\Q(U)$ is a neighbourhood of $\overline{[x]}$ in $\Q(G)$ for each
neighbourhood $U$ of $x$. Now putting $y_n \coloneqq \overline{[x_n]} \in \Q(G)$ for each $n$,
the proof of \cite[Proposition~2.4]{Siwiec} starting from the fifth sentence ``Let $F_n =
f^{-1}(y_n)\dots$,'' establishes the result.
\end{proof}

We will now use \cref{lem:siwiec+} together with harmonious families of bisections to
obtain a complete description of convergence of primitive ideals.

\begin{theorem}\label{thm:necsuff for convergence}
Let $X$ be a second-countable locally compact Hausdorff space and suppose
that $T\colon \N^k\curvearrowright X$ is an action by local homeomorphisms
that admits bisection families. For $x \in X$ and $z \in \T^k$, let
$\pi_{(x,z)}$ be as in \cref{ntn:pi}. Take $(x,z)\in X\times\T^k$, let $\Bb =
\big(B_\alpha)_{\alpha \in \J_x}$ be a harmonious family of bisections at
$x$, and let $P = \ker(\pi_{(x,z)})$. Let $(P_n)_n$ be a sequence in
$\Prim(C^*(G_T))$. Then $P_n\to \ker(\pi_{(x,z)})$ in $\Prim(C^*(G_T))$ if
and only if there is a sequence $(x'_n,z_n)_n$ in $X\times\T^k$ satisfying
\begin{enumerate}[label=(\roman*)]
\item $\ker(\pi_{(x'_n, z_n)}) = P_n$ for all $n\in\N$,
\item $x'_n \to x$ in $X$,
\item $z_n \to z$ along $(H_\Bb(x'_n))_n$.
\end{enumerate}
\end{theorem}
\begin{proof}
First suppose that $P_n \to P$. Choose a sequence $(x_n, z_n)_n$ in $X\times \T^k$ such
that $P_n = \ker(\pi_{(x_n,z_n)})$ for all $n$. Then $\overline{[x_n]} \to
\overline{[x]}$ in $\Q(G_T)$ by~\cref{lem:prim-to-Q-cts}. By \cref{lem:siwiec+},
there is a convergent sequence $(x'_n)_n$ in $X$ such that $\overline{[x'_n]} =
\overline{[x_n]}$ for all $n$ and $x'_n \to x$. We must show that $z_n \to z$ along
$(H_\Bb(x'_n))_n$. For this, fix an open set $V\subset \T^k$ containing $z$.
\cref{thm:neighbourhood-basis} shows that the set $A = A(\Bb,V)$
of~\labelcref{eq:basic-open} is an open neighbourhood of $P$, so $P_n \in A$ for large
$n$. Hence for large $n$ there exists $z_n''\in V$ such that $(z_n)^h = (z_n'')^h$ for
all $h\in H_\Bb(x'_n)$. So $z_n \to z$ along $\big(H_\Bb(x'_n)\big)_n$.

Conversely, suppose that there is a sequence $(x'_n, z_n)$ satisfying the three
conditions in the theorem. Assume for contradiction that $P_n \not\to P$ and choose an
open neighbourhood $A$ of $\pi(x,z)$ and a subsequence $(P_{n_i})_i$ of $(P_n)_n$ such
that $P_{n_i} \not\in A$ for all $i$. By~\cref{thm:Aopen}, there exist an open $U
\subseteq X$ containing $x$ and an open $V \subseteq \T^k$ containing $z$ such that $(U
\times V) \cdot (\Bb^{\ess})^\perp \subseteq \pi^{-1}(A)$. In particular, $P_{n_i} \not
\in \pi((U \times V) \cdot (\Bb^{\ess})^\perp)$ for all $i$. Since $(x'_n, z_n)$
satisfies the conditions of the theorem, we have $x'_{n_i} \in U$ for large $i$. Since
$z_n \to z$ along $(H_\Bb(x_n))_n$, we have $z_{n_i} \to z$ along $(H_\Bb(x_{n_i}))_i$,
which forces $z_{n_i} \in V \cdot H_\Bb(x'_{n_i})^\perp = V \cdot H_\Bb(x_{n_i})^\perp$
for large $i$. Hence $(x_{n_i}, z_{n_i}) \in (U \times V) \cdot (\Bb^{\ess})^\perp$ for
large $i$, contradicting that $\ker(\pi_{(x_{n_i}, z_{n_i})}) = P_{n_i} \notin A$.
\end{proof}

\section{Examples}\label{sec:examples}

In this section, we provide a number of examples. Firstly, we apply our theorems to
recover the primitive-ideal space of the $C^*$-algebra of the ``dumbbell graph'' of
\cref{eg:dumbbell}, and then also those of $C^*$-algebras of arbitrary row-finite graphs
with no sources, and of crossed products of locally compact Hausdorff spaces by actions
of $\Z^k$. These examples are intended to be illustrative. The results are not new.

Next we outline how our results and techniques relate to those of Katsura in
\cite{Katsura2021}. Much of the work on the present paper had been completed when
Katsura's work was posted on the arXiv, so naturally we were interested to determine how
the two approaches relate. It turns out that our techniques can be used in Katsura's
setting, but not by a straightforward application of our main theorem. Combining the two
approaches seems like an avenue for future exploration.

Finally, we demonstrate that our results and hypotheses are checkable for all $2$-graph
groupoids. The question of how to describe the primitive-ideal spaces of $C^*$-algebras
of $2$-graphs was the initial motivation for our line of investigation, so we detail how
our main results answer this question.

\subsection{The Dumbell graph}\label{sec:dumbbell}

The ideal structure of the $C^*$-algebra of the dumbell graph is relatively simple, but
illuminating. It is also well-known---it is an example of the original results of an Huef
and Raeburn on primitive-ideal spaces of Cuntz--Krieger
algebras~\cite{anHuef-Raeburn1997}, and is analysed explicitly
in~\cite[Example~3.4]{Sims-Williams}.

First recall that the dumbell graph is the graph $E$ depicted below
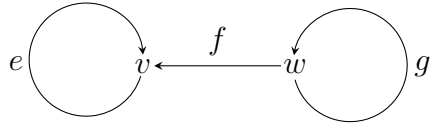
\[
\begin{tikzpicture}
	\node[inner sep=0pt, circle] (v) at (0,0) {$v$};
	\node[inner sep=0pt, circle] (w) at (2,0) {$w$};
	\draw[-stealth] (w) to node[pos=0.5, above] {$f$} (v);
	\draw[stealth-] (v.north) arc[radius=0.75, start angle=5, end angle=343];
	\draw[-stealth] (w.south) arc[radius=0.75, start angle=195, end angle=530];
	\node[inner sep=0pt, circle, anchor=east] at (-1.5,0) {$\vphantom{g}e$};
	\node[inner sep=0pt, circle, anchor=west] at (3.5,0) {$g\vphantom{e}$};
\end{tikzpicture}
\]
As discussed in \cref{eg:dumbbell}, the essential isotropy in $G_E$ is
\[
\{(e^\infty, n, e^\infty) : n \in \Z\} \cup \{(e^mfg^\infty, n, e^mfg^\infty) : m\ge0, n \in \Z\}
	\cup \{(g^\infty, n, g^\infty) : n \in \Z\}.
\]
The unit space is a clopen subset homeomorphic to $\N \cup \{\infty\}$, with $(e^\infty,
n, e^\infty)$ identified with the point at infinity, and the remainder of $\Iess$ is
discrete.

In particular, there are just two orbits, namely $[e^\infty] = \{e^\infty\}$,
and $[g^\infty] = E^\infty \setminus \{e^\infty\}$. Their orbit closures are
$\{e^\infty\}$ and $E^\infty$ respectively. For $x \in E^\infty$ and $z \in
\T$, let $\pi_{(x,z)}$ be as in \cref{ntn:pi}. Then the primitive ideals of
$C^*(E)$ are $\{\ker(\pi_{(e^\infty, z)}), \ker(\pi_{(g^\infty, z)}) : z \in
\T\}$. For each of $x = e^\infty$ and $x = g^\infty$, the group $H(x) =
\coc(\Iess_x)$ is $\Z$, so two elements $z, z' \in \T$ induce the same
character of $H(x)$ if and only if they are equal. Hence $\pi : (x,z) \to
\ker(\pi_{(x, z)})$ is a bijection $\{e^\infty, g^\infty\} \times \T \to
\Prim(C^*(E))$.

To describe the topology, first observe that since $g^\infty$ is an isolated point in
$E^\infty$, we obtain a harmonious family of bisections $\Bb^g$ at $g^\infty$ by putting
$B_{(g^\infty, n, g^\infty)} = \{(g^\infty, n, g^\infty)\}$. By
\cref{thm:neighbourhood-basis}, the sets $\{A(\Bb^g, V) : V \subseteq \T\text{ is
open}\}$ are a basis for the topology on the clopen subset $\{\ker(\pi_{(g^\infty, z)} :
z \in \T\}$, and we deduce that this subset is homeomorphic to $\T$ via
$\ker(\pi_{(g^\infty, z)} \mapsto z$.

Now consider $e^\infty$. To lighten notation, we write $\alpha(n) \coloneqq (e^\infty, n,
e^\infty)$ for each $n \in \Z$. We define bisections $C_{\alpha(n)}$ as follows:
\[
C_{\alpha(n)} \coloneqq \begin{cases}
		Z(e^n, v) &\text{ if $n > 0$}\\
		E^\infty &\text{ if $n = 0$}\\
		Z(v, e^{-n}) & \text{if $n < 0$}
\end{cases}
\]
We claim that $\mathcal{C} \coloneqq (C_{\alpha(n)})_{n \in \Z}$ is a harmonious family
of bisections for $e^\infty$. Clearly each $C_{\alpha(n)}$ is an open bisection.
Conditions \ref{it:Bf units}~and~\ref{it:Bf homogeneous} are immediate from the
definition of the $C_{\alpha(n)}$. Since each $C_{\alpha(n)}$ is a compact open
bisection, it satisfies~\ref{it:compact.open} with $K_{\alpha(n)} = s(C_{\alpha(n)})$.
For conditions \ref{it:Bf inverses}~and~\ref{it:Bf right inv}, it suffices to show that
for $n \not= 0$, we have $C_{\alpha(n)} \cap \Iess = \{\alpha(n)\}$. To see this, observe
that if $x \in r(C_{\alpha(n)})$, then $(x, n, \sigma^n(x))$ is the unique element of
$C_{\alpha(n)}$, so it suffices to fix $x \in r(C_{\alpha(n)}) \setminus \{e^\infty\}$
and show that $\sigma^n(x) \not= x$. For this, note that by definition of $C_{\alpha(n)}$
we have $x = e^n x'$ and $\sigma^n(x) = x'$ for some $x'$. Since $x \not= e^\infty$, we
deduce that there exists $k \ge 0$ such that $x = e^{n+k}fg^\infty$, and hence $x' = e^k
f g^\infty \not= x$.

Now given $z \in \T^k$, \cref{thm:Aopen,thm:neighbourhood-basis} imply that the sets $(U
\times V)\cdot(\mathcal{C}^{\ess})^\perp$ ranging over open neighbourhoods $U$ of
$e^\infty$ and $V$ of $z$ are the pre-images of a neighbourhood basis for
$\ker(\pi_{(e^\infty, z)})$. Since, for any unit $y \not= e^\infty$ in $U$, we have
$\mathcal{C}^{\ess}_y \cap (G_E)_y = \{y\}$, we see that $(\mathcal{C}^{\ess})^\perp_y =
\{y\} \times \T$ for $y \not= e^\infty$. Since $\mathcal{C}^{\ess}_{e^\infty} \cap
(G_E)_{e^\infty} = \{\alpha(n) : n \in \Z\}$, we have $(\Bb^{\ess})^\perp_y =
\{e^\infty\} \times \{1\}$. So for any neighbourhood $U$ of $e^\infty$ and any
neighbourhood $V$ of $z$, we have
\[
	(U \times V)\cdot(\mathcal{C}^{\ess})^\perp = \big(\{e^\infty\}\times V\big) \cup \big((U \setminus \{e^\infty\}) \times \T\big).
\]
Since $\ker(\pi_{(x, z)}) = \ker(\pi_{(g^\infty, z)})$ for $x \in
E^\infty\setminus\{e^\infty\}$, it follows that a basic open neighbourhood of
$\ker(\pi_{(e^{\infty}, z)})$ has the form
\[
\{\ker(\pi_{(e^\infty, w)}) : w \in V\} \cup \{\ker\pi_{(g^\infty, w)} : w \in \T\}.
\]

To summarise, if we put
\[
P_v \coloneqq \{\ker(\pi_{(g^\infty, z)}) : z \in \T\}\quad\text{ and }\quad P_w \coloneqq \{\ker(\pi_{(e^\infty, z)}) : z \in \T\},
\]
then $\Prim(C^*(E)) = P_v \sqcup P_w$; the subset $P_w$ is open; the map
$\ker(\pi_{(e^\infty, z)}) \mapsto z$ is a homeomorphism of $P_v$ in the relative
topology onto $\T$; the map $\ker(\pi_{(g^\infty, z)}) \mapsto z$ is a homeomorphism of
$P_w$ in the relative topology onto $\T$; and for any $z$, the closure of the point
$\ker(\pi_{(g^\infty, z)})$ is $P_v \cup \ker(\pi_{(g^\infty, z)})$.

\begin{remark}
As mentioned in~\cite[Example~3.4]{Sims-Williams}, this example demonstrates that the map $\pi \colon (x, z) \mapsto \ker(\pi_{(x,z)})$ is not an open map,
since the image of $W \coloneqq E^\infty \times \{z \in \T : \operatorname{Re}(z) > 0\}$ is not open.
Indeed, $\pi$ is not even a quotient map: $\pi(W) = \{\ker(\pi_{(e^\infty, z)}), \ker(\pi_{(g^\infty, z)}) : \operatorname{Re}(z) > 0\}$,
and since $\pi^{-1}(\ker(\pi_{(g^\infty, z)})) = \{\ker(\pi_{(\alpha g^\infty, z)}) : \alpha \in E^* w\}$ it follows that $\pi^{-1}(\pi(W)) = W$;
so $\pi(W)$ is a subset of $\Prim(C^*(E))$ that is not open but whose preimage in $E^\infty \times \T$ is open.
\end{remark}

\subsection{Graph groupoids}

The analysis of the dumbbell graph above extends to arbitrary row-finite graphs with no
sources. Our results here are not new---they are special cases of the theorems of
\cite{Hong-Szymanski}, and also appear in \cite{Carlsen-Sims}. We include them to
indicate how our results relate to those papers.

The orbit closures are indexed by the maximal tails $T$ as in \cite{Hong-Szymanski},
but for our picture, we prefer to describe them in terms of orbit closures. The correspondence is as follows.
Given an infinite path $x$ with orbit-closure
\[
\overline{[x]} = \{y :\text{ for each }m \ge 0\text{ there exists } n \ge 0\text{ such that } r(y_m) E^* r(x_n)\not= \varnothing\},
\]
the corresponding maximal tail $T_x$ is the set of vertices $v$ such that $v E^* r(x_n) \not=\varnothing$ for some $n$.
Conversely, given a maximal tail $T$, enumerate the vertices of $T$ as $(v_1, v_2, \dots)$.
Let $w_1 \coloneqq v_1$.
Condition~(3) for maximal tails ensures that there exists $w^0_2$ such that $w_1 E^* w^0_2$ and $v_2 E^* w^0_2$ are nonempty.
Condition~(2) for maximal tails ensures that there exists $e \in w^0_2 E^1$ with $s(e) \in T$. So $w_2 \in T$
has the property that $w_1 E^* w_2 \setminus E^0$ and $v_2 E^* w_2$ are nonempty.
Repeating this procedure, we obtain a sequence $w_n$ such that $w_1 = v_1$ and $w_n E^* w_{n+1} \setminus E^0$ and $v_{n+1} E^* w_{n+1}$ are both nonempty for all $n$.
For each $n$, fix $\mu_n \in w_n E^* w_{n+1}$.
Then $x = \mu_1 \mu_2 \dots$ is an infinite path.
We clam that $T = T_x$.
Since each $v_n E^* w_n \not= \varnothing$, we see that $v_n \in T_x$ for all $n$, and so $T \subseteq T_x$.
For the reverse containment, fix $v \in T_x$.
Then $v E^* r(x_i) \not= \varnothing$ for some $i$, say $\alpha \in v E^* r(x_i)$.
Since each $|\mu_j| \ge 1$ there exists $n$ such that $l \coloneqq |\mu_1 \dots \mu_n| \ge i$.
So $\mu_1 \dots \mu_n = x_1 \dots x_i x_{i+1} \dots x_l$ and in particular $\alpha x_i x_{i+1} \dots x_l \in v E^* r(x_{l+1})$.
So $v \in T$ by condition~(1) for maximal tails. So $T_x \subseteq T$ as required.

If $x$ is an infinite path such that every cycle in $E T_x$ has an entrance in $T_x$,
then $\Iess_x = \{x\}$, so constructing a bisecton family at $x$ is trivial (just take
$B_x = Z(r(x))$). If $x$ is an infinite path and there is a cycle $\mu \in E T_x$ with no
entrance in $E T_x$, then $x = \alpha\mu^\infty$ for some $\alpha$, and we may as well
take $\alpha = r(\mu)$ and $x = \mu^\infty$, and assume that $\mu$ is the cycle of
minimal length such that $\mu^\infty = x$ (that is, that $\mu$ is not a multiple of a
shorter cycle). In this instance, $H(x) = |\mu| \Z$, and $\J_x = \{(\mu^\infty, n|\mu|,
\mu^\infty) : n \in \Z\}$. We write $\alpha(n) \coloneqq (\mu^\infty, n|\mu|, \mu^\infty)$ for
all $n$. Define
\begin{equation}\label{eq:BF for graph}
B_{\alpha(n)} \coloneqq \begin{cases}
	Z(\mu^n, r(\mu)) &\text{ if $n > 0$}\\
	E^\infty &\text{ if $n = 0$}\\
	Z(r(\mu), \mu^{-n}) & \text{if $n < 0$.}
\end{cases}
\end{equation}
We claim that $B_{\alpha(n)} \cap \Iess = \{\alpha(n)\}$ for all $n$. To see this,
suppose that $x \in r(B_{\alpha(n)}) \setminus \{\mu^\infty\}$. Since $\mu$ has no
entrance in $ET_x$, there exists a minimum $i \in \N$ such that $s(x_i) \not\in T$. We
have $x = \mu^n x'$ for some $x'$, and since the vertices on $\mu$ belong to $t$, we have
$i > n|\mu|$. Hence the minimum $j \in \N$ such that $s(x'_j) \not\in T$ is $j = i -
n|\mu| \not= i$, and so $x' \not= x$. Since $(x, n|\mu|, x')$ is the unique element of
$B_{\alpha(n)}$ with range $x$, we deduce that $B_{\alpha(n)} \cap \I(G_E)_x =
\varnothing$. Since $x \in r(B_{\alpha(n)}) \setminus \{\mu^\infty\}$ was arbitrary, we
deduce that $B_{\alpha(n)} \cap \I(G_E) = B_{\alpha(n)} \cap \I(G_E)_{\mu^\infty} =
\{\alpha(n)\}$ as claimed. It now follows just as in the dumbbell graph that $\Bb =
(B_{\alpha(n)})_{n \in \Z}$ is a harmonious family of bisections at $\mu^\infty$.

Let $S \subseteq E^\infty$ be a set containing one representative for each
orbit closure. We define a partial order $\le$ on $S$ by $y \le x$ if and
only if $y \in \overline{[x]}$. Write $S = S_a \sqcup S_p$ where $S_a$
consists of precisely the elements of $S$ such that every cycle in the
associated maximal tail has an entrance. For each $x \in S_p$ we can choose a
cycle $\mu$ of minimal length with no entrance in the tail corresponding to
$x$, and may assume that $x = \mu^\infty$. For $x \in E^\infty$ and $z \in
\T$, let $\pi_{(x,z)}$ be as in \cref{ntn:pi}. Then
\[
\Prim(C^*(E)) = \{\ker(\pi_{(x,z)}) : x \in S\},
\]
and that $\ker(\pi_{(x, z)}) = \ker(\pi_{(x, w)})$ for all $z,w$ if $x \in S_a$, whereas
$\ker(\pi_{(\mu^\infty, z)}) = \ker(\pi_{(\mu^\infty, w)})$ if and only if $z^{|\mu|} =
w^{|\mu|}$ when $\mu^\infty \in S_p$.

To describe point closures, we argue precisely as in the dumbbell-graph example to see
that for any infinite path $y \in S$ and any $z$, the closure of $\ker(\pi_{(y, z)})$ is
$\ker(\pi_{(y, z)}) \cup \{\ker(\pi_{(y', w)}) : y' < y\text{ and } w \in \T\}$.

We can recover \cite[Theorem~4.1]{Carlsen-Sims}, which describes the closure operation in
the hull-kernel topology as follows. Fix a set $Y$ of pairs $(x, z)$ consisting of an
infinite path $x$ and an element $z \in \T$. To describe the corresponding set of pairs
$(T, z)$ consisting of a maximal tail and an element of $\T$ in
\cite[Theorem~4.1]{Carlsen-Sims}, for a maximal tail $T$, we define $\Per(T)$ to be equal
to $n$ if $T$ contains a cycle with no entrance and the minimal length of such a cycle is
$n$, and to be 0 otherwise. Then the set of pairs appearing in
\cite[Theorem~4.1]{Carlsen-Sims} is $\{(T_x, z^{\Per(T_x)}) : (x,z) \in Y\}$. We must
show that $\ker(\pi_{(y, w)})$ belongs to the closure of $\{\ker(\pi_{(x,z)}) : (x,z) \in
Y\}$ if and only if $T_y \subseteq \bigcup_{(x,z) \in Y} T_x$ and if $\Per(T_y) \not= 0$
and the cycle $\mu$ with no entrance in $T_y$ has no entrance in $\bigcup_{(x, z) \in Y}
T_x$ then $z^{\Per(T_y)} \in \overline{\{z^{\Per(T_x)} : (x, z) \in Y\text{ and }T_x =
T_y\}}$.

First observe that if $T_y \not\subseteq \bigcup_{(x,z) \in Y} T_x$ then there is a
neighbourhood $U$ of $y$ that does not intersect $\bigcup_{(x, z) \in Y} [x]$, so for any
harmonious family of bisections $\Bb$ at $y$, the set $\pi\big((U \times \T^k) \cdot
(\Bb^{\ess})^\perp\big)$ is a neighbourhood of $\ker(\pi_{(y, w)})$ that is disjoint from
$\{\ker(\pi_{(x,z)}) : (x,z) \in Y\}$. So we suppose that $T_y \subseteq \bigcup_{(x,z)
\in Y} T_x$ and prove that $\ker(\pi_{(y, w)}) \in \overline{\{\ker(\pi_{(x,z)}) : (x,z)
\in Y\}}$ if and only if, if $\Per(T_y) \not= 0$ and the cycle $\mu$ with no entrance in
$T_y$ has no entrance in $\bigcup_{(x, z) \in Y} T_x$ then $z^{\Per(T_y)} \in
\overline{\{z^{\Per(T_x)} : (x, z) \in Y\text{ and }T_x = T_y\}}$.

First suppose that every cycle in $T_y$ has an entrance. We saw in Example~\ref{eg:graph
groupoids} that in a graph groupoid, $\Iess \setminus E^\infty$ is discrete, so $\J_y =
\Iess_y = \{y\}$. Hence a harmonious family of bisections at $y$ is just an open
neighbourhood $B_y$ of $y$ in $E^\infty$. For any open $V \subseteq \T$, the
corresponding neighbourhood $A(\Bb, V)$ of $\ker(\pi_{(y, w)})$ is $\{\ker(\pi_{(x, z)})
: x \in U, z \in \T\}$, so we see that $\ker(\pi_{(y, w)}) \in
\overline{\{\ker(\pi_{(x,z)}) : (x,z) \in Y\}}$ precisely if $U \cap \bigcup_{(x,z) \in
Y} [x]$ is nonempty for every neighbourhood $U$ of $y$; that is, if and only if $y$ is in
the closure of $\bigcup_{(x,z) \in Y} [x]$, which is precisely if $T_y \subseteq
\bigcup_{(x,z) \in Y} T_x$.

Now suppose that there is a cycle with no entrance in $T_y$, and let $\mu$ be such a
cycle of minimal length. Then $y = \alpha\mu^\infty$ for some $\alpha$ and since
$\ker(\pi_{(\alpha\mu^\infty, z)}) = \ker(\pi_{(\mu^\infty,z)})$ for all $z$, we may as
well assume that $y = \mu^\infty$. We must consider two cases.

Case~1: suppose that $\mu$ has an entrance in $\bigcup_{(x, z) \in Y} T_x$. Then there is
a sequence $(x_n, z_n)$ such that each $[x_n] = [x'_n]$ for some $(x'_n, z_n) \in Y$, and
there is no cycle with an entrance in each $T_{x_n}$. For the harmonious family of
bisections $\Bb$ at $y$ described in~\eqref{eq:BF for graph}, we then have $H_\Bb(x_n) =
\{0\}$ for all $n$. So for any open $U$ containing $y$ and contained in $B_y$ and any
open $V \subseteq \T$, we have $x_n \times \T$ eventually contained in $(U \times V)
\cdot (\Bb(y)^{\ess})^\perp$. Thus $\ker(\pi_{(y, w)}) \in \overline{\{\ker(\pi_{(x,z)})
: (x,z) \in Y\}}$ by \cref{cor:base for topology}.

Case~2: suppose that $\mu$ has no entrance in $\bigcup_{(x, z) \in Y} T_x$. Let $\Bb$ be
the harmonious family of bisections at $y$ described in~\eqref{eq:BF for graph}. Since
$\mu$ has no entrance in any $T_x$, the sets $W_y \coloneqq \{x : [x] = [y]\}$ and $W'
\coloneqq \{x : r(\mu) \not\in T_x\}$ satisfy $\{x : (x, z) \in Y\} = W_y \sqcup W'$. The
set $W_y$ is nonempty because $T_y \not\subseteq \bigcup_{(x,z) \in Y} T_x$. Since $y =
\mu^\infty$, we have $\pi\big((U \times V) \cdot (\Bb^{\ess})^\perp\big) \cap
\{\ker(\pi_{(x, z)}) : x \in W'\} = \varnothing$, and so $\ker(\pi_{(y, w)}) \in
\overline{\{\ker(\pi_{(x,z)}) : (x,z) \in Y\}}$ if and only if
\begin{align*}
\ker(\pi_{(y, w)})
    &\in \overline{\{\ker(\pi_{(x,z)}) : (x,z) \in Y\text{ and }x \in W_y\}}\\
    &= \overline{\{\ker(\pi_{(y,z)}) : (x,z) \in Y\text{ and }x \in W_y\}}.
\end{align*}
The set $\{\ker(\pi_{(y, z)}) : z \in \T\}$ is homeomorphic to $\T$ via the map induced
by $z \mapsto z^{|\mu|}$, so
\begin{align*}
\ker(\pi_{(y, w)}) &{}\in \overline{\{\ker(\pi_{(y,z)}) : (x,z) \in Y\text{ and }x \in W_y\}} \\
&{}\Longleftrightarrow w^{|\mu|} \in \overline{\{z^{|\mu|} : (x,z) \in Y\text{ and }T_x = T_y\}}
\end{align*}
as required.

\subsection{Crossed products by free abelian groups}\label{subsec:cps}

The primitive ideals of crossed products of spaces by free abelian group are well understood, see e.g. \cite[Theorem 8.39]{Williams2007}.
Here we describe how to recover their structure from our results.

Let $X$ be a second-countable locally compact Hausdorff space and suppose that $T\colon \Z^k \curvearrowright X$ is an action by homeomorphisms of $X$.
The groupoid $G_T \coloneqq \{(x, n, y) \in X \times \Z^k \times X : T^n(y) = x\}$ is identical to the Deaconu--Renault groupoid of the restriction of $T$ to $\N^k$,
and the map $(x, n, y) \mapsto (n,y)$ is an isomorphism of $G_T$ onto $\Z^k \times X$.

Note that if $y \in \overline{[x]}$, then $T^{n_i} x \to y$ for some sequence $(n_i)_i$, so $T^{-n_i} y \to x$.
This shows that $\overline{[x]} = \overline{[y]}$ if and only if $y \in \overline{[x]}$.
Moreover, if $T^n x = x$, then for every $m \in \Z^k$, we have $T^n(T^m x) = T^m T^n x = T^m x$,
and so $T^n y = y$ for all $y \in [x]$.
By continuity, it follows that $T^n y = y$ for all $y \in \overline{[x]}$.
Therefore, $\Iess(G_T) = \I(G_T)$ which is closed so $\J_x = \Iess_x = \I(G_T)_x$ for every unit $x\in X$.
In particular, all isotropy is essential isotropy.

For every periodic point $x\in X$ and every $n\in \Z^k$ with $T^n x = x$, the sets
\[
  B^x_{(x, n, x)} \coloneqq \{(T^n y, n, y) : y \in X\}
\]
constitute a harmonious family of bisections at $x$, cf.~\cref{lem:commuting-bisections}.
Hence $T$ admits harmonious families of bisections.

We write $\operatorname{Stab}(x) = \{n \in \Z^k : T^nx = x\}$ for the
stabiliser group. For $x \in X$ and $z \in \T^k$, let $\pi_{(x,z)}$ be as in
\cref{ntn:pi}. The primitive ideals of $C^*(G_T) \cong C_0(X) \rtimes \Z^k$
are the kernels $\{\ker(\pi_{(x, z)}) : x \in X, z \in \Z\}$, and we have
$\ker(\pi_{(x, w)}) = \ker(\pi_{(y, z)})$ precisely if $y \in \overline{[x]}$
and $w^n = z^n$ for all $n \in \Stab(x)$. A neighbourhood base at
$\ker(\pi_{(x, z)})$ is described by \cref{thm:Aopen,thm:neighbourhood-basis}
as follows. Given open neighbourhoods $x \in U \subseteq X$ and $z \in V
\subseteq \T^k$, the set
\[
A^x(U, V) = \{\ker(\pi_{(y, w)}) : y \in U\text{ and }w^n = z^n\text{ for all }n \in \Stab(y) \cap \Stab(x)\}
\]
is an open neighbourhood of $\ker(\pi_{(x, z)})$, and the collection
\[
  \{A^x(U, V) : x \in U \subseteq_{\text{open}} X\text{ and }z \in V\subseteq_{\text{open}} \T^k\}
\]
is a neighbourhood base at $x$.

By \cref{thm:necsuff for convergence}, a sequence $\big( \ker(\pi_{(x_i, z_i)}) \big)_i$ of primitive
ideals of $C_0(X) \rtimes \Z^k$ converges to $\ker(\pi_{(x,z)})$ if and only if
for every subsequence $\big(\ker(\pi_{(x_{i_j}, z_{i_j})}) \big)_j$ there is a sequence $(y_j, w_j)_j$ in $X\times \T^k$ such that
\begin{enumerate}
\item $y_j \in \overline{[x_{i_j}]}$ and $(w_j)^n = z_{i_j}^n$ for all $n \in \Stab(x_{i_j})$ for all $j$;
\item $y_j \to x$; and
\item $z_i$ converges to $z$ along $\big(\Stab(x_i) \cap \Stab(x)\big)_i$ in the sense of \cref{dfn:converges along}.
\end{enumerate}

\subsection{Singly generated dynamical systems}\label{sec:Katsura stuff}
This paper was fairly advanced when Katsura's work \cite{Katsura2021} was posted to the arXiv.
In it, Katsura completely describes the ideals of the $C^*$-algebra of a singly generated dynamical system,
which are merely (partially defined) local homeomorphisms on a locally compact Hausdorff space,
and this had been one of our main objectives.
Katsura applied the technology of $C^*$-correspondences and his adaptation of Cuntz--Pimsner algebras \cite{Katsura2004} to define the $C^*$-algebras of topological graphs.
Here we outline how we can recover his results in the setting of globally defined local homeomorphisms of second-countable locally compact Hausdorff spaces.

Let $X$ be a second-countable locally compact Hausdorff space and suppose that $T\colon \N\curvearrowright X$ is an action by local homeomorphisms.
We will use ideas and terminology from our previous work \cite{BriCarSim:sandwich}.
Our result \cite[Theorem~3.5]{BriCarSim:sandwich} shows that the problem of describing the ideal structure of $C^*(G_T)$
can be reduced to considering nested open invariant subsets $U \subseteq V \subseteq X$
and describing all ideals $J$ of $C^*((G_T)|_{V \setminus U})$ satisfying $J \cap C_0(V \setminus U) = \{0\}$ and $\supp(J) = V \setminus U$.
We called such ideals purely non-dynamical with full support.

Since $U$ and $V$ are invariant, so is $V \setminus U$, and $(G_T)|_{V \setminus U}$ is
precisely the Deaconu--Renault groupoid of the restricted action $T \colon \N
\curvearrowright (V \setminus U)$. So we may as well replace $X$ with $V \setminus U$;
our task is then to describe all of the ideals $J$ of $C^*(G_T)$ such that $J \cap C_0(X)
= \{0\}$ and $\supp(J) = G_T$.
Note that the open invariant sets in our sense are
precisely the $T$-invariant sets in the sense of \cite[Definition~3.1]{Katsura2021}.

Assume that $G_T$ admits a purely non-dynamical ideal $J$ with full support.
By \cite[Lemma~5.2]{BriCarSim:sandwich} (cf. \cite[Example~7.6]{AraLolk}) the groupoid $G_T$
is jointly effective where it is effective. For each $p \in \N$, let
\[
\mathcal{P}_p \coloneqq \bigcup\{V \subseteq X : V\text{ is open and } T^p(y)= y\text{ for all }y \in V\},
\]
and let
\[
P = \bigcup_{p > 0} P_p.
\]
Then \cite[Lemma~5.1]{BriCarSim:sandwich} shows that complement in $X$ of the points at
which $G_T$ is effective in the sense of \cite[Definition~4.1]{BriCarSim:sandwich} is
precisely $P' \coloneqq \{x \in X : [x] \cap P \not= \varnothing\}$. Hence
\cite[Theorem~4.12]{BriCarSim:sandwich} (cf. \cite[Theorem~7.12]{AraLolk}) implies that
$J$ is contained in the ideal $I_{P'}$. Since $\supp(J) = G_T$ we deduce that $P' = X$,
and so $P$ is a full open subset of $X$. Hence by \cite[Example~2.7]{MuhRenWil}, the
inclusion $C^*(G_T|_P) \hookrightarrow C^*(G_T)$ defines a Morita equivalence, and in
particular induces a bijection between ideals of $C^*(G_T|_P)$ and $C^*(G_T)$.

For each $P_n$, the map $T^{n-1}$ is an inverse for $T|_{P_n}$, so $T$ is injective on
each $P_n$ and hence on $P$. Since $T$ is a local homeomorphism, it follows that it
restricts to a homeomorphism of $P$. So $G_T|_P$ is the transformation groupoid
associated to the homeomorphism $T|_P$ on $P$, so it admits harmonious families of
bisections, and its ideals are described by \cref{subsec:cps} above.

\begin{remark}
To relate our discussion to Katsura's work, observe that given an action $T
\colon \N \curvearrowright X$, the representations $\pi_{x,z}$ of $C^*(G_T)$
from \cite{Sims-Williams}, whose kernels coincide with those of the
representations $\pi_{(x,z)}$ of \cref{ntn:pi} by \cref{rmk:H=Zk}, are
precisely the representations described by Katsura in
\cite[Definition~2.6]{Katsura2021}. For an ideal $I$ of $C^*(G_T)$, the set
$W = \{(x,z) : I \not\subseteq \ker(\pi_{(x,z)})\}$ is the complement of the
set $Y = Y_I$ of \cite[Definition~5.1]{Katsura2021}. So \cref{prp:U-V for DR}
shows that the sandwich sets $U$ and $V$ of
\cite[Lemma~3.3]{BriCarSim:sandwich} are the sets $\{x : Y_x = \varnothing\}$
and $\{x : Y_x = \T\}$, respectively, that feature in Katsura's definition
\cite[Definition~5.7]{Katsura2021}.
\end{remark}

\subsection{Rank-two graphs} \label{sec:rank2}
A significant goal for this project was to obtain a systematic approach to calculating
the primitive-ideal spaces of the $C^*$-algebras of $k$-graphs, and in particular to
explicitly calculate the primitive ideal spaces of $C^*$-algebras of $2$-graphs. In this
section, we show that $2$-graph groupoids always admit harmonious families of bisections,
and therefore that \cref{prp:ideal generators} and \cref{thm:necsuff for convergence} can
be used to compute the ideal spaces of arbitrary $2$-graph algebras. This is the first
general computation of the ideal lattices of $2$-graph $C^*$-algebras.

Let $\Lambda$ be a row-finite rank-two graph with no sources as
in~\cite{Kumjian-Pask2000}. The infinite path space $X = \Lambda^\infty$ is
second-countable and locally compact and Hausdorff, and the two translation operations
$T= (T^{\varepsilon_1}, T^{\varepsilon_2})$ are local homeomorphisms~\cite[Remarks
2.5]{Kumjian-Pask2000}, so $T\colon \N^2\curvearrowright X$ is an action of commuting
local homeomorphisms. We will show that any point in the path space admits a harmonious
family of bisections.

Fix $x\in X$. If $\J_x = \{x\}$ then $\mathcal{B} = \{B_x\}$ where $B_x = X$ is trivially
a harmonious family of bisections at $x$. Suppose that $\J_x$ is a rank-one subgroup of
$\Z^2$ with generator $h\in \Z^2$. Let $B_h \subset X\times \{h\}\times X$ be any compact
open bisection containing $(x,h,x)$. Then the association $nh \mapsto (B_h)^n$ defines a
harmonious family of bisections at $x$.

Now suppose that $\J_x$ is a rank-two subgroup of $\Z^2$. By
\cref{lem:positive-generation}, there exists $h_1, h_2\in \coc(\J_x) \cap \N^2$ that
generate $\coc(\J_x)$. Fix $n \in \N^2$ such that
\begin{equation} \label{eq:2-graph-n}
  T^{h_1 + n} x = T^n x = T^{h_2 + n} x.
\end{equation}
We first construct an explicit harmonious family of bisections in the situation where $n
= 0$.

Consider the open sets $U_i \coloneqq Z(x_{[0,h_i]})$ in $X$ and define open bisections
\begin{equation} \label{eq:2-graph-Bf-generators}
  B_{h_i} \coloneqq Z(U_i, h_i, 0, X)
\end{equation}
which contain $(x, h_i, x)$, for $i = 1,2$.

We record some technical properties.

\begin{lemma}\label{lem:2-graph setup}
  The bisections defined in~\labelcref{eq:2-graph-Bf-generators} commute, and they satisfy
 \begin{align}
    B_{h_2} B_{h_1}^{-1} &\subset B_{h_1}^{-1} B_{h_2}, \label{eq:2-graph-Bf-1} \\
    B_{h_1} B_{h_2}^{-1} &\subset B_{h_2}^{-1} B_{h_1}. \label{eq:2-graph-Bf-2}
  \end{align}
\end{lemma}
\begin{proof}
We first verify that the $B_{h_i}$ commute.
Take composable elements $(y_1, h_1, y_2)\in B_{h_1}$ and $(y_2, h_2, y_3)\in B_{h_2}$ so that
\[
  (y_1, h_1 + h_2, y_3) = (y_1, h_1, y_2)(y_2, h_2, y_3) \in B_{h_1} B_{h_2}.
\]
This means that $T^{h_1}y_1 = y_2 \in U_1$ and $y_3 = T^{h_2}y_2 = T^{h_1 + h_2} y_1$.
In particular, $y_1 \in Z(x_{[0, h_1 + h_2]})$ so by the factorisation property we have $T^{h_2}y_1 \in U_1$.
It now follows that
\[
  (y_1, h_1 + h_2, y_3) = (y_1, h_2, T^{h_2}y_1) (T^{h_2}y_1, h_1, T^{h_1+h_2}y_1) \in B_{h_2} B_{h_1}
\]
showing that $B_{h_1} B_{h_2} \subset B_{h_2} B_{h_1}$.
A similar argument shows that $B_{h_2} B_{h_1} \subset B_{h_1} B_{h_2}$.

Next we verify~\cref{eq:2-graph-Bf-1}; a similar argument applies to~\cref{eq:2-graph-Bf-2}.
Take composable elements $(y_1, h_1, h_2)\in B_{h_1}$ and $(y_2, -h_2, y_3)\in B_{h_2}$ so that
\[
  (y_1, h_1 - h_2, y_3) = (y_1, h_1, h_2)(y_2, -h_2, y_3) \in B_{h_1} B_{h_2}^{-1}.
\]
In particular, $y_1 = x_{[0, h_1]} y_2$ and $y_3 = x_{[0,h_2]} y_2$.
Consider $z\coloneqq x_{[0, h_1 + h_2]} y_2$ and observe that $(y_1, -h_2, z)\in {B_{h_2}}^{-1}$ and $(z, h_1, y_3)\in B_{h_1}$ so we have
\[
  (y_1, h_1 - h_2, y_3) = (y_1, -h_2, z) (z, h_1, y_3) \in  B_{h_2}^{-1} B_{h_1},
\]
giving $B_{h_1} B_{h_2}^{-1} \subset B_{h_2}^{-1} B_{h_1}$.
\end{proof}

\begin{corollary}
Let $B_{h_1}$ and $B_{h_2}$ be as in~\labelcref{eq:2-graph-Bf-generators}. For $h\in \J_x$,
write $h = m_1(h) h_1 + m_2(h) h_2$ with $m_1(h), m_2(h) \in \Z$. Define
\begin{equation} \label{eq:2-graph-Bf}
  B_h \coloneqq
  \begin{cases}
    B_{h_1}^{m_1(h)} B_{h_2}^{m_2(h)} &\textrm{if}~m_1(h)\leq m_2 \\
    B_{h_2}^{m_2(h)} B_{h_1}^{m_1(h)} &\textrm{if}~m_2(h) < m_1,
  \end{cases}
\end{equation}
with the convention that $B_{h_i}^0 = \Lambda^\infty$. Then $B_{(x,h,x)} \coloneqq B_h$
defines a harmonious family of bisections at $x$.
\end{corollary}
\begin{proof}
This follows immediately from \cref{lem:2-graph setup} and
\cref{lem:relatively commuting}.
\end{proof}

\begin{remark}
The argument above breaks into two parts corresponding to the situation where $\J_x$ is
singly generated and where $\J_x$ is generated as a group by its intersection with
$c^{-1}(\N^k)$. The point of restricting to $2$-graphs is that these two cases cover all
possibilities thanks to \cref{lem:positive-generation}. In general, the arguments
presented above could be run verbatim to prove the following statement: Let $\Lambda$ be
a $k$-graph and suppose that for each $x \in \Lambda^\infty$, the group $\J_x$ either has
rank~1, or is generated by its intersection with $c^{-1}(\N^k)$. Then $G_\Lambda$ admits
harmonious families of bisections. In particular, if $\operatorname{rank}(\J_x) \in \{0,
1, k\}$ for all $x$, then $G_\Lambda$ admits harmonious families of bisections.
\end{remark}

\subsection{A countable subshift}
The following example is an adaptation of~\cite[Example 7.9]{AraLolk}. Ara and Lolk show
that this example is not relatively strongly topologically free (cf.~\cite[Definition
7.4]{AraLolk}) and hence not amenable to their techniques for studying ideal structure.
This example does however admit harmonious families of bisections.

For $n\in \N$, define $x_n, y_n\colon \N^2 \to \{0,1\}$ by
\[
  x_n(a,b) =
  \begin{cases}
    1 & \textrm{if}~b\leq n, \\
    0 & \textrm{if}~b > n
  \end{cases} \quad \textrm{and} \quad
  y_n(a,b) =
  \begin{cases}
    1 & \textrm{if}~a\leq n, \\
    0 & \textrm{if}~a > n
  \end{cases}
\]
for all $(a,b)\in \N^2$.
Let $\bar{0},\bar{1}\colon \N^2 \to \{0,1\}$ be constantly $0$ and $1$, respectively.
Then $(x_n)_n$ and $(y_n)_n$ converge to $\bar{1}$ as $n\to \infty$,
and $X = \{ \bar{0}, \bar{1}, x_n, y_n : n\in \N\}$ is a compact and Hausdorff subspace of $\{0,1\}^{\N^2}$.
The coordinate translations on $\{0,1\}^{\N^2}$ restrict to an action $T\colon \N^2\curvearrowright X$ given by
\[
  T^{(k,l)} x(a,b) = x(a+k, b+l),
\]
for all $(k,l)\in \N^2$, $(a,b)\in \N^2$, and $x\in X$.
Note that $\bar{0}$ and $\bar{1}$ are fixed points of $T$, and that $T^{(1,0)}(x_{n+1}) = x_n$ and $T^{(1,0)}y_n = y_n$,
and $T^{(0,1)}x_n = x_n$ and $T^{(0,1)} y_{n+1} = g_n$, for all $n\in \N$.

The only point in $X$ that is not isolated is $\bar{1}$.
The set
\[
  U_{(1,0)} = \{ \bar{0},\bar{1}, x_{n+1}, y_n : n\in \N\}
\]
is an open neighborhood of $\bar{1}$ in $X$ and $T^{(1,0)}$ restricts to a homeomorphism on $U_{(1,0)}$;
similarly,
\[
  U_{(0,1)} = \{ \bar{0}, \bar{1}, x_n, y_{n+1} : n\in \N\}
\]
is an open neighborhood of $\bar{1}$ in $X$, and $T^{(0,1)}$ restricts to a homeomorphism on $U_{(0,1)}$.
This shows that $T\colon \N^2\curvearrowright X$ is an action by local homeomorphisms.

We will now verify that $T$ admits harmonious families of bisections. Every isolated
point admits a harmonious family of bisections so we only need to construct one for
$\bar{1}$. The sets
\[
  B_{(1,0)} = Z(U_{(1,0)}, (1,0), (0,0), X) \quad \textrm{and} \quad
  B_{(0,1)} = Z(U_{(0,1)}, (0,1), (0,0), X)
\]
are homogeneous compact open bisections that contain $(\bar{1}, (1,0), \bar{1})$ and
$(\bar{1},(0,1), \bar{1})$, respectively. In fact, the bisections satisfy the conditions
of \cref{lem:relatively commuting}, so $\bar{1}$ admits a harmonious family of
bisections.

\begin{remark}
While we have proved that the preceding example admits harmonious families of bisections
by constructing them by hand, we could also establish this by proving that it is
conjugate to the shift space of a $2$-graph as follows, and then using the general
results of \cref{sec:rank2} above.

Consider the 2-coloured graph
\[
\begin{tikzpicture}[>=stealth, decoration={markings, mark=at position 0.5 with {\arrow{>}}}]
    \node[circle, inner sep=0pt] (v) at (0,0) {$v$};
    \node[circle, inner sep=0pt] (w) at (2,0) {$w$};
    \draw[red, postaction=decorate] (v) .. controls +(-1, 1) and +(1, 1) .. (v) node[pos=0.5, above] {$e$};
    \draw[blue, postaction=decorate] (v) .. controls +(-1, -1) and +(1, -1) .. (v) node[pos=0.5, below] {$a$};
    \draw[red, postaction=decorate, in=30, out=150] (w) to node[pos=0.5, above] {$f$} (v);
    \draw[blue, postaction=decorate, in=330, out=210] (w) to node[pos=0.5, below] {$b$} (v);
    \draw[red, postaction=decorate] (w) .. controls +(-1, 1) and +(1, 1) .. (w) node[pos=0.5, above] {$g$};
    \draw[blue, postaction=decorate] (w) .. controls +(-1, -1) and +(1, -1) .. (w) node[pos=0.5, below] {$c$};
\end{tikzpicture}
\]
By \cite[Section~6]{Kumjian-Pask2000}, there is a unique $2$-graph $\Lambda$ with this
skeleton and satisfying the factorisation rules
\[
ae = ea,\quad af = fc,\quad bg = eb,\quad\text{and}\quad cg=gc.
\]
Writing $T_\Lambda \colon \N^2 \curvearrowright \Lambda^\infty$ for the shift action, it is
routine to check that the formulas
\[
\theta(x_v) = \overline{1}, \quad
\theta(a^n b x_w) = y_{n+1}, \quad
\theta(e^n f x_w) = x_{n+1}, \quad\text{and}\quad
\theta(x_w) = \overline{0},
\]
define a conjugacy $\theta$ from $(\Lambda^\infty, T_\Lambda)$ to $(X, T)$.

\end{remark}

\section{Ideals of higher-rank graphs}\label{sec:hrgs}

In this section, we present a catalogue of all ideals in the $C^*$-algebra of a
higher-rank graph $\Lambda$ whose groupoid admits harmonious families of bisections, in
terms of open subsets $D$ of $\Lambda^0 \times \T^k$.
This extends the catalogue of ideals that are gauge-equivariant.
Although the condition that characterises which subsets $D \subseteq \Lambda^0 \times \T^k$
correspond to ideals is quite technical, in practice the description is usable,
as we show by example at the end of the section.

Our strategy is as follows. We begin by describing the collection $\mathcal{A}_\Lambda$
of subsets of $\Lambda^\infty \times \T^k$ that correspond to ideals of $C^*(\Lambda)$
via \cref{cor:open-if-and-only-if}. We then describe maps $\delta$ and $\alpha$, the
first of which carries subsets of $\Lambda^\infty \times \T^k$ to subsets of $\Lambda^0
\times \T^k$ and the second of which takes subsets of $\Lambda^0 \times \T^k$ to subsets
of $\Lambda^\infty \times \T^k$. We show that $\alpha\circ\delta$ is the identity map on
$\mathcal{A}_\Lambda$, and deduce that the ideals of $C^*(\Lambda)$ are indexed by the
elements of $\delta(\mathcal{A}_\Lambda)$. We then identify exactly which subsets of
$\Lambda^0 \times \T^k$ belong to $\delta(\mathcal{A}_\Lambda)$.

Recall from the beginning of \cref{sec:necessary condition} that if $\Bb =
(B_\gamma)_{\gamma \in \J_x}$ is a harmonious family of bisections based at $x \in
\Lambda^\infty$, then $\Bb^{\ess}$ denotes its intersection with the essential isotropy
of $G_\Lambda$, and that $(\Bb^{\ess})^\perp$ is the group bundle $\{(y,z) \in B_x \times
\T^k : z^{\coc(\gamma)} = 1\text{ for all } \gamma \in \Bb^{\ess}_y\}$, which is a
sub-bundle of the group bundle $B_x \times \T^k$.

\begin{notation}
Let $\Lambda$ be a row-finite higher-rank graph with no sources whose groupoid
$G_\Lambda$ admits harmonious families of bisections. We write $\mathcal{A}_\Lambda$ for
the collection of all subsets $A \subseteq \Lambda^\infty \times \T^k$ such that
\begin{enumerate}[label=(A\arabic*)]
  \item\label{(A1)} if $(x,z)\in A$, and $(x',z') \in X \times \T^k$ satisfy
    $\overline{[x]}=\overline{[x']}$, and $(\overline{z}z')^{\coc(\gamma)} = 1$ for all
    $\gamma \in \Iess_x$, then $(x',z')\in A$; and
  \item \label{(A2)} for each $(x,z)\in A$, then there exist a harmonious family of
      bisections $\Bb = (B_g)_{g\in\J_x}$ based at $x$ and an open neighbourhood
      $V\subseteq\T^k$ of $z$ such that $(B_x \times V) \cdot (\Bb^{\ess})^\perp
      \subseteq A$.
\end{enumerate}
\end{notation}

It follows from \cref{cor:open-if-and-only-if} applied to the $k$-graph
groupoid $G_\Lambda$ associated to the shift action $T \colon \N^k
\curvearrowright \Lambda^\infty$ of \cite{Kumjian-Pask2000} that the map
$\pi$ of \cref{ntn:pi} is a bijection from $\mathcal{A}_\Lambda$ to the set
of open subsets of $\Prim(C^*(G_\Lambda))$.

For the following, recall that if $x \in \Lambda^\infty$ and $n \in \N^k$, then $x(n)$
denotes the vertex $x(n) = r(T^n(x)) \in \Lambda^0$

\begin{lemma}\label{lem:A containment}
Let $\Lambda$ be a row-finite higher-rank graph with no sources whose groupoid
$G_\Lambda$ admits harmonious families of bisections. Suppose that $A \in
\mathcal{A}_\Lambda$. For $x \in \Lambda^\infty$, $n \in \N^k$ and $z \in \T^k$ we have
$(x,z) \in A$ if and only if $(T^n(x),z) \in A$. If $(x,z) \in A$ then there exist $n \in
\N^k$ and an open neighbourhood $V \subseteq \T^k$ of $z$ such that $Z(x(n)) \times V
\subseteq A$.
\end{lemma}
\begin{proof}
  For $x \in \Lambda^\infty$, we have $[x] = [T^n(x)]$. Hence~\labelcref{(A1)} implies that
$(x,z) \in A$ if and only if $(T^n(x),z) \in A$.

Suppose that $(x,z) \in A$. Since $\pi$ is continuous by \cref{cor:pi-continuous}, the set
$A \subseteq X \times \T^k$ is open. So there exist $n \in \N^k$ and an open neighbourhood
$V \subseteq \T^k$ of $z$ such that $Z(x(0,n)) \times V \subseteq A$. Since $[x(0,n) y] =
[y]$ for all $y \in Z(v)$, it follows from \labelcref{(A2)} that $Z(x(n)) \times V
\subseteq A$.
\end{proof}

\begin{lemma}\label{lem:alphabeta}
Let $\Lambda$ be a row-finite higher-rank graph with no sources whose groupoid
$G_\Lambda$ admits harmonious families of bisections. Given a subset $A \subseteq
\Lambda^\infty \times \T^k$, define
\[
\delta(A)\coloneqq \{(v,z)\in\Lambda^0\times\T^k : Z(v)\times\{z\}\subseteq A\}.
\]
Given a subset $D \subseteq \Lambda^0 \times \T^k$, define
\[
\alpha(D) \coloneqq \{(x,z) \in \Lambda^\infty \times \T^k : (x(n), z) \in D\text{ for some }n \in \N^k\}.
\]
Then $\alpha(\delta(A)) = A$ for all $A \in \mathcal{A}_\Lambda$.
\end{lemma}
\begin{proof}
Fix $(x,z) \in A$. By \cref{lem:A containment}, there exist $n \in \N^k$ and an open
neighbourhood $V$ of $z$ such that $Z(x(n)) \times V \subseteq A$. We then have $(x(n),
z) \in x(n) \times V \subseteq \delta(A)$. By definition of $\alpha$, we then have $(y,z)
\in \alpha(\delta(A))$ whenever $y \in \Lambda^\infty$ satisfies $x(n) \Lambda y(m) \not=
\varnothing$ for some $m \in \N^k$. Taking $m = n$ and $y = x$ we see that $(x, z) \in
\alpha(\delta(A))$. That is $A \subseteq \alpha(\delta(A))$.

For the reverse containment, fix $(x,z) \in \alpha(\delta(A))$. Then there exists $n \in
\N^k$ such that $(x(n), z) \in \delta(A)$. Hence $Z(x(n)) \times \{z\} \subseteq A$. In
particular, $(\sigma^n(x), z) \in A$. Since $[\sigma^n(x)] = [x]$ and since $A$
satisfies~\labelcref{(A1)}, we deduce that $(x,z) \in A$. Hence $\alpha(\delta(A)) \subseteq A$.
\end{proof}

To describe the lattice of ideals of $C^*(\Lambda)$, it therefore suffices to describe
the range of the map $\delta \colon \mathcal{A}_\Lambda \to \mathcal{P}(\Lambda^0 \times \T^k)$. We
start with some elementary necessary conditions for membership of $\delta(\mathcal{A}_\Lambda)$.

\begin{lemma}\label{lem:easy conditions}
Let $\Lambda$ be a row-finite higher-rank graph with no sources whose groupoid
$G_\Lambda$ admits harmonious families of bisections. If $A \in \mathcal{A}_\Lambda$,
then
\begin{enumerate}
\item for $\lambda \in \Lambda$, if $(r(\lambda), z) \in \delta(A)$, then $(s(\lambda),
    z) \in \delta(A)$;
\item for $n \in \N^k$, $v \in \Lambda^0$ and $z \in \T^k$, if $(s(\lambda), z) \in
    \delta(A)$ for all $\lambda \in v\Lambda^n$, then $(v, z) \in \delta(A)$; and
\item for each $v \in \Lambda^0$, the set $\{z \in \T^k : (v,z) \in \delta(A)\}$ is
    open.
\end{enumerate}
\end{lemma}
\begin{proof}
(1) Suppose that $(r(\lambda), z) \in \delta(A)$. Then $Z(\lambda) \times \{z\} \subseteq
Z(r(\lambda)) \times \{z\} \subseteq A$. Since $\overline{[\sigma^{d(\lambda)}(x)]} =
\overline{[x]}$ for all $x \in \Lambda^\infty$, and since $Z(s(\lambda)) =
\sigma^{d(\lambda)}(Z(\lambda))$, condition~\labelcref{(A1)} implies that $Z(s(\lambda)) \times \{z\}
\subseteq A$, and hence $(s(\lambda), z) \in \delta(A)$.

(2) Fix $n \in \N^k$, $v \in \Lambda^0$ and $z \in \T^k$ and suppose that $(s(\lambda),
z) \in \delta(A)$ for all $\lambda \in v\Lambda^n$. Then $Z(s(\lambda)) \times \{z\}
\subseteq A$ for all $\lambda \in v\Lambda^n$. For each $\lambda \in v\Lambda^n$ and $x
\in Z(s(\lambda))$, we have $\overline{[\lambda x]} = \overline{[x]}$. Hence
condition~\labelcref{(A1)} implies that $\bigcup_{\lambda \in v\Lambda^n} (Z(\lambda) \times \{z\}) =
\{(\lambda x, z) : \lambda \in \Lambda^n, x \in Z(s(\lambda))\} \subseteq A$. Since $Z(v)
= \bigcup_{\lambda \in v\Lambda^n} Z(\lambda)$, it follows that $Z(v) \times \{z\}
\subseteq A$ and hence $(v, z) \in \delta(A)$.

(3) Fix $v \in \Lambda^0$ and $z \in \T^k$ such that $(v, z) \in \delta(A)$. We must show
that there exists an open neighbourhood $V \subseteq \T^k$ of $z$ such that $\{v\} \times V
\subseteq \delta(A)$. Since $(v,z) \in \delta(A)$, we have $Z(v) \times \{z\} \subseteq
A$. For each $x \in Z(v)$, \cref{lem:A containment} gives $n_x \in \N^k$ and an open
neighbourhood $V_x \subseteq \T^k$ of $z$ such that $Z(x(0, n_x)) \times V_x \subseteq
A$. The sets $Z(x(0,n_x))$ constitute an open cover of $Z(v)$, which is compact. So there
is a finite $F \subseteq Z(v)$ such that $\bigcup_{x \in F} Z(x(0,n_x)) = Z(v)$. Let $V
= \bigcap_{x \in F} V_x$. This is an open neighbourhood around $z$, and
$Z(v) \times V = \bigcup_{x \in F} (Z(x(0,n_x)) \times V) \subseteq A$.
Thus $\{v\} \times V \subseteq \delta(A)$.
\end{proof}

\begin{lemma}\label{lem:betaalpha}
Let $\Lambda$ be a row-finite higher-rank graph with no sources whose groupoid
$G_\Lambda$ admits harmonious families of bisections. Let $D$ be subset of $\Lambda^0
\times \T^k$. Then $D \subseteq \delta(\alpha(D))$. If $D$ satisfies conditions
(1)~and~(2) of \cref{lem:easy conditions}, then $D = \delta(\alpha(D))$.
\end{lemma}
\begin{proof}
For the first statement, fix $(v,z) \in D$. For each $x \in Z(v)$, we have $(x(0), z) \in
D$ and so by definition of $\alpha(D)$ we have $Z(v) \times \{z\} \subseteq \alpha(D)$.
So by definition of $\delta$, we have $(v,z) \in \delta(\alpha(D))$.

For the second statement, suppose that $D$ satisfies conditions (1)~and~(2) of
\cref{lem:easy conditions}. We must show that $\delta(\alpha(D)) \subseteq D$. Fix $(v,z)
\in \delta(\alpha(D))$. Then $Z(v) \times \{z\} \subseteq \alpha(D)$, and so for each $x
\in Z(v)$, there exists $n_x \in \N^k$ such that $(x(n_x), z) \in D$. The sets $Z(x(0,
n_x))_{x \in Z(v)}$ are an open cover of $Z(v)$ in $\Lambda^\infty$, and since $Z(v)$ is
compact, it follows that there is a finite subset $F \subseteq Z(v)$ such that $Z(v) =
\bigcup_{x \in F} Z(x(0,n_x))$. Let $n = \bigvee_{x \in F} n_x$. For each $x \in F$ we
have $Z(x(0,n_x)) = \bigcup_{\tau \in x(n_x)\Lambda^{n-n_x}} Z(x(0,n_x)\tau)$. Since each
$x(0,n_x)\tau$ belongs to $v\Lambda^n$ and since the sets $Z(\eta), \eta \in \Lambda^n$
are mutually disjoint nonempty subsets of $Z(v)$, we deduce that $\{x(0, n_x)\tau : x \in
F, \tau \in x(n_x)\Lambda^{n - n_x}\} = v\Lambda^n$. For $x \in F$ and $\tau \in
x(n_x)\Lambda^{n - n_x}$, the fact that $D$ satisfies \cref{lem:easy conditions}(1) implies that
$(s(\tau), z) \in D$, so $\{(s(\eta), z) : \eta \in v\Lambda^n\} \subseteq D$. Since
$D$ satisfies \cref{lem:easy conditions}(2), we obtain $(v, z) \in D$.
\end{proof}

We now characterise the sets $D \subseteq \Lambda^0 \times \T^k$ that have the form
$\delta(A)$ for some $A \in \mathcal{A}_\Lambda$. The condition is technical, but we will
demonstrate by example that it can be checked in practice.

Throughout what follows, given a harmonious family of bisections $\Bb =
(B_\gamma)_{\gamma \in \J_x}$ based at $x \in \Lambda^\infty$, and given $y \in B_x$, we
write $(\Bb^{\ess}_y)^\perp$ for the subgroup
\[
(\Bb^{\ess}_y)^\perp \coloneqq
    \{z \in \T^k : z^{\coc(\gamma)} = 1\text{ for all } \gamma \in \Bb^{\ess}_y\}.
\]
Thus, the fibre $(\Bb^{\ess})^\perp_y$ of $(\Bb^{\ess})^\perp$ over the point $y$ is
$\{y\} \times (\Bb^{\ess}_y)^\perp$.

\begin{proposition}\label{prp:D description}
Let $\Lambda$ be a row-finite higher-rank graph with no sources whose groupoid
$G_\Lambda$ admits harmonious families of bisections. A set $D \subseteq \Lambda^0 \times
\T^k$ has the form $\delta(A)$ for some $A \subseteq \Lambda^\infty \times \T^k$
satisfying \labelcref{(A1),(A2)} if and only if $D$ satisfies (1)~and~(2) of
\cref{lem:easy conditions} and for every $(v,z) \in D$ and every $x \in Z(v)$, there
exist a bisection family $\Bb = (B_\gamma)_{\gamma \in \J_x}$ based at $x$ and an open
neighbourhood $V$ of $z$ in $\T^k$ such that for every $y \in B_x$ there exists $n \in
\N^k$ such that $\{y(n)\} \times V(\Bb^{\ess}_y)^\perp \subseteq D$.

In particular, $\alpha$ is a bijection from the collection of such sets $D$ to $\mathcal{A}_\Lambda$.
\end{proposition}
\begin{proof}
If $D = \delta(A)$, then $D$ satisfies (1)--(2) by \cref{lem:easy conditions}.
Fix $(v,z) \in D$. Since $D = \delta(A)$ we first observe that $Z(v) \times \{z\} \subseteq A$.

Fix $x \in Z(v)$. By condition~\labelcref{(A2)}, there exist a harmonious family of
bisections $\Bb = (B_\gamma)_{\gamma\in\J_x}$ based at $x$ and an open neighbourhood $V_0
\subseteq\T^k$ of $z$ such that $(B_x \times V_0) \cdot (\Bb^{\ess})^\perp \subseteq A$.
Since $\Lambda^\infty$ is totally disconnected, we can fix a compact open neighbourhood
$C$ of $x$ contained in $B_x$ and apply \cref{lem:bisection family transport} to obtain a
harmonious family of bisections $\{CB_\gamma C : \gamma \in \J_x\}$ with $CB_x C$ compact
open. So we can assume without loss of generality that $B_x$ is compact open. Let $K
\subseteq V_0$ be a compact neighbourhood of $z$, and let $V$ be the interior of $K$. We
claim that $\Bb$ and $V$ have the desired properties.

To see this, fix $y \in B_x$. We must find $n \in \N^k$ such that $\{y(n)\} \times
V(\Bb^{\ess}_y)^\perp \subseteq D$. Since $y \in B_x$, we have
\[
\{y\} \times K(\Bb^{\ess}_y)^\perp
    \subseteq \{y\} \times V_0(\Bb^{\ess}_y)^\perp \subseteq A.
\]
By \cref{lem:A containment}, for each $w \in K(\Bb^{\ess}_y)^\perp$, there exists $n_w
\in \N^k$ and a neighbourhood $U_w$ of $w$ such that $Z(y(n_w)) \times U_w \subseteq A$.
Since $K(\Bb^{\ess}_y)^\perp$ is compact, the open cover $\{U_w : w \in
K(\Bb^{\ess}_y)^\perp\}$ admits a finite subcover $\{U_w : w \in F\}$. Since each
$Z(y(n_w)) \times U_w \subseteq A$, each $\{y(n_w)\} \times U_w \subseteq D$. Let $n =
\bigvee_{w \in F} n_w$. Since $D$ satisfies condition~(1) of \cref{lem:easy conditions},
we have $\{y(n)\} \times U_w \subseteq D$ for all $w \in F$. Hence,
\begin{align*}
\{y(n)\} \times V(\Bb^{\ess}_y)^\perp
    &\subseteq \{y(n)\} \times K(\Bb^{\ess}_y)^\perp \\
    &\subseteq \{y(n)\} \times \textstyle\big(\bigcup_{w \in F} U_w\big)
    = \bigcup_{w \in F} (\{y(n)\} \times U_w)
    \subseteq D.
\end{align*}

Now suppose that $D$ satisfies (1)--(2) of \cref{lem:easy conditions} and that for every
$(v,z) \in D$ and every $x \in Z(v)$, there exist a harmonious family of bisections $\Bb
= (B_\gamma)_{\gamma \in \J_x}$ based at $x$ and an open neighbourhood $V$ of $z$ in
$\T^k$ such that for every $y \in B_x$ there exists $n \in \N^k$ such that $\{y(n)\}
\times V(\Bb^{\ess}_y)^\perp \subseteq D$.

By \cref{lem:betaalpha}, we have $D = \delta(\alpha(D))$, so we just need to show that
$\alpha(D)$ belongs to $\mathcal{A}_\Lambda$. For condition~\labelcref{(A1)}, suppose
that $(x,z) \in \alpha(D)$ and that $\overline{[x']} = \overline{[x]}$ and that $z$ and
$z'$ determine the same character of $\coc(\Iess_{x'})$. Since $(x,z) \in \alpha(D)$
there exists $n \in \N^k$ such that $(x(n), z) \in D$. We have $T^n(x) \in [x] \subseteq
\overline{[x]} = \overline{[x']}$. Hence $\overline{[x']} \cap Z(x(n)) \not=
\varnothing$. Since $Z(x(n))$ is open, it follows that $[x'] \cap Z(x(n)) \not=
\varnothing$. Hence there exists $m \in \N^k$ such that $x(n) \Lambda x'(m) \not=
\varnothing$. Since $(x(n), z) \in D$, \cref{lem:easy conditions}(1) ensures that
$(x'(m), z) \in D$. Hence $(x', z) \in \alpha(D)$. By hypothesis on $D$, there exist a
harmonious family of bisections $\Bb$ based at $x'$, a neighbourhood $V$ of $z$ and $p
\in \N^k$ such that $\{x'(p)\} \times V(\Bb^{\ess}_{x'})^\perp \subseteq D$. Since $D$
satisfies condition~(1) of \cref{lem:easy conditions}, we then have $\{x'(p \vee m)\}
\times V(\Bb^{\ess}_{x'})^\perp \subseteq D$, so by replacing $p$ with $p \vee m$ we can
assume that $p \ge m$. By definition we have $\Bb^{\ess}_{x'} = \bigcup \Bb \cap
\Iess_{x'} \subseteq \Iess_{x'}$, so $(\Bb^{\ess}_{x'})^\perp \supseteq
(\Iess_{x'})^\perp$. Since $z$ and $z'$ determine the same characters of
$\coc(\Iess_{x'})$, we therefore have $z' \in z \coc(\Iess_{x'})^\perp \subseteq
V(\Bb^{\ess}_{x'})^\perp$, and hence $(x'(p), z') \in D$. Therefore, $(x', z') \in
\alpha(D)$.

For~\labelcref{(A2)}, fix $(x,z) \in \alpha(D)$. By hypothesis, there exist a harmonious
family of bisections $\Bb$ based at $x$, and a neighbourhood $V$ of $z$ such that for
every $y \in B_x$ there exists $n_y \in \N^k$ such that $\{y(n_y)\} \times
V(\Bb^{\ess}_y)^\perp \subseteq D$. In particular, for $y \in B_x$, we have
\[
\big((B_x \times V) \cdot (\Bb^{\ess})^\perp\big)_y
    = \{y\} \times V(\Bb^{\ess}_y)^\perp
    \subseteq \alpha(D).
\]
That is, $(B_x \times V) \cdot (\Bb^{\ess})^\perp \subseteq \alpha(D)$.

The final statement follows from \cref{lem:alphabeta,lem:betaalpha}.
\end{proof}

The following complete description of the ideal structure of the higher-rank graph
$C^*$-algebra whose groupoid $G_\Lambda$ admits harmonious families of bisections now
follows directly from \cref{cor:open-if-and-only-if,prp:D description}. We identify
$C^*(\Lambda)$ with the groupoid $C^*$-algebra $C^*(G_\Lambda)$.

\begin{corollary}\label{cor:DLambda description}
Let $\Lambda$ be a row-finite higher-rank graph with no sources whose groupoid
$G_\Lambda$ admits harmonious families of bisections. Let $\Dd_\Lambda$ denote the
collection of subsets $D$ of $\Lambda^0 \times \T^k$ such that
\begin{enumerate}[label=(D\arabic*)]
  \item \label{(D1)} for $\lambda \in \Lambda$, if $(r(\lambda), z) \in \delta(A)$, then
    $(s(\lambda), z) \in \delta(A)$;
  \item \label{(D2)} for $n \in \N^k$, $v \in \Lambda^0$ and $z \in \T^k$, if $s(v\Lambda^n)
    \times \{z\} \subseteq \delta(A)$ then $(v, z) \in \delta(A)$; and
  \item \label{(D3)} for every $(v,z) \in D$ and every $x \in Z(v)$, there exist a bisection
    family $\Bb = (B_\gamma)_{\gamma \in \J_x}$ and an open neighbourhood $V$ of $z$ in
    $\T^k$ such that for every $y \in B_x$ there exists $n \in \N^k$ such that
    $\{y(n)\} \times V(\Bb^{\ess}_y)^\perp \subseteq D$.
\end{enumerate}
Let $\alpha \colon \Dd_\Lambda \to \mathcal{A}_\Lambda$ be the restriction of the map of
\cref{lem:alphabeta}, and for $(x,z) \in \Lambda^\infty \times \T^k$,
let $\pi_{(x,z)}$ be the representation $C^*(G_\Lambda)$ of \cref{rmk:H=Zk}.
Then
\[
  D \mapsto \bigcap_{(x,z) \in (\Lambda^\infty \times \T^k)\setminus \alpha(D)} \ker(\pi_{(x,z)})
\]
is a bijection between $\Dd_\Lambda$ and the collection of ideals of $C^*(\Lambda)$.
\end{corollary}

\begin{remark}
The map from $\mathcal{D}_\Lambda$ to the set of ideals of $C^*(\Lambda)$ described above
generalises the well-known map \cite[Theorem~5.2]{RaeSimYee} from saturated hereditary
subsets of $\Lambda^0$ to gauge-invariant ideals of $C^*(\Lambda)$. Specifically, if
$D\in\mathcal{D}_\Lambda$ and $I$ is the corresponding ideal of $C^*(\Lambda)$, then $I$
is gauge-invariant if and only if $D = H\times\T^k$ for some subset
$H\subseteq\Lambda^0$, and then (D1)~and~(D2) say precisely that $H$ is saturated and
hereditary; specifically, it is the saturated hereditary set $\{v\in\Lambda^0 : p_v\in
I\}$. Observe that if the complement of every saturated hereditary subgraph of $\Lambda$
satisfies the aperiodicity condition~(B) of \cite{RaeSimYee}, then $\Iess(G_\Lambda)$ is
just the unit space, so~(D3) says that $\Dd_\Lambda$ consists of sets of the form $H
\times \T^k$ for $H \subseteq \Lambda^0$, and (D1)~and~(D2) say that these sets $H$ are
saturated and hereditary; so we recover \cite[Theorem~5.3]{RaeSimYee}.
\end{remark}

\begin{example}\label{eg:aHNS example}
We illustrate our results by applying them to the $2$-graph $\Lambda$ with the following
skeleton.
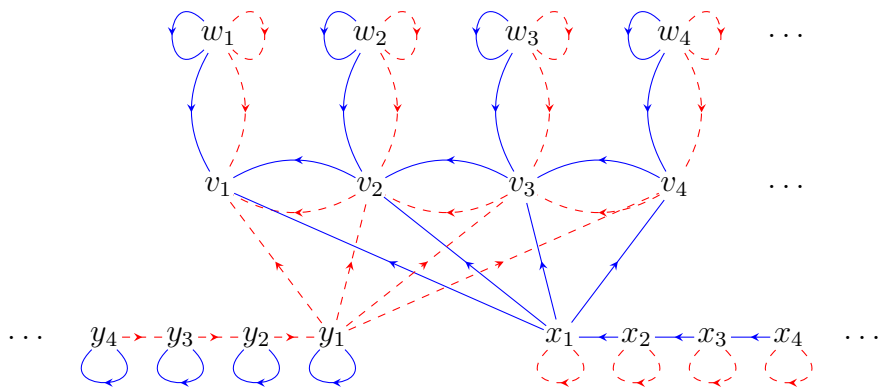
\[
	\begin{tikzpicture}[>=stealth, decoration={markings, mark=at position 0.5 with {\arrow{>}}}]
		\node[circle, inner sep=0pt] (v1) at (0,0) {$v_1$};
		\node[circle, inner sep=0pt] (v2) at (2,0) {$v_2$};
		\node[circle, inner sep=0pt] (v3) at (4,0) {$v_3$};
		\node[circle, inner sep=0pt] (v4) at (6,0) {$v_4$};
		\node[circle, inner sep=0pt] at (7.5,0) {$\dots$};
		\node[circle, inner sep=0pt] (w1) at (0,2) {$w_1$};
		\node[circle, inner sep=0pt] (w2) at (2,2) {$w_2$};
		\node[circle, inner sep=0pt] (w3) at (4,2) {$w_3$};
		\node[circle, inner sep=0pt] (w4) at (6,2) {$w_4$};
		\node[circle, inner sep=0pt] at (7.5,2) {$\dots$};
		\node[circle, inner sep=0pt] (x1) at (4.5,-2) {$x_1$};
		\node[circle, inner sep=0pt] (x2) at (5.5,-2) {$x_2$};
		\node[circle, inner sep=0pt] (x3) at (6.5,-2) {$x_3$};
		\node[circle, inner sep=0pt] (x4) at (7.5,-2) {$x_4$};
		\node[circle, inner sep=0pt] at (8.5,-2) {$\dots$};
		\node[circle, inner sep=0pt] (y1) at (1.5,-2) {$y_1$};
		\node[circle, inner sep=0pt] (y2) at (0.5,-2) {$y_2$};
		\node[circle, inner sep=0pt] (y3) at (-0.5,-2) {$y_3$};
		\node[circle, inner sep=0pt] (y4) at (-1.5,-2) {$y_4$};
		\node[circle, inner sep=0pt] at (-2.5,-2) {$\dots$};
		\draw[blue, postaction=decorate, out=150, in=30] (v2) to (v1);
		\draw[blue, postaction=decorate, out=150, in=30] (v3) to (v2);
		\draw[blue, postaction=decorate, out=150, in=30] (v4) to (v3);
		\draw[red, dashed, postaction=decorate, out=210, in=330] (v2) to (v1);
		\draw[red, dashed, postaction=decorate, out=210, in=330] (v3) to (v2);
		\draw[red, dashed, postaction=decorate, out=210, in=330] (v4) to (v3);
		\draw[blue, postaction=decorate, out=240, in=120] (w1) to (v1);
		\draw[blue, postaction=decorate, out=240, in=120] (w2) to (v2);
		\draw[blue, postaction=decorate, out=240, in=120] (w3) to (v3);
		\draw[blue, postaction=decorate, out=240, in=120] (w4) to (v4);
		\draw[red, dashed, postaction=decorate, out=300, in=60] (w1) to (v1);
		\draw[red, dashed, postaction=decorate, out=300, in=60] (w2) to (v2);
		\draw[red, dashed, postaction=decorate, out=300, in=60] (w3) to (v3);
		\draw[red, dashed, postaction=decorate, out=300, in=60] (w4) to (v4);
		\draw[blue, postaction=decorate] (w1).. controls +(-0.75, 0.75) and +(-0.75, -0.75) .. (w1);
		\draw[blue, postaction=decorate] (w2).. controls +(-0.75, 0.75) and +(-0.75, -0.75) .. (w2);
		\draw[blue, postaction=decorate] (w3).. controls +(-0.75, 0.75) and +(-0.75, -0.75) .. (w3);
		\draw[blue, postaction=decorate] (w4).. controls +(-0.75, 0.75) and +(-0.75, -0.75) .. (w4);
		\draw[red, dashed, postaction=decorate] (w1).. controls +(0.75, 0.75) and +(0.75, -0.75) .. (w1);
		\draw[red, dashed, postaction=decorate] (w2).. controls +(0.75, 0.75) and +(0.75, -0.75) .. (w2);
		\draw[red, dashed, postaction=decorate] (w3).. controls +(0.75, 0.75) and +(0.75, -0.75) .. (w3);
		\draw[red, dashed, postaction=decorate] (w4).. controls +(0.75, 0.75) and +(0.75, -0.75) .. (w4);
		\draw[blue, postaction=decorate] (x1) to (v1);
		\draw[blue, postaction=decorate] (x1) to (v2);
		\draw[blue, postaction=decorate] (x1) to (v3);
		\draw[blue, postaction=decorate] (x1) to (v4);
		\draw[blue, postaction=decorate] (x2) to (x1);
		\draw[blue, postaction=decorate] (x3) to (x2);
		\draw[blue, postaction=decorate] (x4) to (x3);
		\draw[red, dashed, postaction=decorate] (x1).. controls +(0.75, -0.75) and +(-0.75, -0.75) .. (x1);
		\draw[red, dashed, postaction=decorate] (x2).. controls +(0.75, -0.75) and +(-0.75, -0.75) .. (x2);
		\draw[red, dashed, postaction=decorate] (x3).. controls +(0.75, -0.75) and +(-0.75, -0.75) .. (x3);
		\draw[red, dashed, postaction=decorate] (x4).. controls +(0.75, -0.75) and +(-0.75, -0.75) .. (x4);
		\draw[red, dashed, postaction=decorate] (y1) to (v1);
		\draw[red, dashed, postaction=decorate] (y1) to (v2);
		\draw[red, dashed, postaction=decorate] (y1) to (v3);
		\draw[red, dashed, postaction=decorate] (y1) to (v4);
		\draw[red, dashed, postaction=decorate] (y2) to (y1);
		\draw[red, dashed, postaction=decorate] (y3) to (y2);
		\draw[red, dashed, postaction=decorate] (y4) to (y3);
		\draw[blue, postaction=decorate] (y1).. controls +(0.75, -0.75) and +(-0.75, -0.75) .. (y1);
		\draw[blue, postaction=decorate] (y2).. controls +(0.75, -0.75) and +(-0.75, -0.75) .. (y2);
		\draw[blue, postaction=decorate] (y3).. controls +(0.75, -0.75) and +(-0.75, -0.75) .. (y3);
		\draw[blue, postaction=decorate] (y4).. controls +(0.75, -0.75) and +(-0.75, -0.75) .. (y4);
	\end{tikzpicture}
\]
This example appeared in \cite{anHuefNgSims2021} as an illustration of a $2$-graph whose
$C^*$-algebra is stably finite. We have chosen it as an illustration because it has a
reasonably complex essential-isotropy structure and also a reasonably complex ideal
structure that cannot be described using existing results (for example, it is not a
pullback of a $1$-graph, nor a product of $1$-graphs).

We describe $\Dd_\Lambda$. It is convenient here to regard an element of $\Dd_\Lambda$
here as a function $D \colon \Lambda^0 \to \mathsf{Open}(\T^2)$ from the vertices of $\Lambda$
to the open subsets of $\T^2$; a subset $D \subseteq \Lambda^0 \times \T^2$ is
identified with the function $v \mapsto D(v) \coloneqq \{z \in \T^2 : (v,z) \in D\}$.

For $D \in \Dd_\Lambda$, condition~\labelcref{(D1)} implies that:
\begin{itemize}
\item $D(v_i) \subseteq D(v_{i+1})$ for all $i$;
\item $D(v_i) \subseteq D(w_i)$ for all $i$;
\item $D(x_i) \subseteq D(x_{i+1})$ for all $i$;
\item $D(y_i) \subseteq D(y_{i+1})$ for all $i$; and
\item $\bigcup_j D(v_j) \subseteq D(x_1) \cap D(y_1)$.
\end{itemize}

The consequences of condition~\labelcref{(D2)} are as follows.
\begin{itemize}
  \item Applied with $v = x_i$ and $n = (1,0)$, condition~\labelcref{(D2)} forces $D(x_{i+1})
    \subseteq D(x_i)$ for all $i$; combining this with the consequenes of~\labelcref{(D1)} above
    implies that $D$ is constant on $\{x_i : i \in \N\}$.
  \item Applied with $v = y_i$ and $n = (0,1)$, condition~\labelcref{(D2)} forces $D(y_{i+1})
    \subseteq D(y_i)$ for all $i$; combining this with the consequenes of~\labelcref{(D1)} above
    implies that $D$ is constant on $\{y_i : i \in \N\}$.
  \item Applied with $v = v_i$ and $n = (j, 0)$, condition~\labelcref{(D2)} implies that $D(v_{i+j})
    \cap \big(\bigcap_{0 \le l < j} D(w_{i+l})) \cap D(x_1) \subseteq D(v_i)$ for all
    $i$; since each $D(w_{i+l})$ contains $D(v_{i+l})$, and since the $D(v_i)$ are
    increasing and contained in $D(x_1)$, this reduces to $D(v_{i+1}) \cap D(w_i)
    \subseteq D(v_i)$.
\end{itemize}

To understand the consequences of~\labelcref{(D3)}, we must first understand the
harmonious families of bisections in this $2$-graph groupoid. To do this, first note that
each $x_i \Lambda^\infty$, each $y_i \Lambda^\infty$ and each $w_i \Lambda^\infty$ is a
singleton; we will write $\zeta(u)$ for the unique element of $u\Lambda^\infty$ for each
$u \in \{x_i, y_i, z_i : i \in \N\}$. Since each $u\Lambda^\infty$ is clopen in
$\Lambda^\infty$, each $\zeta(u)$ is an isolated point. For a given vertex $u$ and
infinite path $\eta \in u\Lambda^\infty$, if $\Bb$ and $\Bb'$ are harmonious bisection
families based at $\eta$ such that $B_\alpha \subseteq B'_\alpha$ for all $\alpha$, then
the collection of functions $D$ that satisfy~\labelcref{(D3)} with respect to $\Bb$ is
larger than the collection that satisfy~\labelcref{(D3)} with respect to $\Bb'$.
Combining this with \cref{lem:bisection family transport} applied with $C = Z(u)$ we see
that it suffices to consider harmonious families of bisections $\Bb^u$ based at each
$\zeta(u)$ such that $B^u_{\zeta(u)} \subseteq u\Lambda^\infty = \{\zeta(u)\}$. For each
$u$, there is only one such harmonious family of bisections, namely $\mathcal{B}^u =
\{B^u_\alpha : \alpha \in \J_{\zeta(u)}\}$ given by $B^u_\alpha \coloneqq \{\alpha\}$ for
each $\alpha \in \J_{\zeta(u)} = \{(\zeta(u), p-q, \zeta(u)) : T^p(\zeta(u)) =
T^q(\zeta(u))\}$. In particular,
\begin{align*}
 \Bb^{x_i} &= \{(\zeta(x_i), (0,l), \zeta(x_i)) : l \in \Z\},  \\
 \Bb^{y_i} &= \{(\zeta(x_i), (l,0), \zeta(x_i)) : l \in \Z\}, \\
 \Bb^{w_i} &= \{(\zeta(x_i), (l_1,l_2), \zeta(x_i)) : (l_1, l_2) \in \Z^2\}.
\end{align*}
So~\labelcref{(D3)} implies that
\begin{itemize}
\item each $D(x_i)$ is invariant under multiplication by $\T \times \{1\} \subseteq
    \T^2$, and
\item each $D(y_i)$ is invariant under multiplication by $\{1\} \times \T \subseteq
    \T^2$.
\end{itemize}
Since each $((\Bb^{w_i})^{\ess})^\perp_{\zeta(w_i)} = \{1\}$, condition~\labelcref{(D3)} imposes no
condition on the $D(w_i)$.

It remains to analyse the $D(v_i)$. For this, let $S \coloneqq \{v_i : i \in \N\}$. Then
$\Lambda S$ is a sub-$k$-graph of $\Lambda$. There are infinitely many infinite paths in
each $v_i\Lambda^\infty$, but all but one of these has the form $\mu \zeta(u)$ for some
finite path $u$ and some $u \in \{x_i, y_i, z_i : i \in \N\}$. The one remaining infinite
path is the unique infinite path $\zeta(v_i)$ in $v_i\Lambda S$. Each
$\overline{[\zeta(v_i)]} = \{\zeta(v_j) : j \in \N\} = (\Lambda S)^\infty$. We have
$T^n(\zeta(v_i)) = T^m(\zeta(v_i))$ if and only if $m - n \in \Z(1, -1) \subseteq \Z^2$.
So~\labelcref{(D3)} implies that there exists $l \in \N$ such that $D(v_j)$ is closed under
multiplication by $\Z(1, -1)^\perp = \{(w, w) : w \in \T\} \subseteq \T^k$. Since we
already established that $D(v_j) \subseteq D(x_j) = D(x_1)$ and since $D(x_1)$ is also
closed under multiplication by $\T \times \{1\}$, we deduce that if any $D(v_j) \not=
\varnothing$ then $D(x_i) = \T^2$ for all $i$. Similar considerations show that if any
$D(v_j) \not= \varnothing$ then each $D(y_i) = \T^2$.

We are now in a position to describe all of $\Dd_\Lambda$. A function $D \colon \Lambda^0 \to
\mathsf{Open}(\T^2)$ belongs to $\Dd_\Lambda$ if and only if either:
\begin{itemize}
\item $D(v_i) = \varnothing$ for all $i$; $D$ is constant on $\{x_i : i \in \N\}$ and
    $D(x_1)$ is invariant under $\T \times \{1\}$; and $D$ is constant on $\{y_i : i
    \in \N\}$ and $D(y_1)$ is invariant under $\{1\} \times \T$; or
\item $D(x_i) = D(y_i) = \T^2$ for all $i$; each $D(v_i)$ is invariant under
    $\{(w,\overline{w}) : w \in \T\}$ and each $D(v_i) \subseteq D(v_{i+1})$; each
    $D(v_i) \subseteq D(w_i)$; and each $D(w_i) \cap D(v_{i+1}) \subseteq D(v_i)$.
\end{itemize}
\end{example}

\begin{example}
To illustrate how aperiodicity affects the ideal structure, suppose that we modify
\cref{eg:aHNS example} as follows. Fix disjoint subsets $S_0, S_1 \subseteq \N$. For each
$n \in S \coloneqq S_0 \cup S_1$, add an additional red loop and an additional blue loop at
$w_n$ (in the picture below, $w_3 \in S$ while $w_1, w_2, w_4 \not\in S$).

For $n \in S$ there are now multiple red-blue loops at $w_n$ and multiple red-blue paths
from $w_n$ to $v_n$, so we must specify factorisation rules (see \cite[Theorems
4.4~and~4.5]{HazRaeSimWeb}). For each $n \in S$, denote the blue loops at $w_n$ by
$e^n_1, e^n_2$, the red loops by $f^n_1, f^n_2$, the blue edge from $w_n$ to $v_n$ by
$a_n$ and the red edge from $w_n$ to $v_n$ by $b_n$. For all $n \in S$, impose the
factorisation rules $a_n f^n_i = b_n e^n_i$ on red-blue paths from $w_n$ to $v_n$. For
blue-red loops at $w_n$, impose factorisation rules depending on whether $n \in S_0$ or
$n \in S_1$:
\[
e^n_i f^n_j =
    \begin{cases}
        f^n_j e^n_i &\text{ if $n \in S_0$.}\\
        f^n_i e^n_j &\text{ if $n \in S_1$}
    \end{cases}
\]
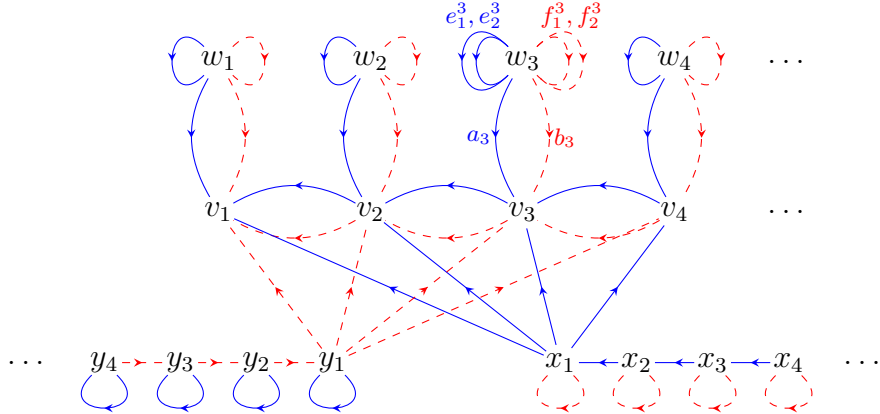
\[
	\begin{tikzpicture}[>=stealth, decoration={markings, mark=at position 0.5 with {\arrow{>}}}]
		\node[circle, inner sep=0pt] (v1) at (0,0) {$v_1$};
		\node[circle, inner sep=0pt] (v2) at (2,0) {$v_2$};
		\node[circle, inner sep=0pt] (v3) at (4,0) {$v_3$};
		\node[circle, inner sep=0pt] (v4) at (6,0) {$v_4$};
		\node[circle, inner sep=0pt] at (7.5,0) {$\dots$};
		\node[circle, inner sep=0pt] (w1) at (0,2) {$w_1$};
		\node[circle, inner sep=0pt] (w2) at (2,2) {$w_2$};
		\node[circle, inner sep=0pt] (w3) at (4,2) {$w_3$};
		\node[circle, inner sep=0pt] (w4) at (6,2) {$w_4$};
		\node[circle, inner sep=0pt] at (7.5,2) {$\dots$};
		\node[circle, inner sep=0pt] (x1) at (4.5,-2) {$x_1$};
		\node[circle, inner sep=0pt] (x2) at (5.5,-2) {$x_2$};
		\node[circle, inner sep=0pt] (x3) at (6.5,-2) {$x_3$};
		\node[circle, inner sep=0pt] (x4) at (7.5,-2) {$x_4$};
		\node[circle, inner sep=0pt] at (8.5,-2) {$\dots$};
		\node[circle, inner sep=0pt] (y1) at (1.5,-2) {$y_1$};
		\node[circle, inner sep=0pt] (y2) at (0.5,-2) {$y_2$};
		\node[circle, inner sep=0pt] (y3) at (-0.5,-2) {$y_3$};
		\node[circle, inner sep=0pt] (y4) at (-1.5,-2) {$y_4$};
		\node[circle, inner sep=0pt] at (-2.5,-2) {$\dots$};
		\draw[blue, postaction=decorate, out=150, in=30] (v2) to (v1);
		\draw[blue, postaction=decorate, out=150, in=30] (v3) to (v2);
		\draw[blue, postaction=decorate, out=150, in=30] (v4) to (v3);
		\draw[red, dashed, postaction=decorate, out=210, in=330] (v2) to (v1);
		\draw[red, dashed, postaction=decorate, out=210, in=330] (v3) to (v2);
		\draw[red, dashed, postaction=decorate, out=210, in=330] (v4) to (v3);
		\draw[blue, postaction=decorate, out=240, in=120] (w1) to (v1);
		\draw[blue, postaction=decorate, out=240, in=120] (w2) to (v2);
		\draw[blue, postaction=decorate, out=240, in=120] (w3) to node[pos=0.5, left, inner sep=1.5pt] {\scriptsize$a_3$} (v3);
		\draw[blue, postaction=decorate, out=240, in=120] (w4) to (v4);
		\draw[red, dashed, postaction=decorate, out=300, in=60] (w1) to (v1);
		\draw[red, dashed, postaction=decorate, out=300, in=60] (w2) to (v2);
		\draw[red, dashed, postaction=decorate, out=300, in=60] (w3) to node[pos=0.5, right, inner sep=1.5pt] {\scriptsize$b_3$} (v3);
		\draw[red, dashed, postaction=decorate, out=300, in=60] (w4) to (v4);
		\draw[blue, postaction=decorate] (w1).. controls +(-0.75, 0.75) and +(-0.75, -0.75) .. (w1);
		\draw[blue, postaction=decorate] (w2).. controls +(-0.75, 0.75) and +(-0.75, -0.75) .. (w2);
		\draw[blue, postaction=decorate] (w3).. controls +(-0.75, 0.75) and +(-0.75, -0.75) .. (w3);
		\draw[blue, postaction=decorate] (w3).. controls +(-1, 1) and +(-1, -1) .. (w3) node[pos=0.25, above, inner sep=2pt] {\scriptsize$e^3_1, e^3_2$};
		\draw[blue, postaction=decorate] (w4).. controls +(-0.75, 0.75) and +(-0.75, -0.75) .. (w4);
		\draw[red, dashed, postaction=decorate] (w1).. controls +(0.75, 0.75) and +(0.75, -0.75) .. (w1);
		\draw[red, dashed, postaction=decorate] (w2).. controls +(0.75, 0.75) and +(0.75, -0.75) .. (w2);
		\draw[red, dashed, postaction=decorate] (w3).. controls +(0.75, 0.75) and +(0.75, -0.75) .. (w3);
		\draw[red, dashed, postaction=decorate] (w3).. controls +(1, 1) and +(1, -1) .. (w3) node[pos=0.25, above, inner sep=2pt] {\scriptsize$f^3_1, f^3_2$};
		\draw[red, dashed, postaction=decorate] (w4).. controls +(0.75, 0.75) and +(0.75, -0.75) .. (w4);
		\draw[blue, postaction=decorate] (x1) to (v1);
		\draw[blue, postaction=decorate] (x1) to (v2);
		\draw[blue, postaction=decorate] (x1) to (v3);
		\draw[blue, postaction=decorate] (x1) to (v4);
		\draw[blue, postaction=decorate] (x2) to (x1);
		\draw[blue, postaction=decorate] (x3) to (x2);
		\draw[blue, postaction=decorate] (x4) to (x3);
		\draw[red, dashed, postaction=decorate] (x1).. controls +(0.75, -0.75) and +(-0.75, -0.75) .. (x1);
		\draw[red, dashed, postaction=decorate] (x2).. controls +(0.75, -0.75) and +(-0.75, -0.75) .. (x2);
		\draw[red, dashed, postaction=decorate] (x3).. controls +(0.75, -0.75) and +(-0.75, -0.75) .. (x3);
		\draw[red, dashed, postaction=decorate] (x4).. controls +(0.75, -0.75) and +(-0.75, -0.75) .. (x4);
		\draw[red, dashed, postaction=decorate] (y1) to (v1);
		\draw[red, dashed, postaction=decorate] (y1) to (v2);
		\draw[red, dashed, postaction=decorate] (y1) to (v3);
		\draw[red, dashed, postaction=decorate] (y1) to (v4);
		\draw[red, dashed, postaction=decorate] (y2) to (y1);
		\draw[red, dashed, postaction=decorate] (y3) to (y2);
		\draw[red, dashed, postaction=decorate] (y4) to (y3);
		\draw[blue, postaction=decorate] (y1).. controls +(0.75, -0.75) and +(-0.75, -0.75) .. (y1);
		\draw[blue, postaction=decorate] (y2).. controls +(0.75, -0.75) and +(-0.75, -0.75) .. (y2);
		\draw[blue, postaction=decorate] (y3).. controls +(0.75, -0.75) and +(-0.75, -0.75) .. (y3);
		\draw[blue, postaction=decorate] (y4).. controls +(0.75, -0.75) and +(-0.75, -0.75) .. (y4);
	\end{tikzpicture}
\]

Let $\Gamma$ be the $2$-graph with this skeleton and these factorisation
rules. For $n \in S_0$, the subgraph $w_n \Gamma w_n$ is a cartesian product
of two copies of the $1$-graph $B_2$ that has one vertex and two loops. So
the reduction of $G_\Gamma$ to $Z(w_n)$ for $n \in S_0$ is the cartesian
product of two copies of the standard groupoid $\Hh_2$ for the Cuntz algebra
$\mathcal{O}_2$, which has trivial essential isotropy. Thus, for $n \in S_0$,
and $x \in Z(w_n)$, we have $\J_x = \{0\}$.

For $n \in S_1$, the subgraph $w_n \Gamma w_n$ is the pullback of the same $1$-graph
$B_2$ by the functor $m \mapsto m_1 + m_2$ from $\N^2$ to $\N$. So by
\cite[Proposition~2.10]{Kumjian-Pask2000}, the reduction of $G_\Gamma$ to $Z(w_n)$ for $n
\in S_1$ is isomorphic to the cartesian product of $\Hh_2$ with the group $\Z$. For $n
\in S_1$, and $x \in Z(w_n)$, we have $\J_x = \{(x, (m, -m), x) : m \in \Z\}$.

The set $\mathcal{D}_\Gamma$ differs from $\mathcal{D}_\Lambda$ in \cref{eg:aHNS example} only
in that if $D \in \mathcal{D}_\Gamma$, then for $n \in S_0$, we have $D(w_n) \in
\{\varnothing, \T^2\}$, and for $n \in S_1$, the set $D(w_n)$ is invariant for
multiplication by $\Z(1,-1)^\perp = \{(w, w) : w \in \T\} \subseteq \T^2$. Aside from
that, the constraints on $D$ are as in \cref{eg:aHNS example}. Writing $S^\circ_0 = \{n
\in S_0 : D(w_n) = \varnothing\}$ and $S^\bullet_0 = \{n \in S_0 : D(w_n) = \T^2\}$, if
$\bigcup D(v_i) \not= \varnothing$ so that $D(x_1) = D(y_1) = \T^2$, we have $D(v_n) =
D(v_{n+1})$ for $n \in S_0^\bullet$.
\end{example}

\end{document}